\documentclass[11pt, reqno]{amsart}
\usepackage{amscd}        

\usepackage{float}

\usepackage{nomencl}

\usepackage{algorithm}
\usepackage{algpseudocode}
\usepackage{pifont}
\usepackage{mathrsfs}

\usepackage{tikz-cd}
\usepackage{graphicx}
\usepackage{rotating}
\usepackage{diagbox}
\usepackage{amssymb}
\usepackage{epstopdf}
\usepackage{tikz}
\definecolor{mintgreen}{RGB}{152,255,152}
\definecolor{pinksalmon}{RGB}{255,102,102}
\definecolor{hueso}{RGB}{245,245,220}
\definecolor{marfil}{RGB}{255,253,208}
\definecolor{amarillo}{RGB}{255,255,0}
\usetikzlibrary{decorations.markings,arrows}
\usetikzlibrary{decorations.pathreplacing}
\usepackage[inner=1.0in,outer=1.0in,bottom=1.0in, top=1.0in]{geometry}

\numberwithin{equation}{section}

\newtheorem{theorem}{Theorem}[section]
\newtheorem{lemma}[theorem]{Lemma}
\newtheorem{proposition}[theorem]{Proposition}
\newtheorem{corollary}[theorem]{Corollary}
\newtheorem{conjecture}[theorem]{Conjecture}

\newtheorem{thmx}{Theorem}

\makeatletter
\def\moverlay{\mathpalette\mov@rlay}
\def\mov@rlay#1#2{\leavevmode\vtop{%
   \baselineskip\z@skip \lineskiplimit-\maxdimen
   \ialign{\hfil$\m@th#1##$\hfil\cr#2\crcr}}}
\newcommand{\charfusion}[3][\mathord]{
    #1{\ifx#1\mathop\vphantom{#2}\fi
        \mathpalette\mov@rlay{#2\cr#3}
      }
    \ifx#1\mathop\expandafter\displaylimits\fi}
\makeatother

\newcommand{\bigcupdot}{\charfusion[\mathop]{\bigcup}{\cdot}}

\newcommand{\suchthat}{\;\ifnum\currentgrouptype=16 \middle\fi|\;}

\newcommand{\Z}{\mathbb{Z}}
\newcommand{\C}{\mathbb{C}}
\newcommand{\Q}{\mathbb{Q}}
\newcommand{\R}{\mathbb{R}}
\newcommand{\Gal}[1]{\operatorname{Gal}#1}
\newcommand{\op}[1]{\operatorname{#1}}

\newcommand{\CM}{\mathcal{CM}}

\theoremstyle{definition}
\newtheorem{remark}[theorem]{Remark}
\newtheorem{definition}[theorem]{Definition}
\newtheorem{example}[theorem]{Example}

\def\C{{\mathbb C}}
\def\R{{\mathbb R}}
\def\Z{{\mathbb Z}}
\def\Q{{\mathbb Q}}

\def\O_K{{\Cal{O}_{K}}}
\def\O_F{{\Cal{O}_{F}}}
\def\N_F{{\Cal{N}_{F/\Q}}}

\begin{document}

\title[On the Colmez conjecture for non-abelian CM fields]{On the Colmez conjecture for non-abelian CM fields}

\author{Adrian Barquero-Sanchez and Riad Masri}

\address{Department of Mathematics, Mailstop 3368, Texas A\&M University, College Station, TX 77843-3368 }

\email{adrianbs11@math.tamu.edu}
\email{masri@math.tamu.edu}


\begin{abstract}
The Colmez conjecture relates the Faltings height of an abelian variety with complex
multiplication by the ring of integers of a CM field $E$
to logarithmic derivatives of certain Artin $L$--functions at $s=0$. In this
paper, we prove that if $F$ is any fixed totally real number field of degree $[F:\Q] \geq 3$,
then there are infinitely many CM extensions $E/F$ such that
$E/\Q$ is \textit{non-abelian} and the Colmez conjecture is true for $E$. Moreover, these CM extensions are 
explicitly constructed to be ramified at ``arbitrary'' prescribed sets of prime ideals of $F$.

We also prove that the Colmez conjecture is true for a generic
class of non-abelian CM fields called Weyl CM fields, and use this 
to develop an arithmetic statistics approach to the Colmez conjecture 
based on counting CM fields of fixed degree and bounded discriminant. 

We illustrate these results 
by evaluating the Faltings height of the Jacobian 
of a genus 2 hyperelliptic curve with complex multiplication by a non-abelian quartic CM field in 
terms of the Barnes double Gamma function at algebraic arguments. 
This can be seen as an explicit non-abelian Chowla-Selberg formula.

A crucial input to the proofs is an averaged version of the Colmez conjecture which was
recently proved independently by Andreatta-Goren-Howard-Madapusi Pera and Yuan-Zhang.

\end{abstract}

\maketitle

\section{Introduction}

\subsection{The Chowla-Selberg formula and the Colmez conjecture}
One of the central objects of study in number theory is the \textit{Dedekind eta function}, which is the
weight $1/2$ modular form for $SL_2(\Z)$ defined by the infinite product
\begin{align*}
\eta(z):=q^{1/24}\prod_{n=1}^{\infty}(1-q^n), \quad q:=e^{2\pi i z}.
\end{align*}
A remarkable formula of Chowla and Selberg \cite{CS67}
relates values of $\eta(z)$ at CM points to values of the Euler Gamma function $\Gamma(s)$ at rational numbers. Here
we briefly recall this formula (see e.g. \cite{Wei76}).
Let $E$ be an imaginary quadratic field of discriminant $-D < 0$. Let $h(-D)$ be the class number,
$w(-D)$ be the number of units,
and $\chi_{-D}$ be the Kronecker symbol. Using Kronecker's first limit formula, one can prove the identity
\begin{align}\label{preCS}
\sum_{C} \log \left(\sqrt{\textrm{Im}(\tau_C)}|\eta(\tau_C)|^2\right)
= \frac{h(-D)}{2}\left(\log\left(\frac{\sqrt{D}}{2}\right) - \log(2\pi)
+ \frac{L^{\prime}(\chi_{-D},0)}{L(\chi_{-D},0)} \right),
\end{align}
where the sum is over a complete set of CM points $\tau_C$ of discriminant $-D$ on $SL_2(\mathbb Z) \backslash \mathbb H$.
There are $h(-D)$ such points, corresponding to the ideal classes $C$ of $E$. On the other hand, a
classical identity of Lerch \cite{Ler87} evaluates the logarithmic derivative of the Dirichlet $L$--function
$L(\chi_{-D},s)$ at $s=0$ in terms of values of $\Gamma(s)$ at rational
numbers,
\begin{align}\label{lerch}
\frac{L^{\prime}(\chi_{-D},0)}{L(\chi_{-D},0)} = -\log(D) +
\frac{w(-D)}{2h(-D)}\sum_{k=1}^{D}\chi_{-D}(k)\log\left(\Gamma\left(\frac{k}{D}\right)\right).
\end{align}
Substituting Lerch's identity (\ref{lerch}) into (\ref{preCS}) then yields the \textit{Chowla-Selberg formula}
\begin{align}\label{cs}
\prod_{C}\sqrt{\textrm{Im}(\tau_{C})}|\eta(\tau_{C})|^2
=\left(\frac{1}{4\pi\sqrt{D}}\right)^{\frac{h(-D)}{2}}\prod_{k=1}^{D}
\Gamma\left(\frac{k}{D}\right)^{\frac{w(-D)\chi_{-D}(k)}{4}}.
\end{align}

There is a beautiful geometric reformulation of the Chowla-Selberg formula (\ref{cs}) as an identity which relates
the Faltings height of a CM elliptic curve to the logarithmic derivative of $L(\chi_{-D},s)$ at $s=0$. In order 
to describe this, we first recall the definition of the (stable) Faltings height of a CM abelian variety.
Let $F$ be a totally real number field of degree $n$. Let $E/F$ be a CM extension of $F$ and $\Phi$
be a CM type for $E$. Let $X_{\Phi}$ be an abelian variety
defined over $\overline{\mathbb Q}$ with complex multiplication by $\mathcal{O}_E$ and
CM type $\Phi$. We call $X_{\Phi}$ a CM abelian variety of type $(\mathcal{O}_E,\Phi)$.
Let $K \subseteq \overline{\mathbb Q}$ be a number field over which $X_{\Phi}$ has everywhere good reduction,
and choose a N\'eron differential $\omega \in H^{0}(X_{\Phi}, \Omega^n_{X_{\Phi}})$. Then the \textit{Faltings height} of $X_{\Phi}$ is defined by
\begin{align*}
h_{\textrm{Fal}}(X_{\Phi}):= -\frac{1}{2[K:\mathbb Q]}\sum_{\sigma : K \hookrightarrow \mathbb{C}}
\log\left| \int_{X_{\Phi}^{\sigma}(\mathbb{C})} \omega^{\sigma} \wedge \overline{\omega^{\sigma}}\right|.
\end{align*}
The Faltings height does not depend on the choice of $K$ or $\omega$.

Now, if $E=\Q(\sqrt{-D})$ is an imaginary quadratic field and
$X_{\Phi}$ is a CM elliptic curve of type $(\mathcal{O}_E,\Phi)$, then one can prove that (see e.g. \cite{Gro80, Sil86})
\begin{align*}
h_{\textrm{Fal}}(X_{\Phi})=-\log(2^{3/2}\pi) -\frac{1}{h(-D)}\sum_{C}\log\left(\sqrt{\textrm{Im}(\tau_C)}|\eta(\tau_C)|^2\right).
\end{align*}
Combining this identity with (\ref{preCS}) allows one to express the Chowla-Selberg formula in the equivalent form
\begin{align}\label{grossid}
h_{\textrm{Fal}}(X_{\Phi}) = -\frac{1}{2}\frac{L^{\prime}(\chi_{-D},0)}{L(\chi_{-D},0)}
-\frac{1}{4}\log\left(D\right) - \frac{1}{2}\log(2\pi).
\end{align}

Colmez \cite{Col93} gave a vast conjectural generalization of the identity 
(\ref{grossid}) which relates the Faltings height of \textit{any} CM abelian variety
$X_{\Phi}$ of type $(\mathcal{O}_E, \Phi)$ to logarithmic derivatives
at $s=0$ of certain Artin $L$--functions constructed from the CM pair $(E,\Phi)$. See Section \ref{colmezstatement} for
the precise statement of the Colmez conjecture.

\subsection{Previous work on the Colmez conjecture} There have been many remarkable works on the Colmez conjecture. 

Colmez \cite{Col93} proved his conjecture when $E/\Q$ is abelian, up to addition of
a rational multiple of $\log(2)$ which was recently shown to equal zero by Obus \cite{Obu13}.

Yang \cite{Yan10a, Yan10b, Ya13} proved the Colmez conjecture for a
large class of non-biquadratic CM fields of degree $[E:\Q]=4$, thus establishing
the only known cases of the Colmez conjecture when $E/\Q$ is \textit{non-abelian}.

In his paper, Colmez \cite{Col93} also stated an \textit{averaged} version of his conjecture,
where the Faltings heights are averaged over the different CM types for the given CM field $E$. 
See Section \ref{AverageColmezConjecture} for the statement of
the average Colmez conjecture. Very recently, Andreatta-Goren-Howard-Madapusi Pera \cite{AGHM15} and Yuan-Zhang \cite{YZ15}
independently proved the average Colmez conjecture. Interest in the average Colmez conjecture
is motivated in part by work of Tsimerman \cite{Tsi15}, who used it
to prove the Andr\'e-Oort conjecture for the moduli space $\mathcal{A}_{g}$
of principally polarized abelian varieties of dimension $g$. The average Colmez conjecture will also play a crucial
role in the proofs of the results in this paper (see e.g. Section \ref{Outline}). 

\subsection{Statement of the main results} 

As discussed, the only known cases of the Colmez conjecture for non-abelian CM fields are due to Yang for a large class of 
CM fields of degree 4. In our first main result, we will prove that if $F$ is any fixed totally real number field of degree $n \geq 3$,
then there are infinitely many CM extensions $E/F$ such that
$E/\Q$ is \textit{non-abelian} and the Colmez conjecture is true for $E$. 

More precisely, let $p$ be a prime number which 
splits in the Galois closure $F^s$ and let $\frak{p}$ be a prime ideal of $F$ lying above $p$. We will prove that
if we fix an ``arbitrary'' finite set  $\mathcal{R}$ of prime ideals of $F$,
then we can explicitly construct infinitely many CM extensions $E/F$ which are ramified only at the primes in the prescribed 
set $\mathcal{R} \cup \{\frak{p}\}$ 
and at exactly one more prime ideal of $F$ (which is different for each of the extensions $E/F$)
such that $E/\Q$ is non-Galois and the Colmez conjecture is true for $E$. Similarly, we can prescribe finite sets $\mathcal{U}_1$
(resp. $\mathcal{U}_2$) of prime ideals of $F$ that will be split (resp. remain inert) in the extensions $E/F$.

\begin{thmx}\label{maintheorem} Let $F$ be a totally real number field of degree $n \geq 3$. Let $p \in \Z$ be a
prime number which splits in the Galois closure $F^s$ and let $\mathfrak{p}$ be a prime ideal of
$F$ lying above $p$. Let $d_{F^s}$ be the discriminant of $F^s$ and $\mathcal{R}$ be a finite set of prime ideals of $F$ not dividing
$pd_{F^s}$.  Let $\mathcal{U}_1$ and $\mathcal{U}_2$ be finite sets of prime ideals of $F$
not dividing $2p d_{F^{s}}$ such that $\mathcal{R}$, $\mathcal{U}_1$ and $\mathcal{U}_2$ are pairwise disjoint.
Then there is a set $\mathcal{S}_{\mathcal{R}, \frak{p}}$ of prime ideals of $F$ which is disjoint from
$\mathcal{R} \cup \mathcal{U}_1 \cup \mathcal{U}_2 \cup \{\frak{p}\}$ such that
the following statements are true.
\begin{itemize}
\item[(i)] $\mathcal{S}_{\mathcal{R}, \frak{p}}$ has positive natural density.\footnote{The \textit{natural density}
of a set $\mathcal{S}$ of prime ideals of a number field $L$ is defined by
\begin{align*}
d(\mathcal{S}):=\lim_{X \rightarrow \infty}\frac{\# \{ \frak{q} \in \mathcal{S} \suchthat N_{L/\Q}(\frak{q}) \leq X \}}
{\# \{\frak{q} \subset \mathcal{O}_L \suchthat \textrm{$\frak{q}$ is a prime ideal with $N_{L/\Q}(\frak{q}) \leq X$}\}},
\end{align*}
provided the limit exists.}
\item[(ii)] For each prime ideal $\frak{q} \in \mathcal{S}_{\mathcal{R}, \frak{p}}$, there is an element
$\Delta_{\frak{q}} \in \mathcal{O}_F$ with prime factorization
\begin{align*}
\Delta_{\frak{q}}\mathcal{O}_F=\frak{p} \frak{q}\prod_{\frak{r} \in \mathcal{R}}\frak{r}.
\end{align*}
\item[(iii)] The field $E_{\frak{q}}:=F(\sqrt{\Delta_{\frak{q}}})$ is a
CM extension of $F$ which is non-Galois over $\Q$ and is ramified only at the prime ideals of
$F$ dividing $\Delta_{\frak{q}}$.
Moreover, each prime ideal in $\mathcal{U}_1$ splits in $E_{\mathfrak{q}}$ and each prime ideal in $\mathcal{U}_2$ remains inert in $E_{\mathfrak{q}}$.
\item[(iv)] The Colmez conjecture is true for $E_{\frak{q}}$.
\end{itemize}
\end{thmx}

\begin{remark}\label{cmremark}
We emphasize that Theorem \ref{maintheorem} is \textit{effective} in the sense that we give an algorithm to construct the set
$\mathcal{S}_{\mathcal{R}, \frak{p}}$ and the associated
CM fields $E_{\frak{q}}$ for $\frak{q} \in \mathcal{S}_{\mathcal{R}, \frak{p}}$. See Section \ref{reflexsection}, and in particular,
Section \ref{algorithm}, Algorithm 1.
\end{remark}

\begin{remark} The set of prime numbers $p \in \Z$ which split in the Galois closure $F^s$ has natural density $1/[F^s:\Q]$.
\end{remark}

In our second main result, we will prove that the Colmez conjecture is true for a generic class of non-abelian
CM fields called Weyl CM fields  (see e.g. \cite{CO12}). As remarked by Oort \cite[p. 5]{Oor12},
``most CM fields are Weyl CM fields''. There are (at least) two different ways in which ``most'' can be understood. In the context of Oort's remark, ``most'' 
refers to density results for isogeny classes of abelian varieties over finite fields. In Section \ref{probability} we 
will give an alternative point of view based on counting CM fields of fixed degree and bounded discriminant, and 
use this to develop a probabilistic approach to the Colmez conjecture.

To define the notion of a Weyl CM field, let $E=\Q(\alpha)$ be a CM field of degree $2g$. Let $m_{\alpha}(X)$
be the minimal polynomial of $\alpha$ and denote its roots by $\alpha_1=\alpha, \overline{\alpha_1}, \ldots, \alpha_g, \overline{\alpha_g}$.
Let $a_{2\ell -1}:=\alpha_{\ell}$ and $a_{2 \ell} := \overline{\alpha_{\ell}}$ for $\ell = 1, \ldots, g$. Then $E^{s}=\Q(a_1, \ldots, a_{2g})$
is the Galois closure of $E$. Let $S_{2g}$ be the symmetric group on the letters $\{a_1, \ldots, a_{2g}\}$ and
$W_{2g}$ be the subgroup of $S_{2g}$ consisting of permutations which map any pair of the form
$\{a_{2j-1}, a_{2j}\}$ to a pair $\{a_{2k-1}, a_{2k}\}$.
The group $W_{2g}$ is called the \textit{Weyl group}.  The Weyl group has order $\# W_{2g} = 2^g g!$ and
fits in the exact sequence
\begin{align*}
1 \longrightarrow (\Z/2\Z)^{g} \longrightarrow W_{2g} \longrightarrow S_{g} \longrightarrow 1.
\end{align*}

Now, it can be shown that the Galois group $\mathrm{Gal}(E^s/\Q)$ is isomorphic to a subgroup of $W_{2g}$. If $E$ is a CM field
such that $\mathrm{Gal}(E^s/\Q) \cong W_{2g}$, then $E$ is called a \textit{Weyl CM field}. Thus, for a CM field to be 
Weyl is analogous to the classical fact that the splitting field of a generic polynomial in $\Q[X]$
of degree $g$ has Galois group isomorphic to $S_g$ (see e.g. \cite{Gal73}).


\begin{thmx}\label{weyl} If $E$ is a Weyl CM field, then the Colmez conjecture is true for $E$.
\end{thmx}

\begin{remark}\label{weylremark} If $E$ is a CM field with $[E:\Q] = 4$, the only possibilities for 
$\Gal(E^s/\Q)$ are $\Z/2\Z \times \Z/2\Z$, $\Z/4\Z$ or $D_4$. Therefore, since $D_4 \cong W_4$, every non-abelian quartic CM field $E$ is Weyl.
It then follows from Theorem \ref{weyl} and the work of Colmez \cite{Col93} and Obus \cite{Obu13} that the Colmez conjecture 
is true for \textit{every} quartic CM field. 
\end{remark}

\begin{remark}\label{Weylremark} We emphasize that if
$g \geq 2$ and $E$ is a Weyl CM field of degree $2g$, then $E/\Q$ is non-Galois since
$\# \mathrm{Gal}(E^s/\Q)=2^g g! > 2g=[E:\Q]$. In particular, any Weyl CM field of degree $2g \geq 4$ is non-abelian.
\end{remark}

\begin{remark} In Section \ref{reflexsection} (see e.g. Remark \ref{notweyl}),
we will prove that the CM fields $E_{\frak{q}}$ which appear
in Theorem \ref{maintheorem} are Weyl CM fields if and only if $[F^s:\Q]=n!$. In particular, if $[F^s:\Q] < n!$, then the fields
$E_{\frak{q}}$ are not Weyl CM fields, so that Theorems \ref{maintheorem} and \ref{weyl} can be viewed as complementary to one another.
\end{remark}

\subsection{Explicit non-abelian Chowla-Selberg formulas} One important feature of the precise form of the Colmez conjecture
for the CM fields appearing in Theorem A or Theorem B is that it allows us to give explicit evaluations of Faltings heights of CM abelian varieties.

Recall that for imaginary quadratic fields,
the Colmez conjecture is a geometric reformulation of the Chowla-Selberg formula which evaluates
the Faltings height of a CM elliptic curve in terms of values of $\Gamma(s)$ at rational numbers. More precisely,
if $E=\Q(\sqrt{-D})$ and $X_{\Phi}$ is a CM elliptic curve of type $(\mathcal{O}_E, \Phi)$, then substituting (\ref{lerch}) into (\ref{grossid}) yields
\begin{align}\label{abelianCS}
h_{\textrm{Fal}}(X_{\Phi})= -\frac{w(-D)}{4h(-D)}\sum_{k=1}^{D}\chi_{-D}(k)\log \left( \Gamma\left(\frac{k}{D}\right)  \right)
+ \frac{1}{4}\log\left(D\right) - \frac{1}{2}\log(2\pi).
\end{align}

Now, if $E$ is a CM field as in Theorem A or Theorem B, and $X_{\Phi}$ is a CM abelian variety of type $(\mathcal{O}_E, \Phi)$, then
the Colmez conjecture takes the form (see Proposition \ref{cc})
\begin{align}\label{niceform}
h_{\textrm{Fal}}(X_{\Phi})=
-\frac{1}{2}\frac{L^{\prime}(\chi_{E/F},0)}{L(\chi_{E/F},0)}
-\frac{1}{4}\log\left(\frac{|d_E|}{d_F}\right) - \frac{n}{2}\log(2\pi),
\end{align}
where $L(\chi_{E/F},s)$ is the (incomplete) $L$--function of the Hecke character $\chi_{E/F}$ associated to the quadratic extension $E/F$ and
$d_E$ (resp. $d_F$) is the discriminant of $E$ (resp. $F$). In fact, we will develop a probabilistic framework 
which predicts that the Colmez conjecture takes this form ``most'' of the time (see Section \ref{probability}). One reason for 
interest in this form of the Colmez conjecture is the appearance of the $L$--function $L(\chi_{E/F},s)$, which allows us
to give explicit ``non-abelian Chowla-Selberg formulas'' analogous to (\ref{abelianCS}) which
evaluate the Faltings heights of CM abelian varieties in terms of values of the Barnes multiple Gamma function at algebraic numbers in $F$.
We will study this problem extensively in the forthcoming papers \cite{BS-M16a, BS-M16b}. Here we give an example of such an evaluation for
the Faltings height of the Jacobian of a genus 2 hyperelliptic curve with complex multiplication by a non-abelian quartic CM field.

\begin{example} Let $E=\Q(\sqrt{-5 - 2\sqrt{2}})$. Then $E$ is a non-abelian quartic CM field of discriminant $d_E=1088$
with real quadratic subfield $F = \Q(\sqrt{2})$ of discriminant $d_F=8$. Moreover, by Remark \ref{weylremark} the CM field $E$ is Weyl, hence
the Colmez conjecture is true for $E$.

Now, by \cite[Theorem 1.1 and Table 2b]{BS15} with the choice $[D, A, B] = [8, 10, 17]$, the Jacobian $J_C$ of the 
genus 2 hyperelliptic curve $C$ over $\Q(\sqrt{17})$ given by the equation
\small
\begin{align}\label{hyper}
y^2 & = x^{6} + (3 + \sqrt{17}) x^5 + \left(\frac{25 + 3 \sqrt{17}}{2}\right) x^4  
+ (3 + 5\sqrt{17})x^3 + \left(\frac{73 - 9\sqrt{17}}{2}\right) x^2 + (-24 + 8\sqrt{17}) x + 10 -  2\sqrt{17}
\end{align}
\normalsize
is a CM abelian surface defined over $\overline{\Q}$ with complex multiplication by the ring of integers 
$\mathcal{O}_E$ of $E$.

\begin{figure}[H]\label{figure}
\begin{tikzpicture}[xscale=0.77,yscale=0.77/2]
\def\r{sqrt(17)}
\draw[pinksalmon, very thick, domain=-3.25 : -1.194497, samples = 100] plot (\x, {sqrt( (\x)^6 + (3 + \r)*(\x)^5 + ((25 + 3*\r)/2)*(\x)^4 + (3 +5*\r)*(\x)^3 + ((73 - 9*\r)/2)*(\x)^2 + (-24 + 8*\r)*\x + 10 - 2*\r)} );
\draw[pinksalmon, very thick, domain=-3.25 : -1.194497, samples = 100] plot (\x, -{sqrt( (\x)^6 + (3 + \r)*(\x)^5 + ((25 + 3*\r)/2)*(\x)^4 + (3 +5*\r)*(\x)^3 + ((73 - 9*\r)/2)*(\x)^2 + (-24 + 8*\r)*\x + 10 - 2*\r)} );
\draw[pinksalmon, very thick, domain=-0.367056 : 0.7, samples = 100] plot (\x, {sqrt( (\x)^6 + (3 + \r)*(\x)^5 + ((25 + 3*\r)/2)*(\x)^4 + (3 +5*\r)*(\x)^3 + ((73 - 9*\r)/2)*(\x)^2 + (-24 + 8*\r)*\x + 10 - 2*\r)} );
\draw[pinksalmon, very thick, domain=-0.367056 : 0.7, samples = 100] plot (\x, -{sqrt( (\x)^6 + (3 + \r)*(\x)^5 + ((25 + 3*\r)/2)*(\x)^4 + (3 +5*\r)*(\x)^3 + ((73 - 9*\r)/2)*(\x)^2 + (-24 + 8*\r)*\x + 10 - 2*\r)} );
\foreach \x in {-3, ..., -1, 1, 2, 3}
    \draw (\x, 0.2) -- (\x, -0.2) node[below] {\x};
\foreach \y in {-5, ..., -1, 1, 2, 3, 4, 5}
    \draw (0.1, \y) -- (-0.1, \y) node[left] {\y};
\draw[thick, ->] (-4,0) -- (4,0) node[right] {$x$};
\draw[thick, ->] (0,-6) -- (0,6) node[right] {$y$};
\end{tikzpicture}
\caption{The hyperelliptic curve $C$.}
\end{figure}
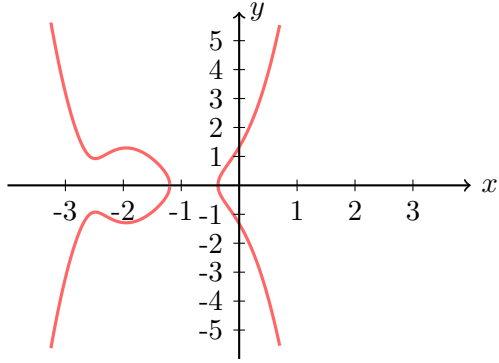

Since the Colmez conjecture is true for $E$, it follows from (\ref{niceform}) that
\begin{align}\label{niceform2}
h_{\textrm{Fal}}(J_C)=
-\frac{1}{2}\frac{L^{\prime}(\chi_{E/F},0)}{L(\chi_{E/F},0)}
-\frac{1}{4}\log(136) - \log(2\pi).
\end{align}
Hence, to complete the evaluation of $h_{\textrm{Fal}}(J_C)$, we need a two-dimensional analog of Lerch's identity
for the logarithmic derivative of $L(\chi_{E/F}, s)$ at $s=0$. For this we require the Barnes double Gamma function (see e.g. \cite{Bar01, Shi77}). 


Let $\omega=(\omega_1, \omega_2) \in \R_{+}^2$ and $z \in \C$. Then the
\textit{Barnes double Gamma function} is defined by 
\begin{align*}
\Gamma_{2}(z,\omega):=F(z,\omega)^{-1},
\end{align*}
where
\small
\begin{align*}
F(z,\omega):= & ~ z\exp\left(\gamma_{22}(\omega)z+\frac{z^2}{2}\gamma_{21}(\omega)\right) 
\prod_{(m,n)}\left(1 + \frac{z}{m\omega_1 + n\omega_2}\right)
\exp\left(-\frac{z}{m\omega_1 + n\omega_2} + \frac{z^2}{2(m\omega_1 + n\omega_2)^2} \right),
\end{align*}
\normalsize
the product being over all pairs of integers $(m,n) \in \Z^{2}_{\geq 0}$ with $(m,n) \neq (0,0)$. The function $F(z,\omega)$ is 
entire, and the constants
$\gamma_{22}(\omega), \gamma_{21}(\omega)$ are explicit ``higher'' analogs of Euler's constant $\gamma$.

Given an element $\alpha \in F$, let $\langle \alpha \rangle = \alpha \mathcal{O}_F$ and 
$\alpha^{\sigma}$ be the image of $\alpha$ under an automorphism $\sigma \in \Gal(F/\Q)$. We also let 
$\alpha^{\prime}$ denote the image of $\alpha$ under the nontrivial automorphism in $\Gal(F/\Q)$. 

Let $\frak{D}_{E/F}$ be the relative discriminant, $h_E$ be the class number of $E$, and  
$\varepsilon > 1$ be the generator of the group $\mathcal{O}_F^{\times, +}$ of 
totally positive units of $F$. Let $B_2(t)=t^2-t+1/6$ be the second Bernoulli polynomial. 

In \cite{BS-M16a}, we use work of Shintani \cite{Shi77} to establish the following two-dimensional analog of Lerch's identity (\ref{lerch}),
\begin{align}\label{lerchrealquad}
\frac{L^{\prime}(\chi_{E/F},0)}{L(\chi_{E/F},0)} & =  -\log(N_{F/\Q}(\frak{D}_{E/F}))\\
& \quad + \frac{[\mathcal{O}_E^{\times}: \mathcal{O}_F^{\times}]}{2h_E}\sum_{z \in \mathcal{R}(\varepsilon, \mathfrak{D}_{E/F}^{-1})} 
\chi_{E/F}\left(\mathfrak{D}_{E/F}\langle z \rangle\right)
\log\left(\prod_{\sigma \in \Gal(F/\Q)} \Gamma_{2} \big(z^{\sigma}, (1, \varepsilon^{\sigma})\big)\right) \notag\\
& \quad  + \frac{\varepsilon - \varepsilon^{\prime}}{2}\log(\varepsilon^{\prime})\frac{[\mathcal{O}_E^{\times}: \mathcal{O}_F^{\times}]}{2h_E}
\sum_{\substack{z \in \mathcal{R}(\varepsilon, \mathfrak{D}_{E/F}^{-1})\\ z = x + y \varepsilon}}\chi_{E/F}\left(\mathfrak{D}_{E/F}\langle z \rangle\right)B_2(x),\notag 
\end{align}
where $\mathcal{R}(\varepsilon, \mathfrak{D}_{E/F}^{-1})$ is a finite subset of $\mathfrak{D}_{E/F}^{-1}$ defined by 
\begin{align*}
\mathcal{R}(\varepsilon, \mathfrak{D}_{E/F}^{-1}) := \left \{ z = x + y\varepsilon \in \mathfrak{D}_{E/F}^{-1}
\suchthat x, y \in \Q, ~ 0 < x \leq 1, ~ 0 \leq y < 1, ~ \mathfrak{D}_{E/F} \langle z \rangle \text{ coprime to } \mathfrak{D}_{E/F} \right \}.
\end{align*}

Here we have $\frak{D}_{E/F}=\langle -5 -2\sqrt{2} \rangle$ and $\varepsilon =3 + 2\sqrt{2}$. We wrote a program in 
\texttt{SageMath} to compute the Shintani set $\mathcal{R}(\varepsilon, \mathfrak{D}_{E/F}^{-1})$. This set can be visualized geometrically in $\R_{+}^{2}$
via the embedding $\alpha \longmapsto (\alpha, \alpha^{\prime})$ as a finite subset of the Shintani cone 
\begin{align*}
C(\varepsilon):=\left\{t_1 (1,1) + t_2 \left(\varepsilon, \varepsilon^{\prime}\right) \suchthat  t_1 > 0,~ t_2 \geq 0 \right\} \subset \mathbb R_+^2
\end{align*}
generated by the vectors $(1, 1)$ and $(\varepsilon, \varepsilon^{\prime})$, as shown in the following figure.\footnote{The shaded parallelogram 
in Figure 2 is the subset of the Shintani cone $\mathcal{C}(\varepsilon)$ determined by the inequalities $0 < t_1 \leq 1$ and $0 \leq t_2 < 1$, which 
correspond to the inequalities appearing in the definition of $\mathcal{R}(\varepsilon, \mathfrak{D}_{E/F}^{-1})$.} 

\small
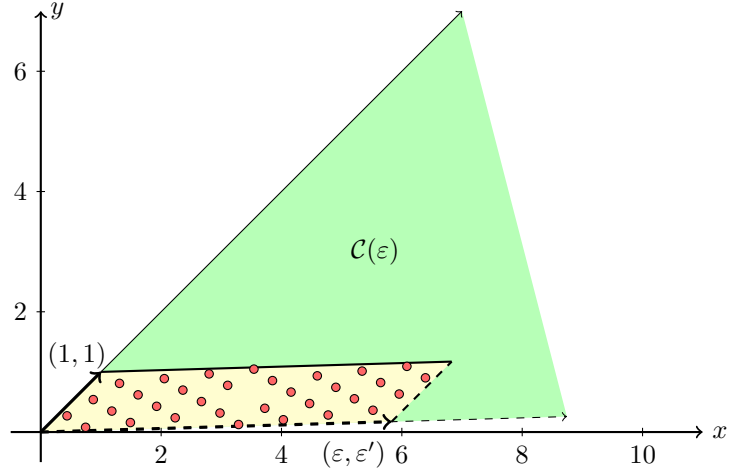
\begin{figure}[h]\label{figure1}
\begin{tikzpicture}[scale = 0.8]
\fill[fill=mintgreen, fill opacity = 0.7] (0,0) -- (7,7) -- (8.74264068711928, 0.257359312880715) -- cycle;
\fill[fill=marfil] (0,0) -- (1,1) -- (6.82842712474619, 1.171572875253810) -- (5.82842712474619, 0.171572875253810) -- cycle;
\draw[thick, ->] (-0.5,0) -- (11,0) node[right] {$x$};
\draw[thick, ->] (0,-0.5) -- (0,7) node[right] {$y$};
\draw (5, 3) node[right] {$\mathcal{C}(\varepsilon)$};
\foreach \x in {2,4,...,10}
\draw (\x,2pt) -- (\x,-2pt) node[below] {\x};
\foreach \y in {2,4,...,6}
\draw (2pt, \y) -- (-2pt, \y) node[left] {\y};
\draw[->] (0, 0) -- (7,7);
\draw[->, dashed] (0, 0) -- (5.82842712474619, 0.171572875253810) -- (8.74264068711928, 0.257359312880715);
\draw[very thick, ->] (0, 0) -- (1,1); 
\draw[very thick, ->, dashed] (0, 0) -- (5.82842712474619, 0.171572875253810); 
\draw (0.6, 0.9) node[above] {$(1, 1)$};
\draw (5.2, 0) node[below] {$(\varepsilon, \varepsilon^{\prime})$};
\draw[thick] (1, 1) -- (6.82842712474619, 1.171572875253810);
\draw[thick, dashed] (5.82842712474619, 0.171572875253810) -- (6.82842712474619, 1.171572875253810);
\filldraw[fill = pinksalmon, fill opacity = 1] (6.39229691519483, 0.901820731863992) circle (2pt);
\filldraw[fill = pinksalmon, fill opacity = 1] (5.95616670564347, 0.632068588474174) circle (2pt);
\filldraw[fill = pinksalmon, fill opacity = 1] (6.08390628654076, 1.09256430169454) circle (2pt);
\filldraw[fill = pinksalmon, fill opacity = 1] (5.33938544833532, 1.01355572813527) circle (2pt);
\filldraw[fill = pinksalmon, fill opacity = 1] (4.90325523878396, 0.743803584745448) circle (2pt);
\filldraw[fill = pinksalmon, fill opacity = 1] (4.15873440057853, 0.664795011186176) circle (2pt);
\filldraw[fill = pinksalmon, fill opacity = 1] (3.28647398147581, 0.125290724406541) circle (2pt);
\filldraw[fill = pinksalmon, fill opacity = 1] (3.85034377192445, 0.855538581016723) circle (2pt);
\filldraw[fill = pinksalmon, fill opacity = 1] (2.97808335282174, 0.316034294237087) circle (2pt);
\filldraw[fill = pinksalmon, fill opacity = 1] (3.54195314327038, 1.04628215084727) circle (2pt);
\filldraw[fill = pinksalmon, fill opacity = 1] (2.66969272416766, 0.506777864067633) circle (2pt);
\filldraw[fill = pinksalmon, fill opacity = 1] (1.92517188596223, 0.427769290508361) circle (2pt);
\filldraw[fill = pinksalmon, fill opacity = 1] (1.48904167641087, 0.158017147118543) circle (2pt);
\filldraw[fill = pinksalmon, fill opacity = 1] (0.744520838205434, 0.0790085735592717) circle (2pt);
\filldraw[fill = pinksalmon, fill opacity = 1] (0.872260419102717, 0.539504286779636) circle (2pt);
\filldraw[fill = pinksalmon, fill opacity = 1] (0.436130209551359, 0.269752143389818) circle (2pt);

\filldraw[fill = pinksalmon, fill opacity = 1] (5.52003649609211, 0.362316445084356) circle (2pt);
\filldraw[fill = pinksalmon, fill opacity = 1] (5.64777607698940, 0.822812158304720) circle (2pt);
\filldraw[fill = pinksalmon, fill opacity = 1] (5.21164586743804, 0.553060014914902) circle (2pt);
\filldraw[fill = pinksalmon, fill opacity = 1] (4.77551565788668, 0.283307871525084) circle (2pt);
\filldraw[fill = pinksalmon, fill opacity = 1] (4.46712502923261, 0.474051441355630) circle (2pt);
\filldraw[fill = pinksalmon, fill opacity = 1] (4.03099481968125, 0.204299297965812) circle (2pt);
\filldraw[fill = pinksalmon, fill opacity = 1] (4.59486461012989, 0.934547154575994) circle (2pt);
\filldraw[fill = pinksalmon, fill opacity = 1] (3.72260419102717, 0.395042867796358) circle (2pt);
\filldraw[fill = pinksalmon, fill opacity = 1] (3.10582293371902, 0.776530007457451) circle (2pt);
\filldraw[fill = pinksalmon, fill opacity = 1] (2.23356251461630, 0.237025720677815) circle (2pt);
\filldraw[fill = pinksalmon, fill opacity = 1] (2.79743230506494, 0.967273577287997) circle (2pt);
\filldraw[fill = pinksalmon, fill opacity = 1] (2.36130209551359, 0.697521433898179) circle (2pt);
\filldraw[fill = pinksalmon, fill opacity = 1] (2.05291146685951, 0.888265003728726) circle (2pt);
\filldraw[fill = pinksalmon, fill opacity = 1] (1.61678125730815, 0.618512860338908) circle (2pt);
\filldraw[fill = pinksalmon, fill opacity = 1] (1.18065104775679, 0.348760716949090) circle (2pt);
\filldraw[fill = pinksalmon, fill opacity = 1] (1.30839062865408, 0.809256430169454) circle (2pt);

\end{tikzpicture}
\caption{The embedding of $\mathcal{R}(\varepsilon, \mathfrak{D}_{E/F}^{-1})$ into $C(\varepsilon)$.}
\end{figure}
\normalsize

In order to give a uniform description of the points in $\mathcal{R}(\varepsilon, \mathfrak{D}_{E/F}^{-1})$, it is convenient to express them 
in terms of a $\Z$-basis for $\mathfrak{D}_{E/F}^{-1}$. In particular, for the $\Z$-basis given by  
\begin{align*}
\mathfrak{D}_{E/F}^{-1} = \Z \cdot 1 + \Z \cdot \left(\frac{6 + \sqrt{2}}{17}\right),
\end{align*} 
we find that   
\begin{align*}
\mathcal{R}(\varepsilon, \mathfrak{D}_{E/F}^{-1})=\left\{z_{m,n}:=-m + (4m + n -1)\left(\tfrac{6 + \sqrt{2}}{17}\right) \suchthat  0 \leq m \leq 8, ~ n \in S(m)\right\},
\end{align*}
where 
\begin{align*}
S(m):=
\begin{cases}
\{ 2, 3, 4 \} & \text{if $m = 0$}\\
\{1,  2, 3, 4 \} & \text{if $m = 1, 2, 3$}\\
\{ 1, 3 \} & \text{if $m = 4$}\\
\{ 0, 1, 2, 3 \} & \text{if $m = 5, 6, 7$}\\
\{ 0, 1, 2 \} & \text{if $m = 8$}.\\
\end{cases}
\end{align*}

We also wrote a program in \texttt{SageMath} to compute the character values
\begin{align*} 
c_{m,n}:=\chi_{E/F}(\mathfrak{D}_{E/F}\langle z_{m,n} \rangle) \in \{\pm 1\},
\end{align*}
which are given in the following table. 
\begin{table}[H]\label{table1}
{\renewcommand{\arraystretch}{1.2}
\begin{tabular}[h]{|c||c|c|c|c|c|c|c|c|c|} \hline
\multicolumn{10}{|c|}{Values of $c_{m,n}$} \\ \hline
\diagbox[width = 1.35cm, height = 0.7cm]{$n$}{$m$} & $0$ & $1$ & $2$ & $3$ & $4$ & $5$ & $6$ & $7$ & $8$ \\ \hline
$0$ & $$ & $$ & $$ & $$ & $$ & $-1$ & $1$ & $1$ & $1$ \\ \hline
$1$ & $$ & $-1$ & $-1$ & $1$ & $-1$ & $1$ & $1$ & $1$ & $-1$ \\ \hline
$2$ & $-1$ & $1$ & $-1$ & $-1$ & $$ & $-1$ & $-1$ & $1$ & $-1$ \\ \hline
$3$ & $-1$ & $1$ & $1$ & $1$ & $-1$ & $1$ & $-1$ & $-1$ & $$ \\ \hline
$4$ & $1$ & $1$ & $1$ & $-1$ & $$ & $$ & $$ & $$ & $$ \\ \hline
\end{tabular}
}
\caption{The character values $c_{m,n}:=\chi_{E/F}(\mathfrak{D}_{E/F}\langle z_{m,n} \rangle)$.}
\end{table}

Since $[\mathcal{O}_E^{\times}: \mathcal{O}_F^{\times}]=1$ and $h_E=1$, the preceding calculations yield the following explicit 
version of (\ref{lerchrealquad}), 
\begin{align}\label{lerchrealquad2}
\frac{L^{\prime}(\chi_{E/F},0)}{L(\chi_{E/F},0)} & =  -\log(17) + \frac{1}{2} \sum_{\substack{0 \leq m \leq 8 \\ n \in S(m)}} c_{m,n}
\log\left(\prod_{\sigma \in \Gal(F/\Q)} \Gamma_{2} \big(z_{m,n}^{\sigma}, (1, \varepsilon^{\sigma})\big)\right) - \frac{4\sqrt{2}}{17}\log(\varepsilon).
\end{align}

Finally, by combining (\ref{niceform2}) and (\ref{lerchrealquad2}) we get an  
explicit evaluation of $h_{\textrm{Fal}}(J_C)$ which is summarized in the following theorem.  

\begin{theorem} Let $C$ be the genus 2 hyperelliptic curve over $\Q(\sqrt{17})$ defined by (\ref{hyper}).
The Jacobian $J_C$ is a CM abelian surface defined over $\overline{\Q}$ with complex 
multiplication by the ring of integers $\mathcal{O}_E$ of the non-abelian quartic CM field $E=\Q(\sqrt{-5-2\sqrt{2}})$ 
with real quadratic subfield $F=\Q(\sqrt{2})$. The Faltings height of $J_C$ is given by 
\begin{align*} 
h_{\mathrm{Fal}}(J_C) & =  -\frac{1}{4}\sum_{\substack{0 \leq m \leq 8 \\ n \in S(m)}}c_{m,n}
\log\left(\prod_{\sigma \in \Gal(F/\Q)}\Gamma_{2} \big(z_{m,n}^{\sigma}, (1, \varepsilon^{\sigma})\big)\right) + \frac{2\sqrt{2}}{17}\log(\varepsilon)
+ \frac{1}{4}\log\left(\frac{17}{8}\right) -\log(2\pi),
\end{align*}
where $z_{m, n} = -m + (4m+n-1)(\frac{6 + \sqrt{2}}{17})$, $\varepsilon = 3 + 2\sqrt{2}$, and the numbers $c_{m,n} \in \{\pm 1\}$ are given in Table 1. 
\end{theorem}

\end{example}

\subsection{An arithmetic statistics approach to the Colmez conjecture}\label{probability}

In this section we develop an approach to the Colmez conjecture based on the study of certain problems of 
arithmetic distribution. 

\subsubsection{The density of Weyl CM fields when ordered by discriminant}

A natural way to count number fields $K/\Q$ which satisfy some property is to order them
by the absolute value of their discriminant $d_K$. 
Here we are interested in the problem of counting number fields 
(and in particular CM fields) with a given Galois group. This problem has a long history and has been studied extensively 
by many authors in recent years. See for example the excellent survey articles \cite{CDO06, Woo16}.

We start by introducing some notation. If $K/\Q$ is a number field, we denote its isomorphism class by $[K/\Q]$. 
For a permutation group $G$ on $n$ letters, we define the counting function
\begin{align*}
N_n(G, X):= \#\{ [K/\Q] \suchthat [K:\Q] = n, \, \Gal(K^s/\Q) \cong G \, \text{and } |d_K| \leq X \},
\end{align*}
which counts the number of isomorphism classes of number fields $K/\Q$ of degree $[K:\Q] = n$ such that
the Galois group of the Galois closure $K^s$ is $\Gal(K^s/\Q) \cong G$ and such that $|d_K| \leq X$. 

Similarly, in order to count isomorphism classes of number fields with a specific signature $(r_1, r_2)$, where $n = r_1 + 2r_2$,
we define the counting function
\begin{align*}
N_{r_1, r_2}(G, X):= \#\{ [K/\Q] \suchthat [K:\Q] = n, \, \op{signature}(K) = (r_1, r_2), \, \Gal(K^s/\Q) \cong G \, \text{and } |d_K| \leq X \}.
\end{align*}

Now, for CM fields we define the counting functions
\begin{align*}
\op{CM}_{n}(G, X):= \#\{ [E/\Q] \suchthat \text{$E$ is a CM field}, \, [E:\Q] = n, \, \Gal(E^s/\Q) \cong G \, \text{and } |d_E| \leq X \}
\end{align*}
and
\begin{align*}
\op{CM}_{n}(X):= \#\{ [E/\Q] \suchthat \text{$E$ is a CM field}, \, [E:\Q] = n \ \text{and} \ |d_E|  \leq X \}.
\end{align*}

We want to study the density of Weyl CM fields of fixed degree $2n$ when ordered by discriminant, i.e., we want to study the limit
\begin{align*}
\rho_{\text{Weyl}}(2n):=\lim_{X \to \infty} \frac{\op{CM}_{2n}(W_{2n}, X)}{\op{CM}_{2n}(X)}, 
\end{align*}
provided the limit exists. Conjectures of Malle \cite{Mal02, Mal04} and various refinements (see e.g. \cite{Bha07, Woo16}) 
concerning asymptotics for the counting functions $N_n(G,X)$ and $N_{r_1, r_2}(G,X)$ suggest that this limit exists and is positive. This is  
of great interest, for if $\rho_{\text{Weyl}}(2n) > 0$ then Theorem \ref{weyl} implies that the Colmez conjecture is true for 
a positive proportion of CM fields of fixed degree $2n$ when ordered by discriminant.

When $n = 1$, a CM field of degree $2$ is just an imaginary quadratic field.
In this case the Weyl group $W_{2} \cong \Z/2\Z$, so trivially every quadratic CM field is Weyl and hence
\begin{align*}
\rho_{\text{Weyl}}(2)=\lim_{X \to \infty} \frac{\op{CM}_{2}(W_{2}, X)}{\op{CM}_{2}(X)} = 1.
\end{align*}

When $n = 2$, the situation is already much more complicated. The following 
table can be extracted from \cite[p. 376]{Coh03}, and strongly suggests that $\rho_{\textrm{Weyl}}(4)$ exists and equals 1. 

\begin{table}[H]
\begin{tabular}{|c||c|c|c|}
\hline
$X$ & $\op{CM}_{4}(W_4, X)$ & $\op{CM}_{4}(X)$ & $\frac{\op{CM}_{4}(W_{4}, X)}{\op{CM}_{4}(X)}$ \\
\hline
$10^4$ & $27$ & $72$ & $37.5\%$ \\
\hline
$10^5$ & $395$ & $613$ & $64.4\%$ \\
\hline
$10^6$ & $4512$ & $5384$ & $83.8\%$ \\
\hline
$10^7$ & $47708$ & $51220$ & $93.1\%$ \\
\hline
$10^8$ & $486531$ & $500189$ & $97.3\%$ \\
\hline
$10^9$ & $4904276$ & $4956208$ & $98.9\%$ \\
\hline
$10^{10}$ & $49190647$ & $49384381$ & $99.6\%$ \\
\hline
$10^{12}$ & $4926673909$ & $4929271179$ & $99.9\%$ \\
\hline
\end{tabular}
\caption{Density of quartic Weyl CM fields}
\end{table}
\vspace{-0.05in}

In fact, we will appeal to the works of
Baily \cite{Bai80}, M\"aki \cite{Mak85}, and Cohen, Diaz y Diaz and Olivier \cite{CDO02, CDO05, CDO06} to deduce
the following result. 

\begin{theorem}\label{rhodensity} 
The density of quartic Weyl CM fields is
\begin{align*}
\rho_{\mathrm{Weyl}}(4)=\lim_{X \to \infty} \frac{\op{CM}_{4}(W_{4}, X)}{\op{CM}_{4}(X)} = 1.
\end{align*}
\end{theorem}

\begin{remark}
It follows from Theorem \ref{weyl} and Theorem \ref{rhodensity} that the Colmez conjecture is true for 100\% of 
quartic CM fields. On the other hand, we have already observed in Remark \ref{weylremark} that the Colmez conjecture is  
true for \textit{every} quartic CM field. Nonetheless, Theorem \ref{rhodensity} supports our belief that the probabilistic approach described here 
can be used to prove (at least in low degree) that the Colmez conjecture is true for a positive proportion of CM fields of fixed degree. 
We are currently investigating this problem for sextic CM fields. 
\end{remark}

\subsubsection{Abelian varieties over finite fields and density results} We now explain how 
to use density results for isogeny classes of abelian varieties over finite fields to prove probabilistic results about 
the Colmez conjecture. 

Let $\mathbb{F}_q$ be a finite field with $q=p^n$ elements. Let $\alpha_A$ be a root of the characteristic polynomial $f_A$
of the Frobenius endomorphism $\pi_A$ of an abelian variety $A/\mathbb{F}_{q}$ of dimension $g$. It is known that if
$A/\mathbb{F}_q$ and $B/\mathbb{F}_q$ are isogenous abelian varieties, then $f_A=f_B$.

Let $\mathcal{A}_g(q)$ be the set of isogeny classes of abelian varieties $A/\mathbb{F}_{q}$ of dimension $g$.
Let $K_{f_A}=\Q(\alpha_A)^s$ be the splitting field of $f_A$ and $\Gal(K_{f_A}/\Q)$ be the Galois group.
Kowalski \cite{Kow06} proved that the proportion of isogeny classes $[A] \in \mathcal{A}_g(p^n)$ which satisfy 
$\Gal(K_{f_A}/\Q) \cong W_{2g}$ approaches 1 as $n \rightarrow \infty$. We will show
that if $\Gal(K_{f_A}/\Q) \cong W_{2g}$ and
$g \geq 2$, then $\Q(\alpha_A)$ is a non-Galois Weyl CM field of degree $2g \geq 4$. By combining these results with
Theorem \ref{weyl}, we will establish the following probabilistic result.

\begin{theorem}\label{density} Suppose that $g \geq 2$. Then
\begin{align*}
\lim_{n \rightarrow \infty} \frac{\# \{[A] \in \mathcal{A}_g(p^n) \suchthat
\textrm{$\Q(\alpha_A)$ is a non-Galois CM field which satisfies the Colmez conjecture}\}}{\# \mathcal{A}_g(p^n)} = 1.
\end{align*}
\end{theorem}

On the other hand, let $\mathcal{A}_g^{s}(q)$ be the set of isogeny classes of \textit{simple} abelian varieties $A/\mathbb{F}_{q}$ of dimension $g$.
We will use work of Greaves-Odoni \cite{GO88} and Honda-Tate (see e.g. \cite{Tat71})
to prove that given a CM field $E$ of degree $2g$ and an integer $n \geq 2$, there is a set of prime numbers $p \in \Z$ with positive natural
density such that $E \cong \Q(\pi_A)$ for some simple abelian variety $A/\mathbb{F}_{p^n}$ of dimension
$g$. It seems likely that a modification of the methods in \cite{Kow06} can be used to prove that
the proportion of isogeny classes $[A] \in \mathcal{A}_g^s(p^n)$ which satisfy 
$\Gal(K_{f_A}/\Q) \cong W_{2g}$ approaches 1 as $n \rightarrow \infty$. As in Corollary \ref{density}, it would follow that if $g \geq 2$,
then
\begin{align*}
\lim_{n \rightarrow \infty} \frac{\# \{[A] \in \mathcal{A}_g^s(p^n) \suchthat
\text{$\Q(\pi_A)$ is a non-Galois CM field which satisfies the Colmez conjecture}\}}{\# \mathcal{A}_g^s(p^n)} = 1.
\end{align*}

\subsection{Outline of the proofs of the main results}\label{Outline}
We now briefly outline the proofs of Theorems \ref{maintheorem} and \ref{weyl}. 

Let $E$ be a CM field of degree $2n$ and $\Phi(E)$ be the set of CM types for $E$. Let $\Q^{\CM}$ be the compositum of all CM fields. 
Then the Galois group $G^{\CM}:=\Gal(\Q^{\CM}/\Q)$ acts on $\Phi(E)$ by composition.  
By a careful study of the action of $G^{\CM}$ on $\Phi(E)$ and a theorem of Colmez \cite[Th\'eoreme 0.3]{Col93} which relates the Faltings height of 
a CM abelian variety $X_{\Phi}$ of type $(\mathcal{O}_E, \Phi)$ to the ``height" of a certain locally constant function on 
$G^{\CM}$ constructed from the CM pair $(E, \Phi)$, we will prove that the Faltings height of 
$X_{\Phi}$ depends only on the $G^{\CM}$-orbit of $\Phi$. Given this result, we will prove that if the action 
of $G^{\CM}$ on $\Phi(E)$ is transitive, then an averaged version of the Colmez conjecture proved recently by 
Andreatta-Goren-Howard-Madapusi Pera \cite{AGHM15} and Yuan-Zhang \cite{YZ15} implies the Colmez conjecture for $E$. 

Now, let $\Phi$ be a CM type and
$E_{\Phi}$ be the associated reflex field. The reflex degree satisfies $[E_{\Phi}:\Q] \leq 2^n$. We will prove
that the action of $G^{\CM}$ on $\Phi(E)$
is transitive if and only if $[E_{\Phi}:\Q]=2^n$. 
In particular, by the results discussed in the previous paragraph, if  $[E_{\Phi}:\Q]=2^n$ then the Colmez conjecture is true for $E$. 
This leads to the problem of constructing CM fields with
reflex fields of maximal degree. 

Roughly speaking, Theorems A and B comprise two different ways of constructing infinite families of CM fields with reflex fields of maximal degree. 
Our approach to Theorem A is as follows. Let $F$ be a fixed totally real number field of degree $n \geq 3$.
Based on an idea of Shimura \cite{Shi70}, in Section \ref{reflexsection} we explicitly construct infinite families of CM extensions $E/F$
such that $E/\Q$ is non-Galois and the reflex fields $E_{\Phi}$ have maximal degree. 
This construction is quite elaborate, and consists of two main parts. First, in Proposition \ref{InfiniteCMFields} 
we explicitly construct infinite families of CM extensions $E/F$ 
with ``arbitrary'' prescribed ramification. Second, in Theorem \ref{Shimura} we prove that if $E/F$ is a CM extension satisfying 
a certain mild ramification condition, then the reflex fields $E_{\Phi}$ have maximal degree, and moreover, if $n \geq 3$ then $E/\Q$ is 
non-Galois. By combining these two results, we 
will obtain Theorem A. For the convenience of the reader, we have summarized this construction in Section \ref{algorithm}, Algorithm 1. 
On the other hand, to prove Theorem B, we will show that the reflex fields of a Weyl CM field have maximal degree.

\section{CM types and their equivalence}\label{CMSection}

In this section we prove some important facts that we will need regarding CM types and their equivalence.

Let $\Q^{\CM}$ be the compositum of all CM fields. Then $\Q^{\CM}/\Q$ is a Galois extension of infinite degree, and the
Galois group $G^{\CM}:= \op{Gal}(\Q^{\CM}/\Q)$ is a profinite group with the Krull topology.
Recall that the open sets of $G^{\CM}$ with the Krull topology are the empty set $\varnothing$ and the arbitrary unions
\begin{align*}
\bigcup_{i \in I} \sigma_i \op{Gal}(\Q^{\CM}/ E_i),
\end{align*}
where for every $i \in I$ we have $\sigma_i \in G^{\CM}$ and $\Q \subseteq E_i \subseteq \Q^{\CM}$ with $[E_i : \Q] < \infty$ and $E_i/\Q$ a Galois extension.
The group $G^{\CM}$ is Hausdorff, compact, and totally disconnected (see e.g. \cite[Chapter IV]{Mor96}). A function $f: G^{\CM} \rightarrow \overline{\Q}$
is \textit{locally constant} if for each $g \in G^{\CM}$, there is a neighborhood $N_g$ of $g$ such that $f$ is constant on $N_g$.

Let $c \in G^{\CM}$ denote complex conjugation.

\begin{definition}\label{UsualCM}
Let $E$ be a CM field of degree $2n$. A \textit{CM type} for $E$ is a set $\Phi_E$ consisting of
embeddings $E \hookrightarrow \overline{\Q}$ such that  $\op{Hom}(E, \overline{\Q}) = \Phi_E \bigcupdot c \Phi_E$.
We denote the set of all CM types for $E$ by $\Phi(E)$. The Galois group $G^{\CM}$ acts on $\Phi(E)$ as follows.
For $\Phi_E:= \{ \sigma_1, \dots, \sigma_n \} \in \Phi(E)$ and $\tau \in G^{\CM}$ let
\begin{align*}
\tau \cdot\Phi_E= \tau\Phi_E := \{ \tau \sigma_1, \dots, \tau \sigma_n \} \in \Phi(E).
\end{align*}
Two CM types $\Phi_E, \Phi_E^{'} \in \Phi(E)$ are said to be \textit{equivalent}
if they lie in the same orbit under the action of $G^{\CM}$, i.e., if there is an element $\tau \in G^{\CM}$ such that $\Phi_E = \tau \cdot \Phi_E^{'}$.
\end{definition}

We also have the following alternative definition.

\begin{definition} \label{FunctionalCM}
A \textit{CM type} is a locally constant function
$\Phi: G^{\CM} \longrightarrow \overline{\Q}$ such that
$\Phi(g) \in \{ 0, 1 \}$ and $\Phi(g) + \Phi(cg) = 1$ for every $g \in G^{\CM}$. We let
\begin{align*}
\CM := \{ \Phi: G^{\CM} \longrightarrow \overline{\Q} \suchthat \text{$\Phi$ is a CM type} \}
\end{align*}
be the set of all CM types. The Galois group $G^{\CM}$ acts on $\CM$ as follows.  For
$\Phi \in \CM$ and $\tau \in G^{\CM}$, let $\tau \cdot \Phi \in \CM$ be the CM type defined by
\begin{align*}
(\tau \cdot \Phi)(g):= \Phi(\tau^{-1} g) \quad \text{for every $g\in G^{\CM}$}.
\end{align*}
Two CM types $\Phi, \Phi^{'} \in \CM$ are said to be \textit{equivalent} if
they lie in the same orbit under the action of $G^{\CM}$, i.e., if there is an element $\tau \in G^{\CM}$ such that
$\Phi(g) = \Phi^{'}(\tau^{-1} g)$ for every $g \in G^{\CM}$.
\end{definition}

The following proposition gives a dictionary relating the two notions of a CM type and their equivalence.

\begin{proposition}\label{dictionary}
The following statements are true.
\begin{itemize}
\item[(i)] Let $E$ be a CM field and $\Phi_E \in \Phi(E)$.
Define the function $\Phi: G^{\CM} \longrightarrow \overline{\Q}$ by
\begin{align*}
\Phi(g):= \chi_{\Phi_E}(g|_{E}), \quad g \in G^{\CM}
\end{align*}
where $ \chi_{\Phi_E}$ denotes the characteristic function of the set $\Phi_E$ and $g|_E$ is the restriction of $g$ to $E$. Then $\Phi \in \CM$. Moreover, if
$\Phi_{E}^{\prime} \in \Phi(E)$ is equivalent to $\Phi_E$ and $\tau \in G^{\CM}$ is such that $\Phi_E=\tau \cdot \Phi_E^{\prime}$,
then $\Phi^{\prime}$ is equivalent to $\Phi$ with $\Phi = \tau \cdot \Phi^{\prime}$.
\vspace{0.05in}

\item[(ii)] Let $\Phi \in \CM$. Then there exists a Galois CM field $E$ such that for every $g \in G^{\CM}$
and every $h \in \op{Gal}(\Q^{\CM}/E)$, we have $\Phi(gh) = \Phi(g)$. Moreover, if $[g]:=g\op{Gal}(\Q^{\CM}/E)$ and we define
\begin{align*}
\Phi_E:=\{ \sigma \in \op{Hom}(E, \overline{\Q}) \suchthat \text{there exists} ~ g \in G^{\CM} \text{ with $\sigma = g|_E$ and $\Phi([g]) = \{ 1 \}$} \},
\end{align*}
then $\Phi_E \in \Phi(E)$. Finally, if $\Phi^{\prime} \in \CM$ is equivalent to $\Phi$ and $\tau \in G^{\CM}$ is such that $\Phi=\tau \cdot \Phi^{\prime}$,
then for every $g \in G^{\CM}$ and every $h \in \op{Gal}(\Q^{\CM}/E)$, we have $\Phi^{\prime}(gh) = \Phi^{\prime}(g)$,
and $\Phi_E^{\prime}$ is equivalent to $\Phi_E$ with $\Phi_E = \tau \cdot \Phi_E^{\prime}$.
\end{itemize}
\end{proposition}

For clarity we divide the proof of Proposition \ref{dictionary} into the following two subsections.

\subsection{Proof of Proposition \ref{dictionary} (i)}\label{functional}
Let $E$ be a CM field and $\Phi_E \in \Phi(E)$ be a CM type for $E$.
Define the function $\Phi: G^{\CM} \longrightarrow \{ 0, 1 \}$ by
\begin{align*}
\Phi(g):= \chi_{\Phi_E}(g|_{E}), \quad g \in G^{\CM}
\end{align*}
where $\chi_{\Phi_E}$ is the characteristic function of the set
$\Phi_E$ and $g|_E \in \op{Hom}(E, \overline{\Q})$ is the restriction of $g$ to $E$.
We now prove that $\Phi \in \CM $.

Let $g \in G^{\CM}$. Since $\op{Hom}(E, \overline{\Q}) = \Phi_E \bigcupdot c \Phi_E$, we either have
$g|_E \in \Phi_E$ or $g|_E \in c\Phi_E$, or equivalently, $g|_E \in \Phi_E$ or $(cg)|_E \in \Phi_E$.
This proves that $\Phi(g) + \Phi(cg) = 1$. It remains to prove that $\Phi$ is locally constant.
Let $E^{s}$ be the Galois closure of $E$. Then $E^s$ is also a CM field (see e.g. \cite[Proposition 5.12]{Shi94}), and it follows that
$g \op{Gal}(\Q^{\CM}/E^{s})$ is an open set containing $g$.
Now, observe that for any $h \in \op{Gal}(\Q^{\CM}/E^s)$, we have $h|_E = \textrm{id}_E$, so that $(gh)|_E = g |_E$.
Therefore
\begin{align*}
\Phi(gh) = \chi_{\Phi_E}((gh)|_{E}) = \chi_{\Phi_E}(g|_{E}) = \Phi(g),
\end{align*}
which implies that $\Phi$ is constant on $g \op{Gal}(\Q^{\CM}/E^s)$.
It follows that $\Phi$ is locally constant, and hence $\Phi \in \CM$.

Now, suppose that $\Phi_E$ and $\Phi_E^{'}$ are equivalent CM types for $E$.
Let $\tau \in G^{\CM}$ be such $\Phi_E = \tau \Phi_E^{'}$. Then for an arbitrary element $g \in G^{\CM}$,
the corresponding CM types $\Phi, \Phi^{'} \in \CM$ satisfy
\begin{align*}
\Phi(g) = \chi_{\Phi_E}(g|_E) = \chi_{\tau \Phi_E^{'}}(g|_E) = \chi_{\Phi_E^{'}}((\tau^{-1} g)|_E) = \Phi^{'}(\tau^{-1} g).
\end{align*}
Therefore, $\Phi$ is equivalent to $\Phi^{\prime}$ with $\Phi = \tau \cdot \Phi^{\prime}$. This completes the proof of
Proposition \ref{dictionary} (i). \qed

\subsection{Proof of Proposition \ref{dictionary} (ii)}\label{regular}

The first assertion of Proposition \ref{dictionary} (ii) is proved in the following lemma.

\begin{lemma}\label{CMFieldE}
Let $\Phi \in \CM$ be a CM type. Then there exists a Galois CM field $E$ such that for every $g \in G^{\CM}$ and every $h \in \op{Gal}(\Q^{\CM}/E)$
we have $\Phi(gh) = \Phi(g)$.
\end{lemma}

\begin{proof} Let $g \in G^{\CM}$. Since $\Phi$ is locally constant, there exists an open set $U_g$ containing $g$ such that $\Phi$ is
constant on $U_g$. Now, by definition of the Krull topology we have
\begin{align*}
U_g=\bigcup_{i \in I}g_i\Gal(\Q^{\CM}/E_i),
\end{align*}
where for every $i \in I$ we have $g_i \in G^{\CM}$ and $\Q \subseteq E_i \subseteq \Q^{\CM}$ with $[E_i : \Q] < \infty$ and $E_i/\Q$ a Galois extension.
Since $g \in U_g$, we have $g \in g_{i_0}\Gal(\Q^{\CM}/E_{i_0})$ for some $i_{0} \in I$. It follows that $g\Gal(\Q^{\CM}/E_{i_0}) \subseteq g_{i_0}\Gal(\Q^{\CM}/E_{i_0})$.
Let $E_g$ be any Galois CM field containing $E_{i_0}$. Then
\begin{align*}
g\Gal(\Q^{\CM}/E_g) \subseteq g\Gal(\Q^{\CM}/E_{i_0}) \subseteq g_{i_0}\Gal(\Q^{\CM}/E_{i_0}) \subseteq U_g.
\end{align*}
From the preceding facts, we conclude that

\begin{align*}
\{ g\Gal(\Q^{\CM}/E_g) \suchthat  g \in G^{\CM} \}
\end{align*}
is an open cover of $G^{\CM}$ such that $\Phi$ is constant on each of the
sets $g\Gal(\Q^{\CM}/E_g)$.

Now, since $G^{\CM}$ is compact, there exists a finite subcover $$\{ g_j\Gal(\Q^{\CM}/E_{g_j}) \}_{j = 1}^{r}$$
for some elements $g_j \in G^{\CM}$. Let $E:= E_{g_1} \cdots E_{g_r}$ be the compositum of the Galois CM fields $E_{g_j}$.
Then $E$ is a Galois CM field (see e.g. \cite[Proposition 5.12]{Shi94}). To complete the proof, we will show that
$\Phi$ is constant on $g \Gal(\Q^{\CM}/E)$ for every $g \in G^{\CM}$.

Since $\Phi$ is constant on each $g_j\Gal(\Q^{\CM}/E_{g_j})$, it suffices to show that there
exists an integer $j \in \{ 1, \dots, r \}$ such that $g \Gal(\Q^{\CM}/E) \subset g_j\Gal(\Q^{\CM}/E_{g_j})$.
Since $$\{ g_j\Gal(\Q^{\CM}/E_{g_j}) \}_{j = 1}^{r}$$  covers $G^{\CM}$, there exists an integer $j \in \{ 1, \dots, r \}$
such that $g \in g_j\Gal(\Q^{\CM}/E_{g_j})$. This implies that $g = g_j h_j$ for some $h_j \in \Gal(\Q^{\CM}/E_{g_j})$.
Let $\sigma \in g \Gal(\Q^{\CM}/E)$. Then $\sigma = gh$ for some $h \in \Gal(\Q^{\CM}/E)$, hence
$\sigma = g_j h_j h$. Moreover, since $\Gal(\Q^{\CM}/E) \subset \Gal(\Q^{\CM}/E_{g_i})$, we have $h_j h \in \Gal(\Q^{\CM}/E_{g_j})$. It
follows that $\sigma \in g_j\Gal(\Q^{\CM}/E_{g_j})$, and so $g \Gal(\Q^{\CM}/E) \subset g_j\Gal(\Q^{\CM}/E_{g_j})$, as desired.
\end{proof}

We now prove the second assertion of Proposition \ref{dictionary} (ii).
Let $\Phi \in \CM$ be a CM type. By Lemma \ref{CMFieldE}, there exists a Galois CM field
$E$ such that $\Phi$ is constant on $g \op{Gal}(\Q^{\CM}/E)$ for every $g \in G^{\CM}$. For notational convenience,
we define $[g]:= g \op{Gal}(\Q^{\CM}/E)$. Since $E/\Q$ is Galois, we have $\op{Hom}(E, \overline{\Q}) = \op{Gal}(E/\Q)$, and
\begin{align}\label{hom}
\frac{G^{\CM}}{\Gal(\Q^{\CM}/E)} \cong \op{Hom}(E, \overline{\Q}).
\end{align}
Define the set
\begin{align*}
\Phi_E:=\{ \sigma \in \op{Hom}(E, \overline{\Q}) \suchthat \text{there exists} ~ g \in G^{\CM} \text{ with $\sigma = g|_E$ and $\Phi([g]) = \{ 1 \}$} \}.
\end{align*}
We now show that $\Phi_E \in \Phi(E)$.

By (\ref{hom}), given an element $\sigma \in \op{Hom}(E, \overline{\Q})$,
there is a unique coset $[g] \in G^{\CM}/\Gal(\Q^{\CM}/E)$ such that $\sigma = g|_E$.
Since $\Phi$ is constant on each coset $[g]$, it follows that either $\Phi([g]) = \{ 0 \}$ or $\Phi([g]) = \{1 \}$.
Suppose that $\sigma \not \in \Phi_E$. Then $\Phi([g]) = \{ 0 \}$, so that $\Phi([cg]) = \{ 1\}$. Moreover, we have
$c \sigma = (cg)|_E$, and thus $c\sigma \in \Phi_E$, or equivalently, $\sigma \in c \Phi_E$. A short calculation shows
that $\Phi_E \cap c\Phi_E = \varnothing$. Hence
$\op{Hom}(E, \overline{\Q}) = \Phi_E \bigcupdot c \Phi_E$, and we conclude that $\Phi_E \in \Phi(E)$.

Finally, we prove the third assertion of Proposition \ref{dictionary} (ii).  Suppose that
$\Phi^{'} \in \CM$ is equivalent to $\Phi$. Let
$\tau \in G^{\CM}$ be such that $\Phi= \tau \cdot \Phi^{\prime}$, i.e.
$\Phi(g) = \Phi^{'}(\tau^{-1} g)$ for every $g \in G^{\CM}$.  Since $\Phi$ is
constant on $\tau g \op{Gal}(\Q^{\CM}/E)$, it follows that
for every $g \in G^{\CM}$ and every $h \in \Gal(\Q^{\CM}/E)$, we have
\begin{align*}
\Phi^{'}(gh) = \Phi(\tau gh) = \Phi(\tau g) = \Phi^{'}(g).
\end{align*}
Let
\begin{align*}
\Phi_E^{'}:=\{ \sigma' \in \op{Hom}(E, \overline{\Q}) \suchthat \text{there exists} ~ g' \in G^{\CM} \text{ with $\sigma = g'|_E$ and $\Phi^{'}([g']) = \{ 1 \}$} \}.
\end{align*}
We will prove that $\Phi_E = \tau \Phi_E^{'}$.
We need only prove the containment $\Phi_E \subseteq \tau \Phi_E^{'}$, since the reverse containment can be proved mutatis mutandis.
Let $\sigma \in \Phi_E$ and let $g \in G^{\CM}$ be such that $\sigma = g|_E$ and $\Phi(g) = 1$.
Then this implies that $\Phi^{'}(\tau^{-1} g) = 1$. Finally, let $\sigma' := (\tau^{-1} g)|_E$. Then $\sigma' \in \Phi_E^{'}$,
and moreover $\sigma = \tau \sigma'$, so that $\sigma \in \tau \Phi_E^{'}$. Hence $\Phi_E \subseteq \tau \Phi_E^{'}$.
This completes the proof of Proposition 2.3 (ii). \qed

\vspace{0.10in}

\textbf{Important Remark}. \textit{In light of Proposition \ref{dictionary}, from here forward
we will use the two different notions of CM type and equivalence of CM types interchangeably, leaving it to the reader
to distinguish  which notion is being used from the context}.

\section{Faltings heights and the Colmez conjecture}\label{colmezstatement}

In this section we review the statement of the Colmez conjecture, following closely the discussion in \cite{Col98} and \cite{Yan10b}.

We begin by recalling the definition of the Faltings height of a CM abelian variety.
Let $F$ be a totally real number field of degree $n$. Let $E/F$ be a CM extension of $F$ and $\Phi \in \Phi(E)$
be a CM type for $E$. Let $X_{\Phi}$ be an abelian variety
defined over $\overline{\mathbb Q}$ with complex multiplication by $\mathcal{O}_E$ and
CM type $\Phi$. We call $X_{\Phi}$ a CM abelian variety of type $(\mathcal{O}_E,\Phi)$.
Let $K \subseteq \overline{\mathbb Q}$ be a number field over which $X_{\Phi}$ has everywhere good reduction
and choose a N\'eron differential $\omega \in H^{0}(X_{\Phi}, \Omega^n_{X_{\Phi}})$. Then the \textit{Faltings height} of $X_{\Phi}$ is defined by
\begin{align*}
h_{\textrm{Fal}}(X_{\Phi}):= -\frac{1}{2[K:\mathbb Q]}\sum_{\sigma : K \hookrightarrow \mathbb{C}}
\log\left| \int_{X_{\Phi}^{\sigma}(\mathbb{C})} \omega^{\sigma} \wedge \overline{\omega^{\sigma}}\right|.
\end{align*}
The Faltings height does not depend on the choice of $K$ or $\omega$.
Moreover, Colmez \cite{Col93} proved that if $X_{\Phi}$ and $Y_{\Phi}$ are CM abelian varieties of
type $(\mathcal{O}_E, \Phi)$, then $h_{\textrm{Fal}}(X_{\Phi})=h_{\textrm{Fal}}(Y_{\Phi})$, i.e., the Faltings height depends
on the CM type $\Phi$, but does not depend on the choice of CM abelian variety $X_{\Phi}$.

Let $H(G^{\CM}, \overline{\Q})$ be the Hecke algebra of Schwartz functions on the Galois group $G^{\CM}$
which take values in $\overline{\Q}$ (see e.g. \cite{Win89}). This is the $\overline{\Q}$-algebra of locally constant, compactly supported
functions $f: G^{\CM} \longrightarrow \overline{\Q}$ with multiplication of functions $f_1, f_2 \in H(G^{\CM}, \overline{\Q})$ given by the convolution
\begin{align*}
(f_1 * f_2) (g) := \int_{G^{\CM}} f_1(h) f_2(h^{-1}g) \, d\mu(h).
\end{align*}
Here $\mu$ is the left-invariant Haar measure on $G^{\CM}$, normalized so that
\begin{align*}
\op{Vol}(G^{\CM}) = \int_{G^{\CM}} \, d\mu(g) = 1.
\end{align*}

The Hecke algebra $H(G^{\CM}, \overline{\Q})$ is an associative algebra with no identity element. For a function $f \in H(G^{\CM}, \overline{\Q})$,
the \textit{reflex function} $f^{\vee} \in H(G^{\CM}, \overline{\Q})$ is defined by $f^{\vee}(g) := \overline{f(g^{-1})}$.
We define a Hermitian inner product on $H(G^{\CM}, \overline{\Q})$ by
\begin{align*}
\langle f_1, f_2 \rangle := \int_{G^{\CM}} f_1(h) \overline{f_2(h)} \, d\mu(h).
\end{align*}

Let $H^0(G^{\CM}, \overline{\Q})$ be the $\overline{\Q}$-subalgebra of $H(G^{\CM}, \overline{\Q})$ of class functions,
i.e., the $\overline{\Q}$-subalgebra of functions $f \in H(G^{\CM}, \overline{\Q})$ satisfying $f(hgh^{-1}) = f(g)$ for all $h, g \in G^{\CM}$.
It is known that an orthonormal basis for $H^0(G^{\CM}, \overline{\Q})$ is given by
the set $$\{\chi_{\pi} \suchthat \text{$\pi$ an irreducible representation of $G^{\CM}$}\}$$ of Artin characters $\chi_{\pi}$
associated to the irreducible representations $\pi$ of $G^{\CM}$.

There is a projection map
\begin{align*}
H(G^{\CM}, \overline{\Q}) &\longrightarrow H^0(G^{\CM}, \overline{\Q})\\
f& \longmapsto f^0
\end{align*}
defined by
\begin{align*}
f^0(g):=  \int_{G^{\CM}} f(hgh^{-1}) \, d\mu(h).
\end{align*}
As a map of $\overline{\Q}$-vector spaces, it corresponds to the orthogonal projection of
$H(G^{\CM}, \overline{\Q})$ onto $H^0(G^{\CM}, \overline{\Q})$. In particular, one has
\begin{align*}
f^0 = \sum_{\chi_{\pi}} \langle f, \chi_{\pi} \rangle \chi_{\pi}.
\end{align*}

Define the functions
\begin{align*}
Z(f^{0},s):=\sum_{\chi_{\pi}}\langle f , \chi_{\pi}\rangle \frac{L^{\prime}(\chi_{\pi},s)}{L(\chi_{\pi},s)} \quad \textrm{and} \quad
\mu_{\textrm{Art}}(f^{0}):=\sum_{\chi_{\pi}}\langle  f , \chi_{\pi} \rangle \log(\frak{f}_{\chi_{\pi}}),
\end{align*}
where $L(\chi_{\pi},s)$ is the (incomplete) Artin $L$--function of $\chi_{\pi}$ and $\frak{f}_{\chi_{\pi}}$ is the analytic Artin conductor of $\chi_{\pi}$.

If $\Phi \in \CM$ is a CM type, we define the function $A_{\Phi}  \in H(G^{\CM}, \overline{\Q})$ by
\begin{align*}
A_{\Phi}:= \Phi * \Phi^{\vee}.
\end{align*}

Colmez \cite{Col93} made the following conjecture.

\begin{conjecture}[Colmez \cite{Col93}]\label{cc2} Let $E$ be a CM field, $\Phi$ be a CM type for $E$, and
$X_{\Phi}$ be a CM abelian variety of type $(\mathcal{O}_E, \Phi)$. Let $A_{E, \Phi}:=[E:\Q]A_{\Phi}.$  Then
\begin{align*}
h_{\mathrm{Fal}}(X_{\Phi})=-Z(A_{E, \Phi}^{0},0) - \frac{1}{2}\mu_{\mathrm{Art}}(A_{E, \Phi}^{0}).
\end{align*}
\end{conjecture}

Colmez \cite{Col93} proved Conjecture \ref{cc2} when $E/\Q$ is abelian,
up to addition of a rational multiple of $\log(2)$ which was recently shown to equal zero by Obus \cite{Obu13}.
Yang \cite{Yan10a, Yan10b, Ya13} proved Conjecture \ref{cc2} for a large class of non-biquadratic quartic CM fields, thus
establishing the only known cases of the Colmez conjecture when $E/\Q$ is \textit{non-abelian}.

\section{The average Colmez conjecture}\label{AverageColmezConjecture}

Let $F$ be a totally real number field of degree $n$. Let $E/F$ be a CM extension of $F$ and $\Phi(E)$ be the set of CM types for $E$. There are
$2^n$ CM types $\Phi \in \Phi(E)$. By averaging both sides of Conjecture \ref{cc2} over $\Phi(E)$, one gets the conjectural identity
\begin{align}\label{acc1}
\frac{1}{2^n}\sum_{\Phi \in \Phi(E)} h_{\mathrm{Fal}}(X_{\Phi})& = \frac{1}{2^n}\sum_{\Phi \in \Phi(E)}(- Z(A_{E, \Phi}^{0},0) - \frac{1}{2}\mu_{\mathrm{Art}}(A_{E, \Phi}^{0})).
\end{align}
The average on the right hand side of (\ref{acc1}) can be simplified. Namely, by \cite[Proposition 8.4.1]{AGHM15} we have
\begin{align}\label{Howard}
\frac{1}{2^n}\sum_{\Phi \in \Phi(E)}(- Z(A_{E, \Phi}^{0},0) - \frac{1}{2}\mu_{\mathrm{Art}}(A_{E, \Phi}^{0})) =
-\frac{1}{2}\frac{L^{\prime}(\chi_{E/F},0)}{L(\chi_{E/F},0)}
-\frac{1}{4}\log\left(\frac{|d_E|}{d_F}\right) - \frac{n}{2}\log(2\pi),
\end{align}
where $L(\chi_{E/F},s)$ is the (incomplete) $L$--function of the Hecke character $\chi_{E/F}$ associated to the quadratic extension $E/F$ and
$d_E$ (resp. $d_F$) is the discriminant of $E$ (resp. $F$).

These identities yield the following averaged version of the Colmez conjecture.

\begin{conjecture}[The Average Colmez Conjecture]\label{acc3} Let $F$ be a totally real number field of degree $n$.
Let $E/F$ be a CM extension of $F$, and for each CM type $\Phi \in \Phi(E)$, let $X_{\Phi}$ be a CM abelian variety of
type $(\mathcal{O}_E, \Phi)$. Then
\begin{align}\label{acc2}
\frac{1}{2^n}\sum_{\Phi \in \Phi(E)} h_{\mathrm{Fal}}(X_{\Phi}) = -\frac{1}{2}\frac{L^{\prime}(\chi_{E/F},0)}{L(\chi_{E/F},0)}
-\frac{1}{4}\log\left(\frac{|d_E|}{d_F}\right) - \frac{n}{2}\log(2\pi).
\end{align}
\end{conjecture}

Conjecture \ref{acc3} was recently proved independently by Andreatta-Goren-Howard-Madapusi Pera \cite{AGHM15} and Yuan-Zhang \cite{YZ15}.

\begin{theorem}[\cite{AGHM15}, \cite{YZ15}]\label{howard} Conjecture \ref{acc3} is true.
\end{theorem}

\section{The action of $G^{\CM}$ on $\Phi(E)$ and the Colmez conjecture}

In this section we prove the following result.

\begin{proposition}\label{cc} Let $F$ be a totally real number field of degree $n$. Let $E/F$ be a CM extension of $F$ and
$\Phi(E)$ be the set of CM types for $E$. If the action of $G^{\CM}$ on $\Phi(E)$ is transitive,
then Conjecture \ref{cc2} is true. In particular, if $\Phi \in \Phi(E)$ and
$X_{\Phi}$ is a CM abelian variety of type $(\mathcal{O}_E, \Phi)$, then
\begin{align}\label{FHid}
h_{\mathrm{Fal}}(X_{\Phi})= -\frac{1}{2}\frac{L^{\prime}(\chi_{E/F},0)}{L(\chi_{E/F},0)}
-\frac{1}{4}\log\left(\frac{|d_E|}{d_F}\right) - \frac{n}{2}\log(2\pi).
\end{align}
\end{proposition}

We will need the following two crucial lemmas.

\begin{lemma}\label{AzeroLemma}
If $\Phi_1, \Phi_2 \in \CM$ are equivalent CM types, then $A_{\Phi_1}^0 = A_{\Phi_2}^0$.
\end{lemma}

\begin{proof}
Since the CM types $\Phi_1$ and $\Phi_2$ are equivalent, there is an element
$\tau^{-1} \in G^{\CM}$ such that $\Phi_1(g) = \Phi_2(\tau g)$ for every $g \in G^{\CM}$. Then we have
\begin{align}\label{inner}
A_{\Phi_1}^0(g) &= \int_{G^{\CM}} A_{\Phi_1}(hgh^{-1}) \, d\mu(h) \notag \\
&= \int_{G^{\CM}} \int_{G^{\CM}} \Phi_1(t) \Phi_1^{\vee}(t^{-1}hgh^{-1}) \, d\mu(t) d\mu(h) \notag \\
&= \int_{G^{\CM}} \int_{G^{\CM}} \Phi_1(t) \Phi_1(hg^{-1}h^{-1} t) \, d\mu(h) d\mu(t) \notag \\
&= \int_{G^{\CM}} \int_{G^{\CM}} \Phi_2(\tau t) \Phi_2(\tau hg^{-1}h^{-1} t) \, d\mu(h) d\mu(t) \notag \\
&= \int_{G^{\CM}} \Phi_2(\tau t) \left( \int_{G^{\CM}}  \Phi_2(\tau hg^{-1}h^{-1} \tau^{-1} \tau t) \, d\mu(h) \right) d\mu(t).
\end{align}

Now, define the function $f_{g, \tau, t}(h):= \Phi_2 (hg^{-1} h^{-1} \tau t)$. Then the inner integral in (\ref{inner}) can be written as
\begin{align}\label{inner2}
 \int_{G^{\CM}}  \Phi_2(\tau hg^{-1}h^{-1} \tau^{-1} \tau t) \, d\mu(h)&=\int_{G^{\CM}}  f_{g, \tau, t}( \tau h) \, d\mu(h)\notag \\
&= \int_{G^{\CM}}  f_{g, \tau, t}( h) \, d\mu(h) \notag\\
&=\int_{G^{\CM}}  \Phi_2 (hg^{-1} h^{-1} \tau t) \, d\mu(h),
\end{align}
where in the second equality we used the left-invariance of the Haar measure. We substitute the identity (\ref{inner2}) for the
inner integral in (\ref{inner}) and continue the calculation to get
\begin{align*}
A_{\Phi_1}^0(g) &=  \int_{G^{\CM}} \Phi_2(\tau t) \left( \int_{G^{\CM}} \Phi_2 (hg^{-1} h^{-1} \tau t) \, d\mu(h) \right) d\mu(t) \\
&= \int_{G^{\CM}} \left( \int_{G^{\CM}}\Phi_2(\tau t)  \Phi_2 (hg^{-1} h^{-1} \tau t) \, d\mu(t) \right) d\mu(h) \\
&= \int_{G^{\CM}} \left( \int_{G^{\CM}}\Phi_2(t)  \Phi_2 (hg^{-1} h^{-1} t) \, d\mu(t) \right) d\mu(h)\\
&= \int_{G^{\CM}} \left( \int_{G^{\CM}}\Phi_2(t)  \Phi_2^{\vee} (t^{-1}hg h^{-1}) \, d\mu(t) \right) d\mu(h)\\
&=\int_{G^{\CM}} A_{\Phi_2}(hgh^{-1}) \, d\mu(h)\\
&= A_{\Phi_2}^0(g),
\end{align*}
where in the third equality we again used the left-invariance of the Haar measure.
\end{proof}

\begin{lemma}\label{htequality} Let $E$ be a CM field, let $\Phi_1$ and $\Phi_2$ be CM types for $E$, and let $X_{\Phi_1}$ and $X_{\Phi_2}$ be
CM abelian varieties of types $(\mathcal{O}_E, \Phi_1)$ and $(\mathcal{O}_E, \Phi_2)$, respectively. If $\Phi_1$ and $\Phi_2$ are equivalent,
then $$h_{\mathrm{Fal}}(X_{\Phi_1})=h_{\mathrm{Fal}}(X_{\Phi_2}).$$
\end{lemma}

\begin{proof} Let $X_{\Phi}$ be a CM abelian variety of type $(\mathcal{O}_E, \Phi)$. Then by Colmez \cite[Th\'eoreme 0.3]{Col93},
there is a unique $\Q$-linear height function $ \textrm{ht}: H^{0}(G^{\CM}, \overline{\Q}) \rightarrow \R$
such that
\begin{align}\label{ht}
h_{\mathrm{Fal}}(X_{\Phi}) = -\textrm{ht}(A_{E, \Phi}^{0}) -\frac{1}{2}\mu_{\textrm{Art}}(A_{E, \Phi}^{0}).
\end{align}
Since $\Phi_1$ and $\Phi_2$ are equivalent, by Lemma \ref{AzeroLemma} we have $A_{\Phi_1}^{0}=A_{\Phi_2}^{0}$, so that
\begin{align*}
A_{E, \Phi_1}^{0} = [E:\Q]A_{\Phi_1}^{0} = [E:\Q]A_{\Phi_2}^{0} = A_{E, \Phi_2}^{0}.
\end{align*}
It follows from (\ref{ht})
that $h_{\mathrm{Fal}}(X_{\Phi_1})=h_{\mathrm{Fal}}(X_{\Phi_2})$.
\end{proof}
\vspace{0.05in}

\textbf{Proof of Proposition \ref{cc}}. Fix a CM type $\Phi_0 \in \Phi(E)$,
and let $X_{\Phi_0}$ be a CM abelian variety of type $(\mathcal{O}_E, \Phi_{0})$.
Since the action of $G^{\CM}$ on $\Phi(E)$ is transitive, we have
\begin{align}\label{middle}
h_{\textrm{Fal}}(X_{\Phi_0}) &= \frac{1}{2^n}\sum_{\Phi \in \Phi(E)}h_{\textrm{Fal}}(X_{\Phi})\notag\\
&=\frac{1}{2^n}\sum_{\Phi \in \Phi(E)}(- Z(A_{E, \Phi}^{0},0) - \frac{1}{2}\mu_{\mathrm{Art}}(A_{E, \Phi}^{0}))\\
&=- Z(A_{E, \Phi_{0}}^{0},0) - \frac{1}{2}\mu_{\mathrm{Art}}(A_{E, \Phi_{0}}^{0}),\notag
\end{align}
where the first equality follows from Lemma  \ref{htequality}, the second equality is the identity
(\ref{acc1}) (which is equivalent to Theorem \ref{howard}),
and the third equality follows from Lemma \ref{AzeroLemma}. Since $\Phi_0$ was
arbitrary, this proves Conjecture \ref{cc2}. The identity (\ref{FHid}) for the Faltings height then
follows from (\ref{middle}) and (\ref{Howard}). \qed

\section{The action of $G^{\CM}$ on $\Phi(E)$ and the reflex degree}

In this section we relate the action of $G^{\CM}$ on $\Phi(E)$ to the degree of the
reflex field of a CM pair $(E,\Phi)$.

Let $G_{\Q}:=\textrm{Gal}(\overline{\Q}/\Q)$ be the absolute Galois group.
The following result can be found in \cite[Proposition 1.16]{Mil06} and \cite[Proposition 28]{Shi98}, for example.

\begin{proposition}\label{ReflexField}
Let $E$ be a CM field and $\Phi$ be a CM type for $E$. Then the following conditions on a subfield
$E_{\Phi}$ of $\overline{\Q}$ are equivalent.
\begin{itemize}
\item[(i)] We have
\begin{align*}
\{ \sigma \in G_{\Q} \suchthat \sigma \, \text{fixes $E_{\Phi}$} \}
= \{ \sigma \in G_{\Q} \suchthat \sigma \Phi = \Phi \},
\end{align*}
that is, $\displaystyle{\op{Gal}(\overline{\Q}/E_{\Phi}) = \op{Stab}_{G_{\Q}}(\Phi)}$.

\item[(ii)] $\displaystyle{E_{\Phi}= \Q \left( \left \{ \op{Tr}_{\Phi}(a) \suchthat a \in E  \right\} \right) }$,
where $\displaystyle{ \op{Tr}_{\Phi}(a):= \sum_{\phi \in \Phi} \phi(a) }$ is the type trace of $a \in E$.
\end{itemize}
\end{proposition}

\begin{definition}
The field $E_{\Phi}$ satisfying the equivalent conditions in Proposition \ref{ReflexField}
is called the \textit{reflex field} of the CM pair $(E, \Phi)$.
\end{definition}

Let $E^s$ denote the Galois closure of $E$.

\begin{proposition}\label{orbit}
Let $E$ be a CM field of degree $2n$ and $\Phi$ be a CM type for $E$. Then
\begin{align*}
[E_{\Phi} : \Q] = \# (\op{Gal}(E^{s}/\Q) \cdot \Phi).
\end{align*}
In particular, $[E_{\Phi}:\Q] \leq 2^n$.
\end{proposition}

\begin{proof} By  Proposition \ref{ReflexField} (ii) we have $E_{\Phi} \subseteq E^{s}$, hence one can replace
$\overline{\Q}$ with $E^s$ and $G_{\Q}$ with $\Gal(E^s/\Q)$ in Proposition \ref{ReflexField} (i) to conclude that
\begin{align}\label{ReflexStab}
\op{Gal}(E^{s}/E_{\Phi}) = \op{Stab}_{\op{Gal}(E^{s}/\Q)}(\Phi).
\end{align}
Then using the fundamental theorem of Galois theory, identity (\ref{ReflexStab}), and the orbit-stabilizer theorem, we have
\begin{align*}
[E_{\Phi} : \Q] = [\op{Gal}(E^{s}/\Q): \op{Gal}(E^{s}/E_{\Phi})] = [\op{Gal}(E^{s}/\Q):  \op{Stab}_{\op{Gal}(E^{s}/\Q)}(\Phi)]
= \# (\op{Gal}(E^{s}/\Q) \cdot \Phi).
\end{align*}
Finally, since $\op{Gal}(E^{s}/\Q) \cdot \Phi \subseteq \Phi(E)$ and $\#\Phi(E) = 2^n$, it follows that $[E_{\Phi}:\Q] \leq 2^n$.
\end{proof}

\begin{corollary}\label{orbit2}
The action of $G^{\CM}$ on $\Phi(E)$ is transitive if and only if $[E_{\Phi}:\Q] = 2^n$ for some CM type $\Phi \in \Phi(E)$.
\end{corollary}

\begin{proof} Since $E^s$ is a CM field, we have
\begin{align*}
\Gal(E^s/\Q) \cdot \Phi = G^{\CM} \cdot \Phi.
\end{align*}
The result now follows from Proposition \ref{orbit} and the fact that $\#\Phi(E) =2^n$.
\end{proof}

\section{CM fields with reflex fields of maximal degree}\label{reflexsection}

Let $F$ be a totally real number field of degree $n$. In the paragraph following \cite[(1.10.1)]{Shi70},
Shimura briefly sketched the construction of a CM extension $E/F$ with reflex fields of maximal degree.
Based on this idea, we undertake an extensive study of the problem of constructing CM fields
with reflex fields of maximal degree and explicitly construct
infinite families of CM extensions $E/F$ with this property. When $n \geq 3$ these CM fields $E$ are non-Galois over $\Q$.

We begin with the following facts and notation which will be needed for the results in this section.

\subsection{Multiplicative congruences, ray class groups, and higher unit groups}

Let $K$ be a number field. For a prime ideal $\mathfrak{P}$ of $K$,
let $v_{\mathfrak{P}}: K \longrightarrow \Z \cup \{ \infty \}$ be the discrete valuation defined by $v_{\mathfrak{P}}(x):= \op{ord}_{\mathfrak{P}}(x)$.
Also, let $K_{\mathfrak{P}}$ be the completion of $K$ with respect to the $\mathfrak{P}$-adic absolute value
$|\cdot|_{\mathfrak{P}}$ induced by the valuation $v_{\mathfrak{P}}$. We denote the ring of $\mathfrak{P}$-adic integers by $
\mathcal{O}_{\mathfrak{P}}$. The unique maximal ideal of $\mathcal{O}_{\mathfrak{P}}$ is $\widehat{\mathfrak{P}} := \mathfrak{P} \mathcal{O}_{\mathfrak{P}}$.

Let $U:= \mathcal{O}_{\mathfrak{P}}^{\times}$ be the group of units of $\mathcal{O}_{\mathfrak{P}}$. For any $n \geq 1$, there
is a subgroup of $U$ defined by
\begin{align*}
U^{(n)}:= 1 + \mathfrak{P}^n \mathcal{O}_{\mathfrak{P}},
\end{align*}
called the \textit{$n$-th higher unit group}. The higher unit groups form a decreasing filtration
\begin{align*}
U \supseteq U^{(1)} \supseteq U^{(2)} \supseteq \cdots \supseteq U^{(n)} \supseteq \cdots .
\end{align*}

For elements $\alpha, \beta \in K^{\times}$, we define the multiplicative congruence by
\begin{align*}
\alpha \overset{\times}{\equiv} \beta \pmod{\mathfrak{P}^n} &\iff \alpha \in \beta (1 + \mathfrak{P}^n \mathcal{O}_{\mathfrak{P}}).
\end{align*}
Thus we see that equivalently
\begin{align*}
\alpha \overset{\times}{\equiv} \beta \pmod{\mathfrak{P}^n} \iff \frac{\alpha}{\beta} \in U^{(n)} \iff v_{\mathfrak{P}}\left( \frac{\alpha}{\beta} - 1 \right) \geq n.
\end{align*}

Let $\mathfrak{m}_0$ be an integral ideal of $K$ and $\mathfrak{m}_{\infty}$ be the formal product
of all the real infinite primes corresponding to the embeddings in $\op{Hom}(K, \R)$. Define the
modulus $\mathfrak{m}:= \mathfrak{m}_0 \mathfrak{m}_{\infty}$. Then we extend the multiplicative congruence by setting
\begin{align*}
\alpha \overset{\times}{\equiv} \beta \pmod{\mathfrak{m}} \iff
\begin{cases}
\alpha \overset{\times}{\equiv} \beta \pmod{\mathfrak{P}^{v_{\mathfrak{P}}(\mathfrak{m}_0) }}  & \textrm{for all $\mathfrak{P} | \mathfrak{m}_0$, and} \\
\dfrac{\sigma(\alpha)}{\sigma(\beta)} >0 &  \textrm{for all $\sigma \in \op{Hom}(K, \R)$}.
\end{cases}
\end{align*}
The multiplicative congruence is indeed multiplicative, i.e., if
\begin{align*}
\alpha_1 \overset{\times}{\equiv} \beta_1 \pmod{\mathfrak{m}} \quad \text{and} \quad \alpha_2 \overset{\times}{\equiv} \beta_2 \pmod{\mathfrak{m}},
\end{align*}
then
\begin{align*}
\alpha_1 \alpha_2 \overset{\times}{\equiv} \beta_1 \beta_2 \pmod{\mathfrak{m}}.
\end{align*}

Let $\mathcal{I}_{K}(\mathfrak{m}_0)$ be the group of all fractional ideals of $K$ that are relatively prime to $\mathfrak{m}_0$. Let
\begin{align*}
K_{\mathfrak{m}, 1}:= \{ x \in K^{\times} \suchthat
\text{$x\mathcal{O}_K$ is relatively prime to $\mathfrak{m}_0$ and } x \overset{\times}{\equiv} 1 \pmod{\mathfrak{m}} \}
\end{align*}
be the ray modulo $\mathfrak{m}$ and $\mathcal{P}_{K}(\mathfrak{m})$ be the subgroup of
$\mathcal{I}_{K}(\mathfrak{m}_0)$ of principal fractional ideals $x\mathcal{O}_K$
generated by elements $x \in K_{\mathfrak{m}, 1}$. Then the \textit{ray class group} of $K$ modulo $\mathfrak{m}$ is the quotient group
\begin{align*}
\mathcal{R}_{K}(\mathfrak{m}):= \mathcal{I}_{K}(\mathfrak{m}_0) / \mathcal{P}_{K}(\mathfrak{m}).
\end{align*}
A coset in the ray class group is called a \textit{ray class} modulo $\mathfrak{m}$.

\subsection{Constructing CM extensions with prescribed ramification}


In the following proposition we explicitly construct infinite families of CM extensions with ``arbitrary'' prescribed 
ramification. This is a variation on \cite[Lemma 1.5]{Shi67}, adapted to the particular
setting we will consider.

\begin{proposition}\label{InfiniteCMFields}
Let $F$ be a totally real number field. Let $p \in \Z$ be a prime number and $m \geq 1$ be a positive integer. Let
$\mathfrak{p}$ be a prime ideal of $F$ lying above $p$. Let $\mathcal{R}$ be a finite set
of prime ideals of $F$ not dividing $pm$. Let $\mathcal{U}_1$ and $\mathcal{U}_2$ be finite sets of prime ideals of $F$ not
dividing $2pm$ such that $\mathcal{R}, \mathcal{U}_1$ and $\mathcal{U}_2$ are pairwise disjoint.
Then there is a set $\mathcal{S}_{\mathcal{R}, \frak{p}}$ of prime ideals of $F$ which is disjoint from
$\mathcal{R} \cup \mathcal{U}_1 \cup \mathcal{U}_2 \cup \{\frak{p}\}$ such that
the following statements are true.
\begin{itemize}
\item[(i)] $\mathcal{S}_{\mathcal{R}, \frak{p}}$ has positive natural density.
\item[(ii)] Each prime ideal $\frak{q} \in \mathcal{S}_{\mathcal{R}, \frak{p}}$ is relatively prime to $pm$.
\item[(iii)] For each prime ideal $\frak{q} \in \mathcal{S}_{\mathcal{R}, \frak{p}}$, there is an element
$\Delta_{\frak{q}} \in \mathcal{O}_F$ with prime factorization
\begin{align*}
\Delta_{\frak{q}}\mathcal{O}_F=\frak{p} \frak{q}\prod_{\frak{r} \in \mathcal{R}}\frak{r}.
\end{align*}
\item[(iv)] The field $E_{\frak{q}}:=F(\sqrt{\Delta_{\frak{q}}})$ is a CM extension of $F$
which is ramified only at the prime ideals of $F$ dividing $\Delta_{\frak{q}}$. Moreover,
each prime ideal in $\mathcal{U}_1$ splits in $E_{\frak{q}}$ and each
prime ideal in $\mathcal{U}_2$ is inert in $E_{\frak{q}}$.
\end{itemize}
\end{proposition}

\begin{remark} Note that if $\mathfrak{q}_1, \mathfrak{q}_2 \in \mathcal{S}_{\mathcal{R}, \mathfrak{p}}$ with $\mathfrak{q}_1 \neq \mathfrak{q}_2$, then the
associated CM extensions $E_{\frak{q}_1}/F$ and $E_{\mathfrak{q}_2}/F$ are
distinct since they are ramified only at the primes in the sets $\mathcal{R} \cup \{\mathfrak{p}, \mathfrak{q}_1\}$ and
$\mathcal{R} \cup \{\mathfrak{p}, \mathfrak{q}_2\}$, respectively.
\end{remark}

In order to prove Proposition \ref{InfiniteCMFields} we will need the following two lemmas.

\begin{lemma}\label{square}
Let $\mathcal{S}$ be a set of prime ideals of $F$ and suppose that $e \in \Z$ satisfies
\begin{align*}
e \geq  2 \op{max}\{ v_{\mathfrak{P}}(2) \suchthat \mathfrak{P} \in \mathcal{S}\} + 1.
\end{align*}
Then for any prime ideal $\mathfrak{P} \in \mathcal{S}$, if $\alpha \in F^{\times}$
and $\alpha \overset{\times}{\equiv} 1 \pmod{\mathfrak{P}^e}$ then $F_{\mathfrak{P}}(\sqrt{\alpha}) = F_{\mathfrak{P}}$.
\end{lemma}

\begin{proof}
Let $\mathfrak{P} \in \mathcal{S}$. Observe that $F_{\mathfrak{P}}(\sqrt{\alpha}) = F_{\mathfrak{P}}$
if and only if $\alpha$ is a perfect square in $F_{\mathfrak{P}}$.
Let $\mathcal{O}_{\mathfrak{P}}$
be the ring of integers of $F_{\mathfrak{P}}$ and $U^{(n)}:= 1 + \mathfrak{P}^{n} \mathcal{O}_{\mathfrak{P}}$ be the
$n$-th higher unit group. Let $v_{\mathfrak{P}}: F \longrightarrow \Z\cup \{ \infty \}$ be the discrete valuation given by
$v_{\mathfrak{P}}(x):= \op{ord}_{\mathfrak{P}}(x)$.
By \cite[Proposition 3-1-6, p. 79]{Wei98},
if $m, i \in \Z$ are integers with $m \geq 1$ and $i \geq v_{\mathfrak{P}}(m) + 1$, then the map
$\phi_{m}: U^{(i)} \longrightarrow U^{(i + v_{\mathfrak{P}}(m))}$ given by $\phi_{m}(x) := x^m$ is
an isomorphism. In particular, when $m = 2$ the surjectivity of the map
$\phi_2$ implies that every element of $U^{(i + v_{\mathfrak{P}}(2) )}$ is a perfect square.

Now, let $i:=\op{max}\{ v_{\mathfrak{P}}(2) \suchthat \mathfrak{P} \in \mathcal{S}\} + 1$. Then because
$i \geq v_{\mathfrak{P}}(2) + 1$, every element of $U^{(i + v_{\mathfrak{P}}(2) )}$ is a perfect square.
On the other hand, if $e \in \Z$ satisfies
\begin{align*}
e \geq  2 \op{max}\{ v_{\mathfrak{P}}(2) \suchthat \mathfrak{P} \in \mathcal{S}\} + 1,
\end{align*}
then $e \geq i + v_{\frak{P}}(2)$. Since the higher unit groups form a decreasing filtration, it follows that
\begin{align*}
U^{(e)} \subseteq  U^{(i + v_{\mathfrak{P}}(2) )}.
\end{align*}
In particular, every element of $U^{(e)}$ is a perfect square.
Finally, since $\alpha \overset{\times}{\equiv} 1 \pmod{\mathfrak{P}^e}$ implies that $\alpha \in U^{(e)}$, the
proof is complete.
\end{proof}

\begin{lemma}\label{QUE}
For each prime ideal $\mathfrak{P}$ of $F$, there exists an element $\alpha_{\mathfrak{P}} \in \mathcal{O}_F$ such that
$F_{\mathfrak{P}}(\sqrt{\alpha_{\mathfrak{P}}})$ is an unramified quadratic extension of $F_{\mathfrak{P}}$.
\end{lemma}

\begin{proof} Up to isomorphism, there is a unique unramified quadratic extension of
$F_{\mathfrak{P}}$, and moreover, it can be obtained by adjoining to $F_{\mathfrak{P}}$ a lifting of a primitive element for the
unique quadratic extension of the finite field
\begin{align*}
\mathcal{O}_{\mathfrak{P}}/\mathfrak{P}\mathcal{O}_{\mathfrak{P}}
\end{align*}
(see e.g. \cite[Theorem 1.2.2, p. 14]{Chi09} or \cite[Proposition 6.54]{KKS11}). Thus, let
\begin{align*}
\widehat{f}(x):= x^2 + \widehat{a_1} x + \widehat{a_0} \in \mathcal{O}_{\mathfrak{P}}/\mathfrak{P}\mathcal{O}_{\mathfrak{P}}[x]
\end{align*}
be an irreducible quadratic polynomial. It is known that the homomorphism
\begin{align*}
\phi: \mathcal{O}_F &\longrightarrow \mathcal{O}_{\mathfrak{P}}/\mathfrak{P}\mathcal{O}_{\mathfrak{P}}\\
\alpha & \longmapsto \alpha + \mathfrak{P}\mathcal{O}_{\mathfrak{P}}
\end{align*}
has kernel $\mathfrak{P}$ and is surjective (see e.g. \cite[Propositions II.4.3 and II.2.4]{Neu99} or \cite[Theorem 11(c)]{FT93}).
Thus every coset of $\mathcal{O}_{\mathfrak{P}}/\mathfrak{P}\mathcal{O}_{\mathfrak{P}}$ has a representative in $\mathcal{O}_F$. Let
$a_0, a_1 \in \mathcal{O}_F$ be such that $\widehat{a_0} = a_0 + \mathfrak{P}\mathcal{O}_{\mathfrak{P}}$ and
$\widehat{a_1} = a_1 + \mathfrak{P}\mathcal{O}_{\mathfrak{P}}$. Then define the polynomial
\begin{align*}
f(x):= x^2 + a_1x + a_0 \in \mathcal{O}_F[x] \subset F_{\mathfrak{P}}[x].
\end{align*}
It follows that $f(x)$ is irreducible in $F_{\mathfrak{P}}[x]$, and moreover by the quadratic formula its roots have the form
\begin{align*}
\frac{-a_1 \pm \sqrt{a_1^2 - 4a_0}}{2}.
\end{align*}
Hence by taking $\alpha_{\mathfrak{P}}:= a_1^2 - 4 a_0 \in \mathcal{O}_F$, we see that
$F_{\mathfrak{P}}(\sqrt{\alpha_{\mathfrak{P}}})$ is an unramified quadratic extension of $F_{\mathfrak{P}}$.
\end{proof}

\textbf{Proof of Proposition \ref{InfiniteCMFields}}. Define the following disjoint sets of prime ideals of $F$.
\begin{align*}
\mathcal{T}_1 &:=\left(\mathcal{U}_1 \cup
\{ \mathfrak{P} \subset \mathcal{O}_F \suchthat \text{$\mathfrak{P}$ divides $pm$}  \}\right) \smallsetminus \{ \mathfrak{p} \},\\
\mathcal{T}_2 &:=\left( \mathcal{U}_2 \cup \{ \mathfrak{P} \subset \mathcal{O}_F \suchthat \text{$\mathfrak{P}$ divides $2$}  \}\right)
\smallsetminus (\mathcal{T}_1 \cup \mathcal{R} \cup \{\mathfrak{p}\}).
\end{align*}
Now, fix an integer $e \in \Z$ satisfying
\begin{align*}
e \geq  2 \op{max}\{ v_{\mathfrak{P}}(2) \suchthat \mathfrak{P} \in \mathcal{T}_1 \cup \mathcal{T}_2\} + 1.
\end{align*}
Then by Lemma \ref{square}, for any prime ideal $\mathfrak{P} \in \mathcal{T}_1 \cup \mathcal{T}_2$, if $\alpha \in F^{\times}$
and $\alpha \overset{\times}{\equiv} 1 \pmod{\mathfrak{P}^e}$ then $F_{\mathfrak{P}}(\sqrt{\alpha}) = F_{\mathfrak{P}}$.
Also, as in Lemma \ref{QUE}, for each prime ideal $\mathfrak{P} \in \mathcal{T}_2$, let $\alpha_{\mathfrak{P}} \in \mathcal{O}_{F}$
be such that $F_{\mathfrak{P}}(\sqrt{\alpha_{\mathfrak{P}}})$ is an unramified quadratic extension of $F_{\mathfrak{P}}$.

Let $\frak{m}_{\infty}$ be the formal product of all the real infinite primes corresponding to the embeddings in $\op{Hom}(F, \R)$.
By an application of the Approximation Theorem (see e.g. \cite[pp. 137-139]{Jan96}),
there exists an element $a \in F^{\times}$ satisfying the following congruences.
\begin{itemize}
\item[(1)] \textrm{$a \overset{\times}{\equiv} -1 \pmod{\mathfrak{m}_{\infty}}$}.
\item[(2)] \textrm{$a \overset{\times}{\equiv} 1 \pmod{\mathfrak{P}^e}$ for every $\mathfrak{P} \in \mathcal{T}_1$}.
\item[(3)] \textrm{$a \overset{\times}{\equiv} \alpha_{\mathfrak{P}} \pmod{\mathfrak{P}^e}$ for every $\mathfrak{P} \in \mathcal{T}_2$}.
\end{itemize}

Define the integral ideal
\begin{align*}
\mathfrak{m}_0:= \prod_{\mathfrak{P} \in \mathcal{T}_1 \cup \mathcal{T}_2} \mathfrak{P}^{e}
\end{align*}
and the modulus $\frak{m}:=\mathfrak{m}_0 \mathfrak{m}_{\infty}$. Let $\mathcal{R}_F(\mathfrak{m})$ be the ray class group modulo
$\mathfrak{m}$. Observe that the fractional ideal
\begin{align}\label{nid}
\mathfrak{n} := a \mathfrak{p}^{-1} \prod \limits_{\mathfrak{r} \in \mathcal{R}} \mathfrak{r}^{-1}
\end{align}
is relatively prime to $\frak{m}_0$. Then we can define the set of prime ideals
\begin{align*}
\mathcal{S}(\mathfrak{n}):=\{\frak{q} \subset \mathcal{O}_F \suchthat \textrm{$\frak{q}$ is a prime ideal and
$[\frak{q}]=\left[\mathfrak{n} \right]$ in $\mathcal{R}_{F}(\mathfrak{m})$}\}.
\end{align*}
Also, define the set of prime ideals
\begin{align*}
\mathcal{S}_{\mathcal{R}, \frak{p}}:=\mathcal{S}(\mathfrak{n}) \smallsetminus (\mathcal{T}_1 \cup \mathcal{T}_2 \cup \mathcal{R} \cup \{\frak{p}\}).
\end{align*}

To prove Proposition \ref{InfiniteCMFields} (i), it is known that the set
$\mathcal{S}(\mathfrak{n})$ has natural density
\begin{align*}
d(\mathcal{S}(\mathfrak{n})):=\lim_{X \rightarrow \infty}\frac{\# \{ \frak{q} \in \mathcal{S}(\mathfrak{n}) \suchthat N_{F/\Q}(\frak{q}) \leq X \}}
{\# \{\frak{q} \subset \mathcal{O}_F \suchthat \textrm{$\frak{q}$ is a prime ideal with $N_{F/\Q}(\frak{q}) \leq X$}\}} =
\frac{1}{\# \mathcal{R}_{F}(\mathfrak{m})}.
\end{align*}
Since the set $\mathcal{T}_1 \cup \mathcal{T}_2 \cup \mathcal{R} \cup \{\frak{p}\}$ is finite, we also have
\begin{align*}
d(\mathcal{S}_{\mathcal{R}, \frak{p}})=\frac{1}{\# \mathcal{R}_{F}(\mathfrak{m})}.
\end{align*}

To prove Proposition \ref{InfiniteCMFields} (ii), note that
if $\frak{q} \in \mathcal{S}_{\mathcal{R}, \frak{p}}$ then $\frak{q} \not \in \mathcal{T}_1 \cup \{\frak{p}\}$, hence
$\frak{q}$ is relatively prime to $pm$.

To prove Proposition \ref{InfiniteCMFields} (iii),
let $\frak{q} \in \mathcal{S}_{\mathcal{R}, \frak{p}}$. Since $[\frak{q}]=[\frak{n}]$ in $\mathcal{R}_{F}(\mathfrak{m})$,
there exists an element $b_{\frak{q}} \in F^{\times}$ such that
\begin{itemize}
\item[(4)] \textrm{$b_{\frak{q}} \overset{\times}{\equiv} 1 \pmod{\mathfrak{m}}$ and $\mathfrak{q} = b_{\frak{q}}\frak{n}$}.
\end{itemize}
By (\ref{nid}) and (4) we have
\begin{align*}
\frak{q}=ab_{\frak{q}}\frak{p}^{-1}\prod_{\frak{r} \in \mathcal{R}}\frak{r}^{-1}.
\end{align*}
Define $\Delta_{\frak{q}}:=ab_{\frak{q}}$. Then
\begin{align}\label{deltafactor}
\Delta_{\frak{q}}\mathcal{O}_F=ab_{\frak{q}}\mathcal{O}_F=\frak{p}\frak{q}\prod_{\frak{r} \in \mathcal{R}}\frak{r}.
\end{align}
Note that this also proves that $\Delta_{\frak{q}} \in \mathcal{O}_F$.

Finally, define the field $E_{\frak{q}}:=F(\sqrt{\Delta_{\frak{q}}})$. Then Proposition \ref{InfiniteCMFields} (iv)
is a consequence of the following lemma.

\begin{lemma}\label{ablemma}
Let $a \in F^{\times}$ be an element satisfying $(1)-(3)$ and $b_{\frak{q}} \in F^{\times}$ be an element satisfying $(4)$. Let
$\Delta_{\frak{q}}:=ab_{\frak{q}} \in \mathcal{O}_F$.
Then the field $E_{\frak{q}}:=F(\sqrt{\Delta_{\frak{q}}})$
is a CM extension of $F$ which satisfies the following properties.
\begin{itemize}
\item[(i)] $E_{\frak{q}}$ is ramified only at the prime ideals of $F$ dividing $\Delta_{\frak{q}}$.
\item[(ii)]  Each prime ideal in $\mathcal{U}_1$ splits in $E_{\frak{q}}$ and each
prime ideal in $\mathcal{U}_2$ is inert in $E_{\frak{q}}$.
\end{itemize}
\end{lemma}

\begin{proof} Since the prime ideals $\frak{p}, \frak{q}$ and $\frak{r} \in \mathcal{R}$ are all distinct,
the identity (\ref{deltafactor}) shows that $\Delta_{\frak{q}}$ is not a perfect square in $F$.
Also, by (1) and (4) we have $\Delta_{\frak{q}}=ab_{\frak{q}} \overset{\times}{\equiv} -1 \pmod{\mathfrak{m}_{\infty}}$,
or equivalently $\Delta_{\frak{q}} \ll 0$. These facts imply
that $E_{\mathfrak{q}}$ is a totally imaginary quadratic extension of $F$, hence a CM field.

Now, since $\Delta_{\frak{q}} \in \mathcal{O}_F$ we have $\sqrt{\Delta_{\frak{q}}} \in \mathcal{O}_{E_{\mathfrak{q}}}$.
Then by (\ref{deltafactor}) we have
\begin{align*}
\mathfrak{p} \mathcal{O}_{E_{\mathfrak{q}}} \mathfrak{q} \mathcal{O}_{E_{\mathfrak{q}}}
\prod_{\frak{r} \in \mathcal{R}}\frak{r}\mathcal{O}_{E_{\mathfrak{q}}} = \Delta_{\frak{q}} \mathcal{O}_{E_{\mathfrak{q}}}
= \left( \sqrt{\Delta_{\frak{q}}} \mathcal{O}_{E_{\mathfrak{q}}} \right)^2.
\end{align*}
This implies that each of the prime ideals of $F$ dividing $\Delta_{\mathfrak{q}}$ is ramified in $E_{\mathfrak{q}}$. Thus, to prove (i),
it remains to show that if $\mathfrak{P}$ is a prime ideal of $F$ not dividing $\Delta_{\mathfrak{q}}$, then $\mathfrak{P}$ is
unramified in $E_{\mathfrak{q}}$.

It is known that if $K$ is a number field and $\alpha$ is a root of the polynomial
$$f(x):= x^2 - \beta \in \mathcal{O}_K[x],$$
then any nonzero prime ideal $\mathfrak{P}$ of $K$ such that $\mathfrak{P}$
does not divide $2\beta$ is unramified in $L:=K(\alpha)$ (see e.g. \cite[Example 6.40, p. 59]{KKS11}). Therefore if $\mathfrak{P}$ is a prime ideal of
$F$ such that $\mathfrak{P}$ does not divide $2\Delta_{\frak{q}}$, then $\mathfrak{P}$ is unramified in $E_{\mathfrak{q}}$.
Thus it suffices to prove that if $\mathfrak{P}$ is a prime ideal of $F$ such that
$\mathfrak{P}$ divides $2$ and $\mathfrak{P}$ does not divide $\Delta_{\mathfrak{q}}$,
then $\mathfrak{P}$ is unramified in $E_{\frak{q}}$.

By (\ref{deltafactor}) we know that the prime ideals of $F$ that divide
$\Delta_{\mathfrak{q}}$ are the primes in the set $\mathcal{R} \cup \{\mathfrak{p}, \mathfrak{q}\}$.
Therefore from the definitions of $\mathcal{T}_1$ and $\mathcal{T}_2$ we see that the
set of prime ideals $\frak{P}$ of $F$ such that $\mathfrak{P}$ divides $2$ and
$\mathfrak{P} \not \in \mathcal{R} \cup \{ \mathfrak{p}, \mathfrak{q} \}$
is a subset of $\mathcal{T}_1 \cup \mathcal{T}_2$. Hence, in the remainder of the proof
we will show that the prime ideals in $\mathcal{T}_1 \cup \mathcal{T}_2$ are unramified in $E_{\mathfrak{q}}$. In fact, we will
show that the prime ideals in $\mathcal{T}_1$ split in $E_{\frak{q}}$ and the prime ideals in $\mathcal{T}_2$ remain inert in $E_{\frak{q}}$.
Since $\mathcal{U}_1 \subset \mathcal{T}_1$ and $\mathcal{U}_2 \subset \mathcal{T}_2$, this will also complete the proof of (ii).

Thus let $\mathfrak{P} \in \mathcal{T}_1 \cup \mathcal{T}_2$
and let $\mathfrak{Q}$ be a prime ideal of $E_{\frak{q}}$ lying above $\mathfrak{P}$. Also, let $\widehat{\mathfrak{P}}$ and $\widehat{\mathfrak{Q}}$
denote the unique prime ideals in the completions $F_{\mathfrak{P}}$ and $E_{\frak{q},\mathfrak{Q}}$, respectively. It is known
that the ramification indices are the same, i.e., we have
\begin{align*}
e(\mathfrak{Q}|\mathfrak{P}) = e(\widehat{\mathfrak{Q}} | \widehat{\mathfrak{P}}).
\end{align*}
We will show that $e(\mathfrak{Q}|\mathfrak{P}) = e(\widehat{\mathfrak{Q}} | \widehat{\mathfrak{P}})=1$.

The minimal polynomial of the primitive element $\sqrt{\Delta_{\frak{q}}}$ of $E_{\frak{q}}$ over $F$ is
$$m_{\Delta_{\frak{q}}}(x):= x^2 - \Delta_{\frak{q}} \in \mathcal{O}_F[x].$$ It is known
that the primes of $E_{\frak{q}}$ lying above $\mathfrak{P}$ are in one to one correspondence with the irreducible factors of $m_{\Delta_{\frak{q}}}(x)$
when considered as a polynomial in $F_{\mathfrak{P}}[x]$ and moreover, if $\mathfrak{Q}$ corresponds to an irreducible factor $m_i(x)$,
then the completion of $E_{\mathfrak{q}}$ at $\mathfrak{Q}$ satisfies
\begin{align*}
E_{\mathfrak{q}, \mathfrak{Q}} \cong \frac{F_{\mathfrak{P}}[x]}{\langle m_{i}(x) \rangle}
\end{align*}
(see for example \cite[Theorem II.6.1, p. 115]{Jan96}).

We have two cases to consider.

\vspace{0.05in}

\textbf{Case 1:} $\mathfrak{P} \in \mathcal{T}_1$. In this case the congruences $(2)$ and $(4)$ satisfied by $a$ and $b_{\frak{q}}$
imply that $\Delta_{\mathfrak{q}} = ab_{\frak{q}} \overset{\times}{\equiv} 1 \pmod{\mathfrak{P}^e}$. Hence by Lemma \ref{square} we conclude that
$F_{\mathfrak{P}}(\sqrt{\Delta_{\frak{q}}}) =  F_{\mathfrak{P}}$.
This implies that there is an element $c \in F_{\mathfrak{P}}$ such that $\Delta_{\frak{q}} = c^2$.
Therefore the polynomial $m_{\Delta_{\frak{q}}}(x)$ factors as $$m_{\Delta_{\frak{q}}}(x) = x^2 - c^2 = (x - c)(x + c)$$
in $F_{\mathfrak{P}}[x]$. Since the prime ideals of $E_{\mathfrak{q}}$ lying over $\mathfrak{P}$ are in one to one correspondence
with the irreducible factors $x - c$ and $x + c$,
and since $E_{\mathfrak{q}} /F$ is a quadratic extension, we see that $\mathfrak{P}$ splits in $E_{\mathfrak{q}}$,
so that $e(\mathfrak{Q} | \mathfrak{P}) = 1$.
\vspace{0.05in}

\textbf{Case 2:} $\mathfrak{P} \in \mathcal{T}_2$. In this case the congruences $(3)$ and $(4)$ satisfied by
$a$ and $b_{\frak{q}}$ imply that
$\Delta_{\mathfrak{q}} = ab_{\frak{q}} \overset{\times}{\equiv} \alpha_{\mathfrak{P}} \pmod{\mathfrak{P}^e}$, or equivalently,
\begin{align*}
\frac{\Delta_{\frak{q}}}{\alpha_{\frak{P}}}  \overset{\times}{\equiv} 1 \pmod{\mathfrak{P}^e}.
\end{align*}
Hence by Lemma \ref{square}, we have $\Delta_{\frak{q}} = c^2 \alpha_{\frak{P}}$ for some $c \in F_{\frak{P}}^{\times}$, which implies that
$$F_{\mathfrak{P}}(\sqrt{\Delta_{\frak{q}}}) =  F_{\mathfrak{P}}(\sqrt{\alpha_{\mathfrak{P}}}).$$ On the other hand,
by Lemma \ref{QUE} we have that $F_{\mathfrak{P}}(\sqrt{\alpha_{\mathfrak{P}}})$ is an
unramified quadratic extension of $F_{\mathfrak{P}}$.
It follows that $m_{\Delta_{\frak{q}}}(x)$ is irreducible in $F_{\mathfrak{P}}[x]$.
Thus $\frak{Q}$ is the only prime ideal
of $E_{\mathfrak{q}}$ lying above $\mathfrak{P}$ and it corresponds to $m_{\Delta_{\frak{q}}}(x)= x^2 - \Delta_{\frak{q}}$. Therefore we have
\begin{align*}
E_{\mathfrak{q}, \mathfrak{Q}} \cong \frac{F_{\mathfrak{P}}[x]}{\langle m_{\Delta_{\frak{q}}}(x) \rangle} \cong F_{\mathfrak{P}}(\sqrt{\Delta_{\frak{q}}}).
\end{align*}
This implies that $E_{\mathfrak{q}, \mathfrak{Q}}$ is an unramified quadratic extension of
$F_{\mathfrak{P}}$, hence $e(\widehat{\mathfrak{Q}} | \widehat{\mathfrak{P}})=1$. Therefore $e(\mathfrak{Q} | \mathfrak{P}) = 1$,
and in particular $\mathfrak{P}$ remains inert in $E_{\mathfrak{q}}$.
\end{proof}

This completes the proof of Proposition \ref{InfiniteCMFields}. \qed

\subsection{Constructing non-abelian CM fields with reflex fields of maximal degree}

In the following theorem we prove that if $E/F$ is a CM extension satisfying 
a certain mild ramification condition, then the reflex fields $E_{\Phi}$ have maximal degree, and moreover, if $n \geq 3$ then $E/\Q$ is 
non-Galois.

\begin{theorem}\label{Shimura}
Let $F$ be a totally real number field of degree $n$.
Let $p \in \Z$ be a prime number that splits in the Galois closure $F^s$ and
let $\mathfrak{p}$ be a prime ideal of $F$ lying above $p$.
Let $d_{F^s}$ be the discriminant of $F^s$ and
$\mathcal{L}$ be a finite set of prime ideals of $F$ not dividing $p d_{F^s}$. Then if $E/F$ is a
CM extension which is ramified only at the prime ideals of $F$ in the set $\mathcal{L} \cup \{ \mathfrak{p}\}$,
the reflex degree $[E_{\Phi}: \Q] = 2^n$ for every CM type $\Phi \in \Phi(E)$.
Moreover, if $n \geq 3$ then $E/\Q$ is non-Galois (hence non-abelian).
\end{theorem}

We will prove Theorem \ref{Shimura} using a sequence of five lemmas which are now proved in succession. 

\begin{lemma}\label{compositumid} Let $F$ be a totally real number field of degree $n$.  Let $E/F$ be a CM extension
and $\Phi = \{\sigma_1, \dots, \sigma_n  \} \in \Phi(E)$ be a CM type for $E$. Let $E_{\Phi}$ be the reflex field
of the CM pair $(E,\Phi)$. Then
\begin{align*}
E_\Phi F^s = E^{\sigma_1} \cdots E^{\sigma_n}=E^s.
\end{align*}
\end{lemma}

\begin{proof} We first prove that
\begin{align*}
E^{\sigma_1} \cdots E^{\sigma_n} \subseteq E_\Phi F^s.
\end{align*}
It suffices to show that $\sigma_j(c) \in E_\Phi F^s$ for all $c \in E$ and $j=1, \ldots, n$.
Let $\alpha_1, \dots, \alpha_n$ be an integral basis for $F$. By Proposition \ref{ReflexField} (ii), the reflex field of
the CM pair $(E, \Phi)$ is given by
\begin{align*}
E_{\Phi} = \Q \left( \{ \op{Tr}_{\Phi}(a) \suchthat a \in E \} \right),
\end{align*}
where $ \op{Tr}_{\Phi}(a) = \sum \limits_{j = 1}^{n} \sigma_{j}(a)$. Then for all $c \in E$ and $i = 1, \dots, n$, we have
\begin{align*}
\op{Tr}_{\Phi}(c \alpha_i) = \sum \limits_{j = 1}^{n} \sigma_{j}(c \alpha_i) = \sum_{j=1}^n \sigma_j(\alpha_i) \sigma_j(c) \in E_{\Phi}.
\end{align*}
In particular, there are elements $\beta_i \in E_{\Phi}$ such that
\begin{align*}
\sum_{j=1}^n \sigma_j(\alpha_i) \sigma_j(c) = \beta_i
\end{align*}
for $i =1, \dots, n$. This yields the linear system
\begin{align*}
\begin{bmatrix}
\sigma_1(\alpha_1) & \sigma_2(\alpha_1) & \cdots & \sigma_n(\alpha_1) \\
\sigma_1(\alpha_2) & \sigma_2(\alpha_2) & \cdots & \sigma_n(\alpha_2) \\
\vdots & \vdots & \ddots & \vdots \\
\sigma_1(\alpha_n) & \sigma_2(\alpha_n) & \cdots & \sigma_n(\alpha_n) \\
\end{bmatrix}
\begin{bmatrix}
\sigma_1(c)\\
\sigma_2(c) \\
\vdots \\
\sigma_n(c)\\
\end{bmatrix}
=
\begin{bmatrix}
\beta_1\\
\beta_2 \\
\vdots \\
\beta_n\\
\end{bmatrix}.
\end{align*}
The matrix $[\sigma_{j}(\alpha_i)] \in M^{n \times n}(F^s)$, and
it is invertible since $\det{[\sigma_{j}(\alpha_i)]^2} = d_F \neq 0$. It follows
from Cramer's rule that
\begin{align}\label{sigmaid}
\sigma_j(c) =
\frac{
\det{
\begin{bmatrix}
\sigma_1(\alpha_1) & \cdots & \sigma_{j-1}(\alpha_1) & \beta_1 & \sigma_{j+1}(\alpha_1) & \cdots  & \sigma_n(\alpha_1) \\
\vdots &  & \vdots & \vdots &\vdots &   & \vdots \\
\sigma_1(\alpha_n) & \cdots & \sigma_{j-1}(\alpha_n) & \beta_n & \sigma_{j+1}(\alpha_n) & \cdots  & \sigma_n(\alpha_n) \\
\end{bmatrix}}
}{
\det{
\begin{bmatrix}
\sigma_1(\alpha_1) & \cdots & \sigma_n(\alpha_1) \\
\vdots & \ddots & \vdots \\
\sigma_1(\alpha_n) & \cdots & \sigma_n(\alpha_n) \\
\end{bmatrix}}
}
\end{align}
for $j = 1, \dots, n$. Since $\sigma_j(\alpha_i) \in F^s$ and $\beta_i \in E_{\Phi}$ for $i, j = 1, \dots, n$,
the denominator in (\ref{sigmaid}) is in $F^s$ and the numerator is in
$ E_{\Phi} F^s $. Therefore, $\sigma_j(c) \in  E_{\Phi} F^s$
for all $c \in E$ and $j = 1, \dots, n$, which implies that
\begin{align*}
E^{\sigma_1} \cdots E^{\sigma_n} \subseteq  E_{\Phi} F^s.
\end{align*}
On the other hand, since the compositum of all the conjugate fields of a number field is equal
to its Galois closure, and since complex conjugation is an automorphism of $E$ that commutes with every embedding (see \cite[Proposition 5.11]{Shi94}),
we have $E^{\sigma_1} \cdots E^{\sigma_n} = E^s$. Therefore, since
$ E_{\Phi} F^s  \subseteq E^s$, we conclude that
\begin{align*}
 E_{\Phi} F^s  = E^{\sigma_1} \cdots E^{\sigma_n} = E^s.
\end{align*}
\end{proof}

\begin{lemma}\label{coprime} Let $F$ be a totally real number field of degree $n$.
Let $p \in \Z$ be a prime number that splits in the Galois closure $F^s$ and
let $\mathfrak{p}$ be a prime ideal of $F$ lying above $p$.  Let $E/F$ be a CM extension
and $\Phi = \{\sigma_1, \dots, \sigma_n  \} \in \Phi(E)$ be a CM type for $E$.
Then the ideals $\mathfrak{p}^{\sigma_1} \mathcal{O}_{F^s}, \ldots, \mathfrak{p}^{\sigma_n}\mathcal{O}_{F^s}$
are pairwise relatively prime.
\end{lemma}

\begin{proof}
Suppose that $\mathfrak{P}$ is a prime of $F^s$ lying above $\mathfrak{p}$.
Thus $\mathfrak{P}$ also lies above $p \in \Z$. Since $F^s/\Q$ is Galois, we have
\begin{align}\label{pdecomp}
p \mathcal{O}_{F^s} = \prod_{\sigma \in \Gal(F^s/\Q)} \sigma(\mathfrak{P}).
\end{align}
Moreover, since $p$ splits in $F^s$, then $\sigma(\mathfrak{P}) \neq \tau(\mathfrak{P})$
for any $\sigma, \tau \in \Gal(F^s/\Q)$ with $\sigma \neq \tau$.

Now, let $G_i:= \Gal(F^s / F^{\sigma_i})$ for $i = 1, \dots, n$. For each $i=1, \ldots , n$ we have that $\mathfrak{p}^{\sigma_i}$
is a prime ideal of $F^{\sigma_i}$ lying above $p$. Hence $\mathfrak{p}^{\sigma_i}$ also splits in $F^s$.
Let $\widetilde{\sigma_i} \in  \Gal(F^s/\Q)$ be an extension of the embedding
$\sigma_i|_F: F \hookrightarrow F^s$, i.e. $\widetilde{\sigma_i}|_F = \sigma_i|_F$.
It follows that $\widetilde{\sigma_i}(\mathfrak{P})$ lies above $\mathfrak{p}^{\sigma_i}$, and since the extension $F^s/F^{\sigma_i}$ is Galois, we have
\begin{align*}
\mathfrak{p}^{\sigma_i} \mathcal{O}_{F^s} = \prod_{\sigma \in G_i} \sigma( \widetilde{\sigma_i}(\mathfrak{P})) = \prod_{\tau \in G_i \widetilde{\sigma_i}} \tau(\mathfrak{P}).
\end{align*}
Since $G_i\widetilde{\sigma_i} \subseteq \Gal(F^s/\Q)$ for $i=1, \ldots, n$ and
$\sigma(\mathfrak{P}) \neq \tau(\mathfrak{P})$ for any $\sigma, \tau \in \Gal(F^s/\Q)$ with $\sigma \neq \tau$,
it suffices to prove that  $G_i \widetilde{\sigma_i} \cap G_j \widetilde{\sigma_j} = \varnothing$ for $i \neq j$.

Suppose by contradiction that there exists an element $\sigma \in G_i \widetilde{\sigma_i} \cap G_j \widetilde{\sigma_j}$ for $i \neq j$.
Then there are elements $\tau_i \in G_i$ and $\tau_j \in G_j$ such that
$\sigma = \tau_i \widetilde{\sigma_i}$ and $\sigma = \tau_j \widetilde{\sigma_j}$.
Since $\{\sigma_1, \ldots, \sigma_n\}$ is a CM type for $E$, then $\op{Hom}(F, \overline{\Q}) = \{ \sigma_1|_{F}, \ldots, \sigma_n|_{F}\}$
and therefore the embeddings $\sigma_i|_F$ and $\sigma_j|_F$ are different. Hence there
is an element $x\in F$ such that $\sigma_i(x) \neq \sigma_j(x)$. Since
$\sigma_i(x) \in F^{\sigma_i}$ and $\tau_i|_{F^{\sigma_i}} = \textrm{id}_{F^{\sigma_i}}$, it follows that
\begin{align*}
\sigma_i(x) = \widetilde{\sigma_i}(x) = \tau_i( \widetilde{\sigma_i}(x)) = \tau_j(\widetilde{\sigma_j}(x)) = \widetilde{\sigma_j}(x) = \sigma_j(x),
\end{align*}
which is a contradiction. Thus for $i \neq j$, we have $G_i \widetilde{\sigma_i} \cap G_j \widetilde{\sigma_j} = \varnothing$, which
shows that the ideals $\mathfrak{p}^{\sigma_i} \mathcal{O}_F^s$ and $\mathfrak{p}^{\sigma_i} \mathcal{O}_F^s$ are relatively prime.
\end{proof}

For an extension of number fields $L/K$, let $\frak{D}(L/K)$ be the relative different, which is an integral ideal of $L$.

\begin{lemma}\label{differentid} Let $F$ be a totally real number field.
Let $p \in \Z$ be a prime number that splits in the Galois closure $F^s$ and
let $\mathfrak{p}$ be a prime ideal of $F$ lying above $p$. Let $\mathcal{L}$ be a finite set of prime ideals of $F$ not dividing
$p d_{F^s}$. Let $E/F$ be a
CM extension which is ramified only at the prime ideals of $F$ in the set $\mathcal{L} \cup \{ \mathfrak{p}\}$. Then
\begin{align*}
\frak{D}(E F^s/F^s) = \frak{D}(E/F)\mathcal{O}_{EF^s}.
\end{align*}
\end{lemma}

\begin{proof} We have the following towers of fields.
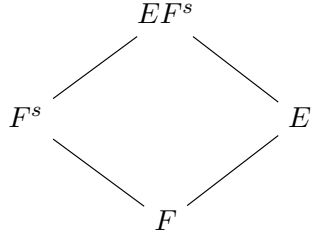
\begin{equation*}
\begin{tikzcd}
\, & EF^s \arrow[-]{dl} \arrow[-]{dr} & \\
F^s \arrow[-]{dr} & & E \arrow[-]{dl}\\
 & F &
\end{tikzcd}
\end{equation*}
Since the relative different is multiplicative in towers, we have the identity
\begin{align}\label{DifferentProd}
\mathfrak{D}(EF^s/F^s) \mathfrak{D} (F^s/F) = \mathfrak{D}(EF^s/E) \mathfrak{D}(E/F).
\end{align}
We will prove that $\mathfrak{D}(EF^s/F^s)$ and $\mathfrak{D}(EF^s/E)$ are relatively prime,
and that $\mathfrak{D} (F^s/F)$ and $ \mathfrak{D}(E/F)$ are relatively prime as ideals in $\mathcal{O}_{EF^s}$. Then
(\ref{DifferentProd}) would imply that
\begin{align}\label{DifferentEq}
\mathfrak{D}(EF^s/F^s) = \mathfrak{D}(E/F) \mathcal{O}_{EF^s}.
\end{align}

First, we prove that $\mathfrak{D} (F^s/F)$ and $\mathfrak{D}(E/F)$ are relatively prime as ideals in $\mathcal{O}_{EF^s}$. To see this, suppose by
contradiction that there is a prime ideal $\mathfrak{P}_{EF^s}$ of $\mathcal{O}_{EF^s}$ such that
\begin{align*}
\mathfrak{P}_{EF^s} | \mathfrak{D}(F^s/F) \mathcal{O}_{EF^s} \quad \text{and} \quad \mathfrak{P}_{EF^s} | \mathfrak{D}(E/F) \mathcal{O}_{EF^s}.
\end{align*}
Define the prime ideals $\mathfrak{P}_F:= \mathfrak{P}_{EF^s} \cap \mathcal{O}_F$,
$\mathfrak{P}_{F^s}:= \mathfrak{P}_{EF^s} \cap \mathcal{O}_{F^s}$ and $\mathfrak{P}_{E}:= \mathfrak{P}_{EF^s} \cap \mathcal{O}_{E}$.
Then $\mathfrak{P}_{F^s}$ is a prime in $F^s$ that divides $\mathfrak{D}(F^s/F)$ and
hence $\mathfrak{P}_F=\mathfrak{P}_{F^s} \cap \mathcal{O}_F$
ramifies in the extension $F^s / F$. Similarly, $\mathfrak{P}_{E}$
is a prime ideal of $E$ that divides $\mathfrak{D}(E/F)$ and hence $\mathfrak{P}_F=\mathfrak{P}_E \cap \mathcal{O}_F$ ramifies in the extension $E/F$.

Now, since the only primes of $F$ that ramify in $E$ are the primes in the set $\mathcal{L} \cup \{ \mathfrak{p} \}$, it follows that
$\mathfrak{P}_F = \mathfrak{p}$ or $\mathfrak{P}_F = \mathfrak{l}$ for some $\mathfrak{l} \in \mathcal{L}$. We will see now that each of these two
possibilities leads to a contradiction. If $\mathfrak{P}_F = \mathfrak{p}$, then $\mathfrak{p}$ would be
ramified in $F^s$. But this would contradict the fact that $p$ splits in $F^s$, since $\mathfrak{p}$ lies above $p$.
On the other hand, if $\mathfrak{P}_F = \mathfrak{l}$ for some $\mathfrak{l} \in \mathcal{L}$, then $\mathfrak{l}$ would be ramified in $F^s$. Hence the
rational prime $\ell$ such that $\ell \Z = \mathfrak{l} \cap \Z$
would be ramified in $F^s$, which implies that  $\ell$ divides $d_{F^s}$ and hence that $\mathfrak{l}$ divides $d_{F^s}$. However,
this is a contradiction since we assumed that the prime ideals in the set $\mathcal{L}$ do not divide $pd_{F^s}$.
Thus $\mathfrak{D} (F^s/F)$ and $\mathfrak{D}(E/F)$ are relatively prime as ideals in $\mathcal{O}_{EF^s}$, as claimed.

Next, we prove that $\mathfrak{D}(EF^s/F^s)$ and $\mathfrak{D}(EF^s/E)$ are relatively prime.  By \cite[Section 13.2, \textbf{U}. (1), p. 253]{Rib01}, we have that
$\mathfrak{D}(EF^s/F^s)|\mathfrak{D}(E/F)\mathcal{O}_{EF^s}$ and $\mathfrak{D}(EF^s/E)|\mathfrak{D} (F^s/F)\mathcal{O}_{EF^s}$. Since we proved that
$\mathfrak{D}(E/F)\mathcal{O}_{EF^s}$ and $\mathfrak{D} (F^s/F)\mathcal{O}_{EF^s}$ are relatively prime, it follows
that $\mathfrak{D}(EF^s/F^s)$ and $\mathfrak{D}(EF^s/E)$ are relatively prime. This completes the proof of the lemma.

\end{proof}

For an extension of number fields $L/K$, let $\frak{d}(L/K)$ be the relative discriminant, which is an integral ideal of $K$.

\begin{lemma}\label{relative} Let $F$ be a totally real number field of degree $n$.
Let $p \in \Z$ be a prime number that splits in the Galois closure $F^s$ and
let $\mathfrak{p}$ be a prime ideal of $F$ lying above $p$. Let $\mathcal{L}$ be a finite set of prime ideals of $F$ not dividing
$p d_{F^s}$. Let $E/F$ be a
CM extension which is ramified only at the prime ideals of $F$ in the set $\mathcal{L} \cup \{ \mathfrak{p}\}$. Let
$\Phi = \{\sigma_1, \dots, \sigma_n  \} \in \Phi(E)$ be a CM type for $E$. Then the
relative discriminant $\frak{d}(E^{\sigma_i}F^s/F^s)$ is divisible by $\frak{p}^{\sigma_i}\mathcal{O}_{F^s}$, but relatively prime
to $\frak{p}^{\sigma_j}\mathcal{O}_{F^s}$ for $j \neq i$.
\end{lemma}

\begin{proof} We first prove the following claim.
\vspace{0.05in}

\textbf{Claim}. \textit{The relative different $\frak{D}(E/F)\mathcal{O}_{EF^s}$ is divisible by the primes of $EF^s$ lying
above the primes in the set $\mathcal{L} \cup \{\frak{p}\}$, and by no other primes of $EF^s$.}
\vspace{0.05in}

\textbf{Proof of the Claim}. Since the primes in the set $\mathcal{L} \cup \{\frak{p}\}$
are the only primes of $F$ which ramify in $E$, we have
\begin{align*}
\frak{d}(E/F)=\frak{p}^{a_{\frak{p}}}\prod_{\frak{l} \in \mathcal{L}}\frak{l}^{a_{\frak{l}}}
\end{align*}
for some positive integers $a_{\frak{p}}$ and $a_{\frak{l}}$ for $\frak{l} \in \mathcal{L}$. Moreover,
since $E/F$ is quadratic, there is a prime ideal $\frak{P}$ of $E$ such that $\frak{p}\mathcal{O}_E=\frak{P}^2$
and a set of prime ideals $\{\frak{P}_{\frak{l}} \suchthat \frak{l} \in \mathcal{L}\}$ of $E$
such that $\frak{l}\mathcal{O}_E=\frak{P}_{\frak{l}}^2$ for each $\frak{l} \in \mathcal{L}$.
Therefore, the relative different factors as
\begin{align*}
\frak{D}(E/F)=\frak{P}^{u_{\frak{P}}}\prod_{\frak{l} \in \mathcal{L}}\frak{P}_{\frak{l}}^{u_{\frak{l}}}
\end{align*}
for some positive integers $u_{\frak{P}}$ and $u_{\frak{l}}$ for $\frak{l} \in \mathcal{L}$.
By extending the relative different to $EF^s$, we see that $\frak{D}(E/F)\mathcal{O}_{EF^s}$ is divisible by the primes of
$EF^s$ lying above the primes in the set $\{\frak{P}\} \cup \{\frak{P}_{\frak{l}} \suchthat \frak{l} \in \mathcal{L}\}$,
and by no other primes of $EF^s$. It follows that
$\frak{D}(E/F)\mathcal{O}_{EF^s}$ is divisible by the primes of $EF^s$ lying
above the primes in the set $\mathcal{L} \cup \{\frak{p}\}$, and by no other primes of $EF^s$. This completes the proof of the claim. \qed
\vspace{0.05in}

Now, since $p$ splits in $F^s$, then $\frak{p}$ splits in $F^s$. Hence
\begin{align}\label{psplit}
\frak{p}\mathcal{O}_{F^s}=\frak{p}_1 \cdots \frak{p}_{g},
\end{align}
where $g=[F^s:F]$ and the $\frak{p}_k$ are distinct prime ideals of $F^s$. For $k=1, \ldots, g$, we have
\begin{align*}
\frak{p}_k\mathcal{O}_{EF^s}=\prod_{t=1}^{a_k}\frak{P}_{k,t}^{b_{k,t}}
\end{align*}
for distinct prime ideals $\frak{P}_{k,t}$ of $EF^s$ and some positive integers $a_k$ and $b_{k,t}$. Thus
\begin{align*}
\frak{p}\mathcal{O}_{EF^s}=\prod_{k=1}^{g}\prod_{t=1}^{a_k}\frak{P}_{k,t}^{b_{k,t}}.
\end{align*}
The prime ideals $\frak{P}_{k,t}$ are the primes of $EF^s$ lying above $\frak{p}$. Hence by the Claim, we see that
$\frak{P}_{k,t}$ divides $\frak{D}(E/F)\mathcal{O}_{EF^s}$. However, by Lemma \ref{differentid},
\begin{align}\label{claimeq}
\frak{D}(E F^s/F^s) = \frak{D}(E/F)\mathcal{O}_{EF^s},
\end{align}
hence $\frak{P}_{k,t}$ divides $\frak{D}(EF^s/F^s)$. It follows that $\frak{p}_k = \frak{P}_{k,t} \cap \mathcal{O}_{F^s}$ divides $\frak{d}(EF^s/F^s)$
for $k=1, \ldots, g$.

Similarly, for a prime ideal $\frak{l} \in \mathcal{L}$, starting with the factorization
\begin{align*}
\frak{l}\mathcal{O}_{F^s}=\frak{P}_{\frak{l},1}^{r(\frak{l},1)}\cdots \frak{P}_{\frak{l},g_{\frak{l}}}^{r(\frak{l},g_{\frak{l}})}
\end{align*}
for distinct prime ideals $\frak{P}_{\frak{l},k}$ of $F^s$ and some positive integers
$r(\frak{l},k)$ for $k=1, \ldots, g_{\frak{l}}$, an analogous argument shows that $\frak{P}_{\frak{l},k}$ divides $\frak{d}(EF^s/F^s)$
for $k=1, \ldots, g_{\frak{l}}$.

By the Claim and the identity (\ref{claimeq}), the primes of $EF^s$ lying above the primes in the set
$$\{\frak{p}_k \suchthat k=1, \ldots, g\} \cup \bigcup_{\frak{l} \in \mathcal{L}}\{\frak{P}_{\frak{l},k}\suchthat k=1, \ldots, g_{\frak{l}}\}$$
are the only primes of $EF^s$ which divide $\frak{D}(EF^s/F^s)$. Hence, the relative discriminant factors as
\begin{align}\label{discfactor}
\frak{d}(EF^s/F^s)=\frak{p}_1^{c_1} \cdots \frak{p}_{g}^{c_{g}}\prod_{\frak{l} \in \mathcal{L}}\prod_{k=1}^{g_{\frak{l}}}\frak{P}_{\frak{l},k}^{d(\frak{l},k)}
\end{align}
for some positive integers $c_1, \ldots, c_{g}$ and $d(\frak{l},k)$ for $\frak{l} \in \mathcal{L}$ and
$k=1, \ldots, g_{\frak{l}}$.

Now, for each embedding $\sigma_i \in \Phi$, let $\widetilde{\sigma_i}$ be an extension of $\sigma_i$ to $EF^s$. Then
since $F^s/\Q$ is Galois, we have $\widetilde{\sigma_i}(F^s) = F^s$, and therefore conjugating by $\widetilde{\sigma_i}$
in equation (\ref{discfactor}) yields
\begin{align}\label{discfactor2}
\frak{d}(E^{\sigma_i}F^s/F^s)=\widetilde{\sigma_i}(\frak{p}_1)^{c_1} \cdots \widetilde{\sigma_i}(\frak{p}_{g})^{c_{g}}
\prod_{\frak{l} \in \mathcal{L}}\prod_{k=1}^{g_{\frak{l}}}\widetilde{\sigma_i}\left(\frak{P}_{\frak{l},k}\right)^{d(\frak{l},k)}.
\end{align}
It follows from (\ref{psplit}) and (\ref{discfactor2}) that
\begin{align}\label{psigmafactor}
\frak{p}^{\sigma_i}\mathcal{O}_{F^s}=\widetilde{\sigma_i}(\frak{p}_1) \cdots \widetilde{\sigma_i}(\frak{p}_{g})
\end{align}
divides $\frak{d}(E^{\sigma_i}F^s/F^s)$. This proves the first part of the lemma.

It remains to prove that $\frak{p}^{\sigma_j}\mathcal{O}_{F^s}$ is relatively prime to  $\frak{d}(E^{\sigma_i}F^s/F^s)$
for $j \neq i$. By Lemma \ref{coprime}, the ideal $\frak{p}^{\sigma_j}\mathcal{O}_{F^s}$ is relatively prime to
$\frak{p}^{\sigma_i}\mathcal{O}_{F^s}$ for $j \neq i$, and hence relatively prime to
$\widetilde{\sigma_i}(\frak{p}_1)^{c_1} \cdots \widetilde{\sigma_i}(\frak{p}_{g})^{c_g}$ by equation (\ref{psigmafactor}). Thus, by
(\ref{discfactor2}) it suffices to prove that $\frak{p}^{\sigma_j}\mathcal{O}_{F^s}$ is relatively prime to
$\widetilde{\sigma_i}\left(\frak{P}_{\frak{l},k}\right)$ for each $\frak{l} \in \mathcal{L}$ and $k=1, \ldots, g_{\frak{l}}$. To see this,
recall that the prime ideal $\mathfrak{l}$ does not divide $pd_{F^s}$,
hence $\mathfrak{l}$ lies above a rational prime $\ell \in \Z$ with $\ell \neq p$.
Since $\widetilde{\sigma_i}\left(\frak{P}_{\frak{l},k}\right)$ lies above $\ell$,
and each of the prime factors of $\frak{p}^{\sigma_j} \mathcal{O}_{F^s}$ lies above $p$, it follows that $\frak{p}^{\sigma_j} \mathcal{O}_{F^s}$
must be relatively prime to $\widetilde{\sigma_i}\left(\frak{P}_{\frak{l},k}\right)$. This proves
the second part of the lemma.
\end{proof}

\begin{lemma}\label{Degree} Let $F$ be a totally real number field of degree $n$.
Let $p \in \Z$ be a prime number that splits in the Galois closure $F^s$ and
let $\mathfrak{p}$ be a prime ideal of $F$ lying above $p$. Let $\mathcal{L}$ be a finite set of prime ideals of $F$ not dividing
$p d_{F^s}$. Let $E/F$ be a
CM extension which is ramified only at the prime ideals of $F$ in the set $\mathcal{L} \cup \{ \mathfrak{p}\}$. Let
$\Phi = \{\sigma_1, \dots, \sigma_n  \} \in \Phi(E)$ be a CM type for $E$ and $E_{\Phi}$ be the reflex field of the CM pair
$(E, \Phi)$. Then $[E_{\Phi} F^s: F^s] = 2^n$ and $[E^s:\Q]=2^n[F^s:\Q]$.
\end{lemma}

\begin{proof} By Lemma \ref{compositumid} we have $E_{\Phi} F^s = E^{\sigma_1} \cdots E^{\sigma_n} F^s=E^s$. Hence, to prove that
$[E_{\Phi} F^s: F^s] = 2^n$, we will show that in the tower of extensions
\begin{align*}
F^s \subseteq E^{\sigma_1} F^s \subseteq E^{\sigma_1} E^{\sigma_2} F^s \subseteq \cdots \subseteq E^{\sigma_1} \cdots E^{\sigma_n} F^s,
\end{align*}
each successive extension
\begin{equation*}
\begin{tikzcd}
E^{\sigma_1} \cdots E^{\sigma_{i-1}} E^{\sigma_i} F^s \arrow[-]{d} \\
E^{\sigma_1} \cdots E^{\sigma_{i-1}} F^s
\end{tikzcd}
\end{equation*}
is quadratic.
First, observe that there is an element $\Delta \in \mathcal{O}_F$ with $\Delta \ll 0$ and $E = F(\sqrt{\Delta})$.
Therefore $E^{\sigma_i} F^{s} = F^s (\sqrt{\sigma_i(\Delta}))$, and hence for each $i=1, \dots, n$ we have
\begin{align*}
E^{\sigma_1} \cdots E^{\sigma_{i}} F^s = F^s(\sqrt{\sigma_1(\Delta)}, \dots, \sqrt{\sigma_i(\Delta)}).
\end{align*}
This implies that
\begin{align*}
[E^{\sigma_1} \cdots E^{\sigma_{i-1}} E^{\sigma_i} F^s :E^{\sigma_1} \cdots E^{\sigma_{i-1}} F^s] \leq 2.
\end{align*}

Now, for each $i = 1, \dots , n$, let $\mathfrak{p}_i$ be a prime ideal of $F^s$ dividing the ideal
$\mathfrak{p}^{\sigma_i} \mathcal{O}_{F^s}$. Then for $i \neq j$, Lemma \ref{coprime} implies that
$\mathfrak{p}_i \neq \mathfrak{p}_j$, and moreover, by Lemma \ref{relative}, the relative discriminant
$\mathfrak{d}(E^{\sigma_i} F^s/F^s)$ is divisible by $\mathfrak{p}_i$, but not by $\mathfrak{p}_j$.
This implies that $\mathfrak{p}_i$ is ramified in $E^{\sigma_i}F^s$, but $\mathfrak{p}_j$ is unramified in $E^{\sigma_i}F^s$.

By the preceding paragraph, for each $i =1, \dots, n,$ the prime ideal $\mathfrak{p}_i$ is unramified in the extensions
$E^{\sigma_1}F^s, \dots, E^{\sigma_{i-1}} F^s$. Now, it is known that if a prime ideal of a number field $M$ is unramified
in the extensions $K/M$ and $L/M$,
then it is unramified in their compositum $KL/M$ (see e.g \cite[Proposition 4.9.2]{Koc00}).
Therefore, it follows that
$\mathfrak{p}_i$ is unramified in the compositum $E^{\sigma_1} \cdots E^{\sigma_{i-1}} F^s$. On the other hand,
since $\mathfrak{p}_i$ is ramified in $E^{\sigma_i} F^s /F^s$,
it is ramified in $E^{\sigma_1} \cdots E^{\sigma_{i-1}} E^{\sigma_i} F^s / F^s$.

Let $\mathfrak{P}$ be a ramified prime ideal of $E^{\sigma_1} \cdots E^{\sigma_{i-1}} E^{\sigma_i} F^s$ lying above $\mathfrak{p}_i$. Then
$\mathfrak{Q}:=\mathfrak{P} \cap \mathcal{O}_{E^{\sigma_1} \cdots E^{\sigma_{i-1}} F^s}$
is an unramified prime ideal of $E^{\sigma_1} \cdots E^{\sigma_{i-1}} F^s$ lying
above $\mathfrak{p}_i$.  In terms of ramification indices, we have $e(\mathfrak{P}|\mathfrak{p}_i) \geq 2$ and $e(\mathfrak{Q}|\mathfrak{p}_i) =1$.
Then by multiplicativity of the ramification index, we have
\begin{align*}
e(\mathfrak{P}|\mathfrak{p}_i)=e(\mathfrak{P}|\mathfrak{Q})e(\mathfrak{Q}|\mathfrak{p}_i)=e(\mathfrak{P}|\mathfrak{Q}).
\end{align*}
Hence
\begin{align*}
[E^{\sigma_1} \cdots E^{\sigma_{i-1}} E^{\sigma_i} F^s :E^{\sigma_1} \cdots E^{\sigma_{i-1}} F^s] \geq e(\mathfrak{P}|\mathfrak{Q})=e(\mathfrak{P}|\mathfrak{p}_i)\geq 2.
\end{align*}
We conclude that
\begin{align*}
[E^{\sigma_1} \cdots E^{\sigma_{i-1}} E^{\sigma_i} F^s :E^{\sigma_1} \cdots E^{\sigma_{i-1}} F^s] = 2.
\end{align*}
This completes the proof that $[E_{\Phi}F^s:F^s]=2^n$.

Finally, since $E^s=E_{\Phi}F^s$ and $[E_{\Phi}F^s:F^s]=2^n$, it follows that
\begin{align*}
[E^s:\Q] = [E^s:F^s][F^s:\Q]=2^n[F^s:\Q].
\end{align*}

\end{proof}

\textbf{Proof of Theorem \ref{Shimura}}. We have the following towers of fields.
\begin{equation*}
\begin{tikzcd}
\, & E_{\Phi}F^s \arrow[-]{dl} \arrow[-]{dr} & \\
E_{\Phi} \arrow[-]{dr} &  &  F^s \arrow[-]{dl}\\
 & \Q &
\end{tikzcd}
\end{equation*}
Therefore,
\begin{align*}
[E_{\Phi} F^s : E_{\Phi}] [E_{\Phi} : \Q] = [E_{\Phi} F^s : F^s] [F^s : \Q],
\end{align*}
hence by Lemma \ref{Degree} we have
\begin{align*}
[E_{\Phi} : \Q] = 2^n \frac{[F^s : \Q]}{[E_{\Phi} F^s : E_{\Phi}]}.
\end{align*}
Now, it is known that if $K/M$ is a finite Galois extension and $L/M$ is an arbitrary extension, then $[KL:L]$ divides $[K:M]$
(see e.g. \cite[Corollary VI.1.13]{Lan02}). Since $F^s/\Q$ is Galois, we have that $[E_{\Phi} F^s : E_{\Phi}]$ divides $[F^s : \Q]$.
This implies that $[E_{\Phi} : \Q] \geq 2^n$. On the other hand, by Proposition \ref{orbit},
we also know that $[E_{\Phi} : \Q] \leq 2^n$, thus we conclude that $[E_{\Phi} : \Q] = 2^n$, as desired.

Finally, since $\Q \subseteq E_{\Phi} \subseteq E^s$, it follows that $[E_{\Phi}:\Q] = 2^n$ divides $[E^s : \Q]$. Then if
$n \geq 3$ we have
$[E^s : \Q] \geq 2^n > 2n = [E : \Q],$
which proves that the field extension $E/\Q$ is non-Galois, therefore non-abelian. \qed


\begin{remark}\label{notweyl} Let $E/F$ be a CM extension as in Lemma \ref{Degree}. Then since $[E^s:\Q]=2^n[F^s:\Q]$, the CM field
$E$ is a Weyl CM field if and only if $[F^s:\Q]=n!$.
\end{remark}

\subsection{Algorithm for constructing CM fields with reflex fields of maximal degree}\label{algorithm}

By combining (the proof of) Proposition \ref{InfiniteCMFields} with the choice $m = d_{F^s}$
and Theorem \ref{Shimura} with the choice $\mathcal{L}=\mathcal{R} \cup \{\frak{q}\}$, we obtain
the following algorithm for constructing infinite families of CM extensions which are non-Galois over $\Q$
and with reflex fields of maximal degree.

 \begin{algorithm}[H]
 \caption{CM fields with reflex fields of maximal degree}
 \label{CMFieldsAlgorithm}
\begin{algorithmic}[1]
\vspace{0.05in}
 \item \textbf{Input}: A tuple $(F, p, \mathfrak{p}, \mathcal{R}, \mathcal{U}_1, \mathcal{U}_2)$ consisting of a totally real number field $F$ of degree $n$,
a rational prime $p \in \Z$ that splits in $F^s$, a prime ideal $\mathfrak{p}$ of $F$ lying above $p$, a finite set $\mathcal{R}$ of
prime ideals of $F$ not dividing $pd_{F^s}$, and finite sets $\mathcal{U}_1$ and $\mathcal{U}_2$ of prime ideals of $F$ not dividing $2 p d_{F^s}$
such that $\mathcal{R}$, $\mathcal{U}_1$ and $\mathcal{U}_2$ are pairwise disjoint.
\vspace{0.05in}

 \item \textbf{Output}: A pair $(\mathfrak{q}, \Delta_{\frak{q}})$ where
$\mathfrak{q}$ is a prime ideal of $F$ not dividing $p d_{F^s}$, and $\Delta_{\frak{q}}$ is an element of $\mathcal{O}_F$ with
prime factorization
\begin{align*}
\Delta_{\frak{q}}\mathcal{O}_F=\frak{p} \frak{q}\prod_{\frak{r} \in \mathcal{R}}\frak{r}.
\end{align*}
The field $E_{\mathfrak{q}}:=F(\sqrt{\Delta_{\frak{q}}})$
is a CM extension of $F$ ramified only at the prime ideals of $F$ dividing $\Delta_{\frak{q}}$ with
reflex fields of maximal degree $2^n$. Moreover, each prime ideal in $\mathcal{U}_1$ splits in $E_{\mathfrak{q}}$ and
each prime ideal in $\mathcal{U}_2$ remains inert in $E_{\mathfrak{q}}$. If $n \geq 3$ then $E_{\frak{q}}/\Q$ is non-Galois.
\vspace{0.05in}
 \item Set $\mathcal{T}_1 := ( \mathcal{U}_1 \cup \{ \mathfrak{P} \subset \mathcal{O}_F \suchthat \text{$\mathfrak{P}$
 divides $p d_{F^s}$}  \}) \smallsetminus \{ \mathfrak{p} \}$.
 \item Set $\mathcal{T}_2:=( \mathcal{U}_2 \cup \{ \mathfrak{P} \subset \mathcal{O}_F \suchthat
\text{$\mathfrak{P}$ divides 2}\}) \smallsetminus (\mathcal{T}_1 \cup \mathcal{R} \cup \{ \mathfrak{p} \})$.
 \item Choose an integer $e \in \Z$ satisfying $e \geq  2 \op{max}\{ v_{\mathfrak{P}}(2) \suchthat \mathfrak{P} \in \mathcal{T}_1 \cup \mathcal{T}_2\} + 1$.
 \item Set $\mathfrak{m}_{\infty}$ to be the formal product of all the embeddings in $\op{Hom}(F, \R)$.
 \item Set $$\mathfrak{m}_0:= \prod \limits_{\mathfrak{P} \in \mathcal{T}_1 \cup \mathcal{T}_2} \mathfrak{P}^{e}$$
and $\mathfrak{m}:= \mathfrak{m}_0 \mathfrak{m}_{\infty}$.
 \item For each $\mathfrak{P} \in \mathcal{T}_2$ find an element $\alpha_{\mathfrak{P}} \in \mathcal{O}_F$ such that
 $F_{\mathfrak{P}}(\sqrt{\alpha_{\mathfrak{P}}})$ is an unramified quadratic extension of $F_{\mathfrak{P}}$.
\item Find an element $a \in F^{\times}$ satisfying the following congruences.
\begin{itemize}
\item[(i)] \textrm{$a \overset{\times}{\equiv} -1 \pmod{\mathfrak{m}_{\infty}}$}.
\item[(ii)] \textrm{$a \overset{\times}{\equiv} 1 \pmod{\mathfrak{P}^e}$ for every $\mathfrak{P} \in \mathcal{T}_1$}.
\item[(iii)] \textrm{$a \overset{\times}{\equiv} \alpha_{\mathfrak{P}} \pmod{\mathfrak{P}^e}$ for every $\mathfrak{P} \in \mathcal{T}_2$}.
\end{itemize}
\item Set $$\mathfrak{n}:= a \mathfrak{p}^{-1}\prod_{\frak{r} \in \mathcal{R}}\frak{r}^{-1}.$$
 \item Choose a prime ideal $\mathfrak{q} \subset \mathcal{O}_F$ lying in the ray class of $\mathfrak{n}$
 modulo $\mathfrak{m}$ such that $\mathfrak{q} \not \in \mathcal{T}_1 \cup \mathcal{T}_2 \cup \mathcal{R} \cup \{ \mathfrak{p} \}$.
 \item Find an element $b_{\mathfrak{q}} \in F^{\times}$ such that $b_{\mathfrak{q}} \overset{\times}{\equiv} 1 \pmod{\mathfrak{m}}$
 and $\mathfrak{q} = b_{\mathfrak{q}} \mathfrak{n}$.
\item Set $\Delta_{\frak{q}}:=ab_{\frak{q}}$.
 \item \textbf{Return}: $(\mathfrak{q}, \Delta_{\frak{q}})$.
 \end{algorithmic}
 \end{algorithm}

\begin{remark}
For steps 5 and 8 in the algorithm, see Lemmas \ref{square} and \ref{QUE}, respectively.
\end{remark}

\begin{remark}
The congruences in steps $9$ and $11$ of the algorithm are chosen to force the given prime ideal
$\mathfrak{P} \in \mathcal{T}_1 \cup \mathcal{T}_2$ to be unramified in the extension $E_{\mathfrak{q}}$. In fact, as was shown in the proof of
Lemma \ref{ablemma}, the congruence 9 (ii) forces $\mathfrak{P}$ to split in $E_{\mathfrak{q}}$,
while the congruence 9 (iii) forces $\mathfrak{P}$ to remain inert in $E_{\mathfrak{q}}$.
\end{remark}

\begin{remark}
Recall from the proof of Proposition \ref{InfiniteCMFields} that $\mathcal{S}_{\mathcal{R}, \frak{p}}:=\mathcal{S}(\mathfrak{n})
\smallsetminus (\mathcal{T}_1 \cup \mathcal{T}_2 \cup \mathcal{R} \cup \{\frak{p}\})$,
where
\begin{align*}
\mathcal{S}(\mathfrak{n}):=\{\frak{q} \subset \mathcal{O}_F \suchthat \textrm{$\frak{q}$ is a prime ideal and
$[\frak{q}]=[\mathfrak{n}]$ in $\mathcal{R}_{F}(\mathfrak{m})$}\}.
\end{align*}
Also, as was shown in the proof of Proposition \ref{InfiniteCMFields}, the set $\mathcal{S}_{\mathcal{R}, \frak{p}}$
has natural density $d(\mathcal{S}_{\mathcal{R}, \frak{p}})=1/\# \mathcal{R}_{F}(\mathfrak{m})$.
\end{remark}

\section{Proof of Theorem \ref{maintheorem}} Let $F$ be a totally real
number field of degree $n \geq 3$. Let $p \in \Z$ be a prime number
which splits in the Galois closure $F^s$ and let $\frak{p}$ be a prime ideal of $F$ lying above $p$.
Let $\mathcal{R}$ be a finite set of prime ideals of $F$ not dividing $pd_{F^s}$. Let $\mathcal{U}_1$ and $\mathcal{U}_2$ be finite sets of prime ideals of $F$ not
dividing $2pd_{F^s}$ such that $\mathcal{R}, \mathcal{U}_1$ and $\mathcal{U}_2$ are pairwise disjoint.
Then by Proposition \ref{InfiniteCMFields} with the choice $m=d_{F^s}$, there is a set $\mathcal{S}_{\mathcal{R}, \frak{p}}$ of prime ideals of $F$ which is disjoint from
$\mathcal{R} \cup \mathcal{U}_1 \cup \mathcal{U}_2 \cup \{\frak{p}\}$ such that
the following statements are true.
\begin{itemize}
\item[(i)] $\mathcal{S}_{\mathcal{R}, \frak{p}}$ has positive natural density.
\item[(ii)] Each prime ideal $\frak{q} \in \mathcal{S}_{\mathcal{R}, \frak{p}}$ is relatively prime to $pd_{F^s}$.
\item[(iii)] For each prime ideal $\frak{q} \in \mathcal{S}_{\mathcal{R}, \frak{p}}$, there is an element
$\Delta_{\frak{q}} \in \mathcal{O}_F$ with prime factorization
\begin{align*}
\Delta_{\frak{q}}\mathcal{O}_F=\frak{p} \frak{q}\prod_{\frak{r} \in \mathcal{R}}\frak{r}.
\end{align*}
\item[(iv)] The field $E_{\frak{q}}:=F(\sqrt{\Delta_{\frak{q}}})$ is a CM extension of $F$
which is ramified only at the prime ideals of $F$ dividing $\Delta_{\frak{q}}$.
Moreover, each prime ideal in $\mathcal{U}_1$ splits in $E_{\frak{q}}$ and each
prime ideal in $\mathcal{U}_2$ is inert in $E_{\frak{q}}$.
\end{itemize}
It follows from Theorem \ref{Shimura} with the choice $\mathcal{L} = \mathcal{R} \cup \{ \mathfrak{q} \}$ that for each prime
ideal $\frak{q} \in \mathcal{S}_{\mathcal{R}, \frak{p}}$, the degree of the reflex
field $E_{\frak{q}, \Phi}$ is $[E_{\frak{q}, \Phi}:\Q]=2^n$
for every CM type $\Phi \in \Phi(E_{\mathfrak{q}})$, and moreover, since $n \geq 3$ then $E_{\frak{q}}/\Q$ is non-Galois.

Now, by Proposition \ref{cc} and Corollary \ref{orbit2}, if $E$ is a CM field and there exists a CM type $\Phi \in \Phi(E)$ such that
the degree of the reflex field $E_{\Phi}$ is $[E_{\Phi}:\Q]=2^n$,
then Conjecture \ref{cc2} is true for $E$.  It then follows from the previous paragraph
that for each prime ideal $\frak{q} \in \mathcal{S}_{\mathcal{R}, \frak{p}}$, Conjecture \ref{cc2} is true for $E_{\frak{q}}$. \qed

\section{Weyl CM fields and the proof of Theorem \ref{weyl}}

Let $E=\Q(\alpha)$ be a CM field of degree $2g$. Let $m_{\alpha}(X)$
be the minimal polynomial of $\alpha$ and denote its roots by $\alpha_1=\alpha, \overline{\alpha_1}, \ldots, \alpha_g, \overline{\alpha_g}$.
Let
\begin{align}\label{permute0}
a_{2\ell -1}:=\alpha_{\ell} \qquad \textrm{and} \qquad  a_{2 \ell}:= \overline{\alpha_{\ell}}
\end{align}
for $\ell = 1, \ldots, g$. Then $E^{s}=\Q(a_1, \ldots, a_{2g})$
is the Galois closure of $E$. Let $S_{2g}$ be the symmetric group on the elements $\{a_1, \ldots, a_{2g}\}$ and
$W_{2g}$ be the subgroup of
$S_{2g}$ consisting of permutations which map any pair of the form $\{a_{2j-1}, a_{2j}\}$ to a pair $\{a_{2k-1}, a_{2k}\}$.
The group $W_{2g}$ is called the \textit{Weyl group}. It can be shown that $\# W_{2g} = 2^g g!$ and that $W_{2g}$
fits in the exact sequence
\begin{align*}
1 \longrightarrow (\Z/2\Z)^{g} \longrightarrow W_{2g} \longrightarrow S_{g} \longrightarrow 1.
\end{align*}

\begin{proposition}\label{weylprop}
The Galois group $\mathrm{Gal}(E^s/\Q)$ is isomorphic to a subgroup of $W_{2g}$.
\end{proposition}

\begin{proof} There is an injective group homomorphism $\phi: \Gal(E^s/\Q) \longrightarrow S_{2g}$
given by restriction $\sigma \longmapsto \sigma|_{\{a_1, \ldots, a_{2g}\}}$. Hence
$\Gal(E^s/\Q) \cong \phi(\Gal(E^s/\Q)) < S_{2g}$, so it suffices to prove that
$\phi(\Gal(E^s/\Q)) \subseteq W_{2g}$, or equivalently, that given $\sigma \in \Gal(E^s/\Q)$ and a pair $\{a_{2j-1}, a_{2j}\}$,
we have $\sigma(\{a_{2j-1}, a_{2j}\})=\{a_{2k-1}, a_{2k}\}.$
Since $\sigma$ permutes the elements $\{a_1, \ldots, a_{2g}\}$, we have $\sigma(a_{2j-1})=a_{2k-1}$ or $\sigma(a_{2j-1})=a_{2k}$ for some $k$.
Now, since $E$ is a CM field, given any $b \in E$, we have $\overline{\sigma(b)}=\sigma(\overline{b})$ for all $\sigma \in \Gal(E^s/\Q)$. Moreover,
from (\ref{permute0}), we have $\overline{a_{2\ell-1}}=a_{2\ell}$ and $\overline{a_{2\ell}}=a_{2\ell-1}$. Combining these facts yields
\begin{align*}
\sigma(a_{2j})=
\sigma(\overline{a_{2j-1}})=\overline{a_{2k-1}}=a_{2k} \quad \textrm{or} \quad \sigma(a_{2j})=\sigma(\overline{a_{2j-1}})=\overline{a_{2k}}=a_{2k-1}.
\end{align*}
This completes the proof.
\end{proof}

\begin{definition}
If $E$ is a CM field such that $\mathrm{Gal}(E^s/\Q)\cong W_{2g}$, then $E$ is called a \textit{Weyl CM field}.
\end{definition}

Observe that if $g \geq 2$ and $E$ is a Weyl CM field of degree $2g$, then $E/\Q$ is non-Galois since
$\# \mathrm{Gal}(E^s/\Q)=2^g g! > 2g=[E:\Q]$. In particular, any Weyl CM field of degree $2g \geq 4$ is non-abelian.
\vspace{0.05in}

\textbf{Proof of Theorem \ref{weyl}}.
Let $E$ be a Weyl CM field. Then by Proposition \ref{cc} and Corollary \ref{orbit2}, it suffices to prove that
there exists a CM type $\Phi$ for $E$ such that the reflex field $E_{\Phi}$
has degree $[E_{\Phi}:\Q]=2^g$. For $i=1, \ldots, g$ let $\tau_i: E \hookrightarrow \C$ be the embedding defined by
$\tau_i(\alpha_1) = \alpha_i$, where $\alpha_1=\alpha$. Then $\textrm{Hom}(E,\C)=\{\tau_1, \overline{\tau_1}, \ldots, \tau_g, \overline{\tau_g}\}$.
Note that $\tau_i(a_1)=a_{2i-1}$ for $i=1, \ldots, g$. Fix the choice of CM type $\Phi=\{\tau_1, \ldots, \tau_g\}$. We will prove
that $[E_{\Phi}:\Q]=2^g$.

Since $E$ is a Weyl CM field, we have $\mathrm{Gal}(E^s/\Q)\cong W_{2g}$, and thus
$\# \mathrm{Gal}(E^s/\Q) = 2^g g!$. Moreover, the calculations in the proof of Lemma \ref{orbit} yield
\begin{align*}
[E_{\Phi}:\Q]=\frac{ \# \Gal(E^{s}/\Q)}{\# \op{Stab}_{\op{Gal}(E^{s}/\Q)}(\Phi)}=\frac{2^gg!}{\# \op{Stab}_{\op{Gal}(E^{s}/\Q)}(\Phi)}.
\end{align*}
Hence it suffices to prove that $\# \op{Stab}_{\op{Gal}(E^{s}/\Q)}(\Phi)=g!$.

Let $S_{2g}^{\textrm{odd}}$ be the symmetric group on the odd-indexed elements $\{a_1, a_3, \ldots, a_{2g-1}\}$.
Then
\begin{align*}
\sigma \in \op{Stab}_{\op{Gal}(E^{s}/\Q)}(\Phi) & \iff  \sigma \Phi = \Phi \\
& \iff \textrm{for all $i$, there exists a $j$ such that $\sigma \tau_i (a_1) = \tau_j(a_1)$}\\
&  \iff \textrm{for all $i$, there exists a $j$ such that $\sigma(a_{2i-1})=a_{2j-1} $} \\
&  \iff \sigma|_{\{a_1, a_3, \ldots, a_{2g-1}\}} \in S_{2g}^{\textrm{odd}}.
\end{align*}
Thus we have a map $\widetilde{\phi}: \op{Stab}_{\op{Gal}(E^{s}/\Q)}(\Phi) \longrightarrow S_{2g}^{\textrm{odd}}$ given by restriction $\sigma \longmapsto
\sigma|_{\{a_1, a_3, \ldots, a_{2g-1}\}}$. We will prove that $\widetilde{\phi}$ is bijective.
\vspace{0.05in}

\textbf{Surjectivity}: Let $\pi \in S_{2g}^{\textrm{odd}}$. Then for all $i$, there exists a $j$ such that $\pi(a_{2i-1})=a_{2j-1}$.
There is a unique lift $\widetilde{\pi}$ of $\pi$ to $W_{2g}$ given by $\widetilde{\pi}(a_{2i-1})=a_{2j-1}$ and $\widetilde{\pi}(a_{2i})=a_{2j}$. Since
$E$ is a Weyl CM field, we have an isomorphism $\op{Gal}(E^{s}/\Q) \cong \phi(\op{Gal}(E^{s}/\Q))=W_{2g}$, where $\phi$ is the restriction map
$\sigma \longmapsto \sigma|_{\{a_1, \ldots, a_{2g}\}}$ in the proof of Proposition \ref{weylprop}. Hence there exists a
unique element $\sigma \in \op{Gal}(E^{s}/\Q)$ such that $\phi(\sigma)=\widetilde{\pi}$. Observe that
\begin{align*}
\sigma|_{\{a_1, a_3, \ldots, a_{2g-1}\}}
= \widetilde{\pi}|_{\{a_1, a_3, \ldots, a_{2g-1}\}} = \pi \in S_{2g}^{\textrm{odd}}.
\end{align*}
It follows that
$\sigma \in \op{Stab}_{\op{Gal}(E^{s}/\Q)}(\Phi)$ with $\widetilde{\phi}(\sigma)=\pi$. This proves that $\widetilde{\phi}$ is surjective.
\vspace{0.05in}

\textbf{Injectivity}: Let $\sigma_1, \sigma_2 \in \op{Stab}_{\op{Gal}(E^{s}/\Q)}(\Phi)$ with $\widetilde{\phi}(\sigma_1)=\widetilde{\phi}(\sigma_2)$.
Then
\begin{align*}
{\sigma_1}|_{\{a_1, a_3, \ldots, a_{2g-1}\}}={\sigma_2}|_{\{a_1, a_3, \ldots, a_{2g-1}\}},
\end{align*}
i.e., $\sigma_1(a_{2i-1})=\sigma_2(a_{2i-1})$ for $i=1, \ldots, g$.
On the other hand, arguing as in the proof of Proposition \ref{weylprop}, we have
\begin{align*}
\sigma_1(a_{2i})=\sigma_1(\overline{a_{2i-1}})=\overline{\sigma_1(a_{2i-1})} = \overline{\sigma_2(a_{2i-1})} = \sigma_2(\overline{a_{2i-1}}) =\sigma_2(a_{2i})
\end{align*}
for $i=1, \ldots, g$. Thus, $\sigma_1=\sigma_2$. This proves that $\widetilde{\phi}$ is injective.
\vspace{0.05in}

Since $\widetilde{\phi}$ is bijective, we have $\# \op{Stab}_{\op{Gal}(E^{s}/\Q)}(\Phi) = \# S_{2g}^{\textrm{odd}} = g!$. This completes the proof.
\qed

\section{Proof of Theorem \ref{rhodensity}}

As mentioned in Remark \ref{weylremark}, if $E$ is a quartic CM field then the only possible Galois groups
of its Galois closure are $C_4: = \Z/4\Z$, $V_4:= \Z/2\Z \times \Z/2\Z$ and $D_4$.
Moreover, the Weyl group $W_4 \cong D_4$. It is known that if a quartic number field $K$ has Galois group $\Gal(K^s/\Q) \cong C_4$
or $V_4$, then $K$ contains a unique real quadratic subfield. Hence for the signature $(0, 2)$ (which corresponds to the quartic totally complex case) we have
\begin{align}\label{CMC4V4}
\op{CM}_4(C_4, X) = N_{0, 2}(C_4, X) \quad \text{and} \quad \op{CM}_4(V_4, X) = N_{0, 2}(V_4, X).
\end{align}
When $\Gal(K^s/\Q) \cong D_4$, then $K$ can either contain a real quadratic subfield or an imaginary quadratic subfield.
Thus in that case one can define a refined counting function $N_{0, 2}^{+}(D_4, X)$ which counts only isomorphism classes of quartic number fields
$K/\Q$ with signature $(0, 2)$ containing a real quadratic subfield, and such that $\Gal(K^s/\Q) \cong D_4$ and $|d_K| \leq X$ (see \cite{CDO02}).
With this notation we then have
\begin{align}\label{CMD4}
\op{CM}_4(W_4, X) = \op{CM}_4(D_4, X) = N_{0, 2}^{+}(D_4, X).
\end{align}


 
By Cohen, Diaz y Diaz and Olivier \cite[Corollary 4.5 (2), p. 501]{CDO05} (which refines earlier work of Baily \cite{Bai80} and M\"aki \cite{Mak85}) 
we have
\begin{align}\label{AsympoticC4}
N_{0, 2}(C_4, X) = c(C_4)X^{\frac{1}{2}} + O(X^{\frac{1}{3} + \varepsilon}),
\end{align}
for some explicit positive constant $c(C_4)$ and any $\varepsilon > 0$. Similarly, in \cite[p. 582]{CDO06} we find the asymptotic formula
\begin{align}\label{AsymptoticV4}
N_{0, 2}(V_4, X) = c(V_4)X^{\frac{1}{2}} \log^{2}{X} + O(X^{\frac{1}{2}} \log{X}),
\end{align}
for some explicit positive constant $c(V_4)$. Finally, in \cite[Proposition 6.2, p. 88]{CDO02} we find the asymptotic formula
\begin{align}\label{AsymptoticD4}
N_{0, 2}^{+}(D_4, X) = c(D_4)^{+} X + O(X^{\frac{3}{4} + \varepsilon}),
\end{align}
where again $c(D_4)^{+}$ is an explicit positive constant. 

Since
\begin{align*}
\op{CM}_4(X) = \op{CM}_4(D_4, X) + \op{CM}_4(C_4, X) + \op{CM}_4(V_4, X),
\end{align*}
it follows from equations (\ref{CMC4V4})--(\ref{AsymptoticD4}) that
\begin{align*}
\op{CM}_4(X) =  c(D_4)^{+}X + O(X^{\frac{3}{4} + \varepsilon}).
\end{align*}
Because this is the same asymptotic formula satisfied by the counting function $\op{CM}_4(W_4, X)$,
we conclude that the density of quartic Weyl CM fields is
\begin{align*}
\rho_{\text{Weyl}}(4)=\lim_{X \to \infty} \frac{\op{CM}_{4}(W_{4}, X)}{\op{CM}_{4}(X)} = 1.
\end{align*}
\qed

\section{Abelian varieties over finite fields, Weil $q$-numbers, and density results}

We first review some facts concerning Weil $q$-numbers and abelian varieties over finite fields.

\subsection{Weil $q$-numbers and abelian varieties over $\mathbb{F}_q$}
Let $q=p^n$ where $p$ is a prime number and $n$ is a positive integer. A \textit{Weil $q$-number} is an algebraic integer $\pi$ such that for
every embedding $\sigma : \Q(\pi) \hookrightarrow \C$ we have $|\sigma(\pi)|=q^{1/2}$.  Let
$W(q)$ denote the set of Weil $q$-numbers. Two Weil $q$-numbers $\pi_1$ and $\pi_2$
are \textit{conjugate} if there exists an isomorphism $\Q(\pi_1) \rightarrow \Q(\pi_2)$ which maps $\pi_1$ to $\pi_2$. In this case, we write
$\pi_1 \sim \pi_2$.

We have the following facts about Weil $q$-numbers (see e.g. \cite[p. 1 and Corollary 2.1]{GO88}).

\begin{lemma}\label{q-number}
Let $q=p^n$ and $\pi \in W(q)$.
\begin{itemize}

\item[(i)] If $\sigma(\pi) \in \R$ for some embedding
$\sigma: \Q(\pi) \hookrightarrow \C$, then $\Q(\pi)=\Q$ if $n$ is even and $\Q(\pi)=\Q(\sqrt{p})$ if $n$ is odd.
\vspace{0.05in}
\item[(ii)] If $\sigma(\pi) \in \C \smallsetminus \R$ for all embeddings
$\sigma: \Q(\pi) \hookrightarrow \C$, then
$\Q(\pi)$ is a CM field with maximal totally real subfield $\Q(\pi + q/\pi)$.
\end{itemize}
\end{lemma}

Let $\mathbb{F}_q$ be a finite field of characteristic $p$ with $q=p^n$ elements. Let $A/\mathbb{F}_q$ be an abelian
variety of dimension $g$ defined over $\mathbb{F}_q$ and let $\pi_A \in \textrm{End}(A)$ be the Frobenius endomorphism of $A$. Let
$f_A$ be the characteristic polynomial of $A$, which is a monic polynomial of degree $2g$ with coefficients in $\Z$.
Let $\Q[\pi_A]$ be the $\Q$-subalgebra of $\textrm{End}(A) \otimes_{\Z} \Q$ generated by $\pi_A$. It is known that
$\Q[\pi_A]$ is a field if and only if $A/\mathbb{F}_q$ is simple.

Weil proved that the roots of $f_A$ are Weil $q$-numbers. Moreover, if $A/\mathbb{F}_q$ is simple,
then the image of $\pi_A$ under any homomorphism $\phi: \Q[\pi_A] \rightarrow \C$ is a Weil $q$-number. Any such homomorphism
$\phi$ maps $\pi_A$ to a root $\alpha_A$ of $f_A$. From here forward, we identify $\pi_A$ with $\phi(\pi_A)$ for some choice of $\phi$. This choice
does not matter, since we will only consider Weil $q$-numbers up to conjugacy.

If $A/\mathbb{F}_q$ and $B/\mathbb{F}_q$ are isogenous simple abelian varieties, then $f_A=f_B$. In particular,
$\pi_A \sim \pi_{B}$. This gives a well-defined map  $A \mapsto \pi_A$
between the set of isogeny classes of simple abelian varieties $A/\mathbb{F}_q$ and Weil $q$-numbers up to conjugacy. A celebrated theorem of
Honda and Tate (see e.g. \cite{Tat71}) asserts that this map is a bijection.

\subsection{Density results and the proof of Theorem \ref{density}}
Let $A/\mathbb{F}_q$ be an abelian variety and $\alpha_A$ be a root of $f_A$.
Let $K_{f_A}=\Q(\alpha_A)^s$ be the splitting field of $f_A$ and $G_{f_A}=\textrm{Gal}(K_{f_A}/\Q)$. Define the sets
\begin{align*}
\mathcal{A}_g(q)&:=\{\textrm{isogeny classes of abelian varieties $A/\mathbb{F}_{q}$ with $\textrm{dim}(A)=g$}\},\\
\mathcal{B}_g(q)&:=\{\textrm{isogeny classes of abelian varieties $A/\mathbb{F}_{q}$ with $\textrm{dim}(A)=g$ and  $G_{f_A} \cong W_{2g}$}\}.\\
\end{align*}

Kowalski \cite[Proposition 8]{Kow06} proved the following density result.

\begin{theorem}\label{Kdensity} With notation as above, one has
\begin{align*}
\lim_{n \rightarrow \infty} \frac{\# \mathcal{B}_g(p^n)}{\# \mathcal{A}_g(p^n)} = 1.
\end{align*}
\end{theorem}

On the other hand, we have the following result.

\begin{proposition}\label{cmcharacterization}
Let $A/\mathbb{F}_q$ be an abelian variety of dimension $g \geq 2$. If $G_{f_A} \cong W_{2g}$,
then $\Q(\alpha_A)$ is a non-Galois Weyl CM field of degree $2g$.
\end{proposition}

\begin{proof}
Let $m := [\Q(\alpha_A):\Q]$. Since $\alpha_A$ is a root of $f_A$ and $f_A \in \Z[x]$ is monic of degree $2g$,
then $\alpha_A$ is an algebraic integer of degree $m$ where $m|2g$.
Suppose by contradiction that $m < 2g$. Since $m$ is a proper divisor of $2g$, we have $m \leq g$.
Hence the Galois closure $K_{f_A}=\Q(\alpha_A)^s$ has degree $[K_{f_A}:\Q] \leq m! \leq g!$. However, $\# W_{2g} = 2^{g} g!$,
which contradicts the assumption that $G_{f_A} \cong W_{2g}$. Thus $m=2g \geq 4$, hence it follows from Lemma \ref{q-number} that
$\Q(\alpha_A)$ is a CM field. Finally, since $2g \geq 4$, we conclude that $\Q(\alpha_A)$ is a
non-Galois Weyl CM field of degree $2g$.
\end{proof}
\vspace{0.10in}

\textbf{Proof of Theorem \ref{density}}. By Theorem \ref{Kdensity} and Proposition \ref{cmcharacterization}, if $g \geq 2$
then the proportion of isogeny classes $[A] \in A_g(p^n)$ for which $\Q(\alpha_A)$ is a non-Galois Weyl CM field approaches 1 as $n \rightarrow
\infty$. Theorem \ref{density} now follows from Theorem \ref{weyl}. \qed
\vspace{0.10in}

We next show that any CM field $E$ is isomorphic to a CM field of the form $\Q(\pi_A)$ for a simple abelian
variety $A/\mathbb{F}_q$.

The following result is a consequence of \cite[Theorems 1 and 2 (i)]{GO88}.

\begin{theorem}\label{gotheorem} Let $E$ be a CM field. Then for each integer $n \geq 2$,
there exists a prime number $p=p(E,n)$ such that $E=\Q(\pi_p)$ for some Weil $p^n$-number $\pi_p \in W(p^n)$.
\end{theorem}

Greaves and Odoni used the Chebotarev density theorem to deduce the following corollary.

\begin{corollary}\label{gocorollary}
There exists an integer $a(E,n) \geq 1$ such that
\begin{align*}
& \# \{p=p(E,n):~ 2 \leq p \leq X, ~ E=\Q(\pi_p), ~ \pi_p \in  W(p^n) \} = \\
& \qquad \qquad \frac{a(E,n)}{[H(E^s):\Q]}\mathrm{Li}(X) + O_{E,n}\left(X\exp\left(-c(E,n)\sqrt{\log(X)}\right)\right)
\end{align*}
as $X \rightarrow \infty$, where $H(E^s)$ denotes the Hilbert class field of
$E^s$, $\mathrm{Li}(X):=\int_{2}^{X}dt/\log(t)$, and $c(E,n) > 0$.
\end{corollary}

We have the following corollary.

\begin{corollary}\label{keycor}
Let $E$ be a CM field. Then for each integer $n \geq 2$, there is a set of prime numbers $p=p(E,n)$
with positive natural density such that $E \cong \Q(\pi_A)$ for some simple abelian variety $A/\mathbb{F}_{p^n}$.
\end{corollary}

\begin{proof} Let $E$ be a CM field. Then by Corollary \ref{gocorollary}, for each integer $n \geq 2$
there is a set of prime numbers $p=p(E,n)$ with positive natural density
such that $E=\Q(\pi_p)$ for some $\pi_p \in W(p^n)$. On the other hand, by the Honda-Tate theorem,
for each such prime number $p$, there exists a simple abelian variety $A/\mathbb{F}_{p^n}$ such that $\pi_A \sim \pi_p$. Therefore,
$\Q(\pi_A) \cong \Q(\pi_p)=E$.
\end{proof}

Given Corollary \ref{keycor}, it is natural to ask whether a density result analogous to Theorem \ref{Kdensity}
holds for simple abelian varieties. Define the sets
\begin{align*}
\mathcal{A}_g^s(q)&:=\{\textrm{isogeny classes of simple abelian varieties $A/\mathbb{F}_{q}$ with $\textrm{dim}(A)=g$}\},\\
\mathcal{B}_g^s(q)&:=\{\textrm{isogeny classes of simple abelian varieties $A/\mathbb{F}_{q}$ with $\textrm{dim}(A)=g$ and $G_{f_A} \cong W_{2g}$}\}.
\end{align*}
It seems likely that a modification of the methods in \cite{Kow06} can be used to prove that
\begin{align*}
\lim_{n \rightarrow \infty}\frac{\# \mathcal{B}_g^{s}(p^n)}{\# \mathcal{A}_g^{s}(p^n)} = 1.
\end{align*}
If true, then arguing as in the proof of Theorem \ref{density}, it would follow that if $g \geq 2$,
then the proportion of isogeny classes  $[A] \in \mathcal{A}_g^s(p^n)$
for which $\Q(\pi_A)$ is a CM field that satisfies the Colmez conjecture approaches 1 as $n \rightarrow \infty$.

\section{Acknowledgments}

The authors were partially supported by the NSF grants DMS-1162535 and DMS-1460766 during the preparation of this work. A. B-S. was also
partially supported by the University of Costa Rica.


\begin{thebibliography}{}



\bibitem[AGHM15]{AGHM15} F. Andreatta, E. Z. Goren, B. Howard and K. Madapusi Pera,
{\em Faltings heights of abelian varieties with complex multiplication.\/} {\tt arXiv:1508.00178 [math.NT]}.


\bibitem[Bai80]{Bai80} A. M. Baily, {\em On the density of discriminants of quartic fields.\/} J. Reine Angew. Math. \textbf{315} (1980), 190--210.

\bibitem[Bar01]{Bar01} E. W. Barnes, {\em The theory of the double Gamma function\/}. Phil. Trans. R. Soc. Lond. A  \textbf{196} (1901), 265--387.

\bibitem[BS-M16a]{BS-M16a} A. Barquero-Sanchez and R. Masri, {\em Faltings heights of CM abelian surfaces and the non-abelian Chowla-Selberg formula\/}, preprint.

\bibitem[BS-M16b]{BS-M16b} A. Barquero-Sanchez and R. Masri, {\em Faltings heights of CM abelian threefolds and the non-abelian Chowla-Selberg formula\/}, 
in preparation.

\bibitem[Bha07]{Bha07} M. Bhargava, {\em Mass formulae for extensions of local fields, and conjectures on the density of number field discriminants.\/} 
Int. Math. Res. Not. IMRN 2007, no. 17, Art. ID rnm052, 20 pp.

\bibitem[BS15]{BS15} F. Bouyer and M. Streng, {\em Examples of CM curves of genus two defined over the reflex field.\/} 
LMS J. Comput. Math. \textbf{18} (2015), no. 1, 507--538.

\bibitem[CO12]{CO12}  C. Chai and F. Oort, {\em Abelian varieties isogenous to a Jacobian.\/} Ann. of Math. (2) \textbf{176} (2012), no. 1, 589--635.

\bibitem[Chi09]{Chi09} N. Childress, {\em Class field theory.\/} Universitext. Springer, New York, 2009.



 \bibitem[CS67]{CS67} S. Chowla and A. Selberg, {\em On Epstein's zeta-function.\/}
J. Reine Angew. Math. \textbf{227} (1967), 86--110.


\bibitem[Coh03]{Coh03}H. Cohen, {\em Enumerating quartic dihedral extensions of $\Q$ with signatures.\/} 
Ann. Inst. Fourier (Grenoble) \textbf{53} (2003), no. 2, 339--377.

\bibitem[CDO02]{CDO02}H. Cohen, F. Diaz y Diaz and M. Olivier, 
{\em Enumerating quartic dihedral extensions of $\Q$.\/} Compositio Math. \textbf{133} (2002), no. 1, 65--93.


\bibitem[CDO05]{CDO05} H. Cohen, F. Diaz y Diaz and M. Olivier, {\em Counting cyclic quartic extensions of a number field.\/}  
J. Th\'eor. Nombres Bordeaux \textbf{17} (2005), no. 2, 475--510.

\bibitem[CDO06]{CDO06} H. Cohen, F. Diaz y Diaz and M. Olivier, {\em Counting discriminants of number fields.\/} 
J. Th\'eor. Nombres Bordeaux \textbf{18} (2006), no. 3, 573--593.



\bibitem[Col93]{Col93} P. Colmez,
{\em P\'eriodes des vari\'et\'es ab\'eliennes \`a multiplication complexe\/}.
Ann. of Math. \textbf{138} (1993), 625--683.

\bibitem[Col98]{Col98} P. Colmez, {\em Sur la hauteur de Faltings des vari\'et\'es ab\'eliennes \`a multiplication complexe. (French)
[On the Faltings height of abelian varieties with complex multiplication]\/} Compositio Math. \textbf{111} (1998), no. 3, 359--368.


\bibitem[FT93]{FT93} A. Fr\"ohlich and M. J. Taylor, {\em Algebraic number theory.\/}
Cambridge Studies in Advanced Mathematics, \textbf{27}. Cambridge University Press, Cambridge, 1993.

\bibitem[Gal73]{Gal73} P. X. Gallagher, {\em The large sieve and probabilistic Galois theory.\/}
Analytic number theory (Proc. Sympos. Pure Math., Vol. XXIV, St. Louis Univ., St. Louis, Mo., 1972),
pp. 91--101. Amer. Math. Soc., Providence, R.I., 1973.

\bibitem[GO88]{GO88} A. Greaves and R. W. K. Odoni, {\em Weil-numbers and CM-fields. I. \/} J. Reine Angew. Math. 391 (1988), 198--212.

\bibitem[Gro80]{Gro80} B. H. Gross, {\em Arithmetic on elliptic curves with complex multiplication. With an appendix by B. Mazur.\/}
Lecture Notes in Mathematics, \textbf{776}. Springer, Berlin, 1980.

\bibitem[Jan96]{Jan96} G. J. Janusz, {\em Algebraic number fields.\/} Second edition.
Graduate Studies in Mathematics, \textbf{7}. American Mathematical Society, Providence, RI, 1996.

\bibitem[KKS11]{KKS11}  K. Kato,  N. Kurokawa, T. Saito, {\em Number theory. 2. Introduction to class field theory.\/}
Translated from the 1998 Japanese original by Masato Kuwata and Katsumi Nomizu. Translations of Mathematical Monographs, \textbf{240}.
Iwanami Series in Modern Mathematics. American Mathematical Society, Providence, RI, 2011.

\bibitem[Koc00]{Koc00} H. Koch, {\em Number theory. Algebraic numbers and functions.\/}
Translated from the 1997 German original by David Kramer. Graduate Studies in Mathematics, \textbf{24}. American Mathematical Society, Providence, RI, 2000.

\bibitem[Kow06]{Kow06} E. Kowalski, {\em Weil numbers generated by other Weil numbers and torsion fields of abelian varieties.\/}
J. London Math. Soc. (2) \textbf{74} (2006), no. 2, 273--288.


\bibitem[Lan02]{Lan02} S. Lang, {\em Algebra. \/} Revised third edition. Graduate Texts in Mathematics, \textbf{211}. Springer-Verlag, New York, 2002.



\bibitem[Ler87]{Ler87} M. Lerch, {\em Sur quelques formules relatives au nombre des classes.\/} Bull. d. sci. math. (2) \textbf{21} (1897), 302--303.

\bibitem[M{\"a}k85]{Mak85} S. M\"aki, {\em On the density of abelian number fields\/}. Thesis, Helsinki, 1985. 

\bibitem[Mal02]{Mal02} G. Malle, {\em On the distribution of Galois groups.\/} J. Number Theory \textbf{92} (2002), 315--329. 

\bibitem[Mal04]{Mal04} G. Malle, {\em On the distribution of Galois groups. II.\/} Experiment. Math. \textbf{13} (2004), 129--135. 

\bibitem[Mil06]{Mil06} J. S. Milne, {\em Complex Multiplication\/}. {\tt http://www.jmilne.org/math/CourseNotes/CM.pdf}


\bibitem[Mor96]{Mor96} P. Morandi, {\em Field and Galois theory.\/} Graduate Texts in Mathematics, \textbf{167}. Springer-Verlag, New York, 1996.

\bibitem[Neu99]{Neu99} J. Neukirch, {\em Algebraic number theory.\/} Translated from the 1992 German original and with a note by Norbert Schappacher.
With a foreword by G. Harder.
Grundlehren der Mathematischen Wissenschaften [Fundamental Principles of Mathematical Sciences], \textbf{322}. Springer-Verlag, Berlin, 1999.

\bibitem[Obu13]{Obu13} A. Obus,
{\em On Colmez's product formula for periods of CM-abelian varieties\/}.
Math. Ann. \textbf{356} (2013), 401--418.

\bibitem[Oor12]{Oor12} F. Oort, {\em CM Jacobians\/}. Notes from a talk at the conference ``Galois covers and deformations'', Bordeaux, June 25--29, 2012.
{ \tt http://www.staff.science.uu.nl/$\sim$oort0109/Bord2-VI-12.pdf}


\bibitem[Rib01]{Rib01} P. Ribenboim, {\em Classical theory of algebraic numbers.\/} Universitext. Springer-Verlag, New York, 2001.

\bibitem[Shi67]{Shi67}  G. Shimura, {\em Construction of class fields and zeta functions of algebraic curves.\/} Ann. of Math. (2) \textbf{85} (1967), 58--159.

\bibitem[Shi70]{Shi70} G. Shimura, {\em On canonical models of arithmetic quotients of bounded symmetric domains.\/} Ann. of Math. (2) \textbf{91} (1970), 144--222.

\bibitem[Shi94]{Shi94}  G. Shimura, {\em Introduction to the arithmetic theory of automorphic functions.\/} Reprint of the 1971 original.
Publications of the Mathematical Society of Japan, \textbf{11}. Kan$\hat{\textrm{o}}$ Memorial Lectures, 1. Princeton University Press, Princeton, NJ, 1994.

\bibitem[Shi98]{Shi98} G. Shimura, {\em Abelian varieties with complex multiplication and modular functions.\/} Princeton Mathematical Series, \textbf{46}.
Princeton University Press, Princeton, NJ, 1998.

\bibitem[Shi77]{Shi77}, T. Shintani, {\em On a Kronecker limit formula for real quadratic fields.\/} 
J. Fac. Sci. Univ. Tokyo Sect. IA Math. \textbf{24} (1977), no. 1, 167--199.

\bibitem[Sil86]{Sil86} J. H. Silverman, {\em Heights and elliptic curves.\/} Arithmetic geometry (Storrs, Conn., 1984), 253--265, Springer, New York, 1986.

\bibitem[Tat71]{Tat71} J. Tate, {\em Classes d'isog\'{e}nie des vari\'{e}t\'{e}s ab\'{e}liennes sur un corps fini (d'apr\`{e}s T. Honda). (French)
[Isogeny classes of abelian varieties over finite fields (after T. Honda)]\/} S\'{e}minaire Bourbaki. Vol. 1968/69:
Expos\'{e}s 347--363, Exp. No. 352, 95--110, Lecture Notes in Math., \textbf{175}, Springer, Berlin, 1971.

\bibitem[Tsi15]{Tsi15} J. Tsimerman, {\em A proof of the Andr\'{e}-Oort conjecture for $\mathcal{A}_g$\/}. {\tt arXiv:1506.01466 [math.NT]}.

\bibitem[Wei76]{Wei76} A. Weil, {\em Elliptic functions according to Eisenstein and Kronecker.\/}
Ergebnisse der Mathematik und ihrer Grenzgebiete, Band 88. Springer-Verlag, Berlin-New York, 1976.

\bibitem[Wei98]{Wei98} E. Weiss, {\em Algebraic number theory.\/} Reprint of the 1963 original. Dover Publications, Inc., Mineola, NY, 1998.

\bibitem[Win89]{Win89} K. Wingberg, {\em Representations of locally profinite groups.\/}
Representation theory and number theory in connection with the local Langlands conjecture (Augsburg, 1985), 117--125,
Contemp. Math., \textbf{86}, Amer. Math. Soc., Providence, RI, 1989.

\bibitem[Woo16]{Woo16} M. Wood, {\em Asymptotics for number fields and class groups\/}. Women in Numbers 3: Research Directions in Number Theory, to appear. 

\bibitem[Yan10a]{Yan10a} T. H. Yang, {\em An arithmetic intersection formula on Hilbert modular surfaces.\/} Amer. J. Math. \textbf{132} (2010), 1275--1309.

\bibitem[Yan10b]{Yan10b} T. H. Yang, {\em The Chowla-Selberg formula and the Colmez conjecture.\/} Canad. J. Math. \textbf{62} (2010), 456--472.

\bibitem[Yan13]{Ya13} T. H. Yang, {\em Arithmetic intersection on a Hilbert modular surface and the Faltings height.\/}
Asian J. Math. \textbf{17} (2013), 335--381.

\bibitem[YZ15]{YZ15} S. Yuan and S. Zhang, {\em On the Averaged Colmez Conjecture.\/} {\tt arXiv:1507.06903 [math.NT]}.




\end{thebibliography}
\end{document}